\documentclass[10pt]{amsart}
\usepackage{amssymb,amsmath}

\usepackage{graphicx}

\begin{document}

\newtheorem{theorem}{Theorem}[section]
\newtheorem{cor}[theorem]{Corollary}
\newtheorem{pro}[theorem]{Property}
\newtheorem{rmk}[theorem]{Remark}
\newtheorem{lemma}[theorem]{Lemma}
\newtheorem{gen}{Generalization}
\newtheorem{prop}[theorem]{Proposition}
\newtheorem{claim}[theorem]{Claim}
\newtheorem{observation}[theorem]{Observation}
\newtheorem{notation}[theorem]{Notation}
\newtheorem{conjecture}[theorem]{Conjecture}
\newtheorem{defin}[theorem]{Definition}
\newtheorem{defins}[theorem]{Definitions}

\newcommand{\Sssp}{\mbox{$\Sigma_{t_+}$}}
\newcommand{\Sssm}{\mbox{$\Sigma_{t_-}$}}
\newcommand{\Ggg}{\mbox{$\Gamma$}}
\newcommand{\map}{\mbox{$\rightarrow$}}
\newcommand{\bbb}{\mbox{$\beta$}}
\newcommand{\la}{\mbox{$\lambda$}}
\newcommand{\aaa}{\mbox{$\alpha$}}
\newcommand{\eee}{\mbox{$\epsilon$}}
\newcommand{\Rrr}{\mbox{$\mathcal{R}$}}
\newcommand{\lpd}{\mbox{$L^{\mathcal{V}(P,D^*)}$}}
\newcommand{\fpd}{\mbox{$\mathcal{V}(P,D^*)$}}
\newcommand{\bdd}{\mbox{$\partial$}}
\newcommand{\Sss}{\mbox{$\Sigma$}}
\newcommand{\Li}{\mbox{$L_+^{in}$}}
\newcommand{\Lo}{\mbox{$L_+^{out}$}}

\title{Width is not additive}
\author{Ryan Blair}
\author{Maggy Tomova}
\thanks{Research partially supported by an NSF grant.}

\begin{abstract}
We develop the construction suggested by Scharlemann and Thompson in \cite{SchTh} to obtain an infinite family of pairs of knots $K_{\alpha}$ and $K'_{\alpha}$ so that $w(K_{\alpha} \# K'_{\alpha})=max\{w(K_{\alpha}), w(K'_{\alpha})\}$. This is the first known example of a pair of knots such that $w(K\#K')<w(K)+w(K')-2$ and it establishes that the lower bound $w(K\#K')\geq max\{w(K),w(K')\}$ obtained in \cite{SchSch} is best possible. Furthermore, the knots $K_{\alpha}$ provide an example of knots where the number of critical points for the knot in thin position is greater than the number of critical points for the knot in bridge position.
\end{abstract}
\maketitle

\section{Introduction}

Thin position for knots was first defined by Gabai in his proof of property $R$, \cite{Ga}. The idea of width has had important applications in 3-manifold topology. In particular, width played an integral role in three celebrated results: the solution to the knot complement problem \cite{GL}, the recognition problem for $S^3$ \cite{AT}, and the leveling of unknotting tunnels \cite{GST}. However, surprisingly little is known about its intrinsic properties. Most strikingly, the behavior of knot width under connected sums has remained one of the most interesting and difficult problems to elucidate. In an attempt to shed light on this question, width has been compared to bridge number -- the least number of maxima over all projections of the knot. Just like bridge number, width depends on the number of critical points of a projection, but it also takes into account their relative heights. The behavior of bridge number under connected sum was first established by Schubert, \cite{Sch}. Later, Schultens gave a considerably more elegant proof of the result, \cite{Schu2}.  Stacking the two knots vertically and connecting a minimum of the top one to a maximum of the bottom one shows that \begin{equation}\label{eq:bridgeineq}b(K\#K')\leq b(K)+b(K')-1.\end{equation} Schubert's result affirms that inequality \ref{eq:bridgeineq} is in fact an equality.

This construction also gives an easy inequality for the width of a connected sum, namely \begin{equation}\label{eq:widthadditiveineq}w(K\#K')\leq w(K)+w(K')-2.\end{equation} Based on Schubert's result, it was conjectured that inequality \ref{eq:widthadditiveineq} is also an equality. However, proving this remained an open problem. Partial results and special cases have been solved. Most notably, Scharlemann and Schultens showed in \cite{SchSch} that \begin{equation}\label{eq:schschineq}w(K\#K')\geq max\{w(K),w(K')\}\end{equation} and Rieck and Sedgwick showed in \cite{RS} that the equality $w(K\#K')= w(K)+w(K')-2$ holds for meridionally small knots. The main result in this paper is that inequality \ref{eq:widthadditiveineq} is strict for some knots and that inequality \ref{eq:schschineq} is best possible if no restrictions are placed on the knots.

\begin{theorem}\label{thm:counterwidth}
There exists an infinite family of knots $K_{\alpha}$ and $K'_{\alpha}$ so that $w(K_{\alpha} \# K'_{\alpha})=max\{w(K_{\alpha}), w(K'_{\alpha})\}$.
\end{theorem}

Moreover, the construction also yields examples of another interesting phenomenon.

\begin{theorem}\label{thm:counterbridge}
There exists an infinite family of knots $K_{\alpha}$ so that the minimal bridge position for $K_{\alpha}$ has fewer critical points than the thin position of $K_{\alpha}$.
\end{theorem}

Our construction is based on ideas proposed by Scharlemann and Thompson in \cite{SchTh}. In their paper, the authors give a large family of pairs of knots $K_\alpha$ and $K'_\alpha$ for which it appears that $w(K_\alpha\#K'_\alpha)= max\{w(K_\alpha),w(K'_\alpha)\}$. This equality relies on the assumption that the projections of $K_\alpha$ and $K'_\alpha$ considered by Scharlemann and Thompson have minimal width amongst all possible projections of the knots, i.e., that $K_\alpha$ and $K'_\alpha$ are in thin position. One of the knots in each pair is quite simple and its width is easily established. However, the authors could not verify the width of the second knot in any of their pairs. In \cite{BT}, the current authors establish that in fact for most of these pairs the second knot is not in thin position. Therefore, most of these pairs do not provide the desired counterexample to the conjectured equality $w(K\#K')= w(K)+w(K')-2$.

In this paper, we construct a family of pairs of knots $K_\alpha$ and $K'_\alpha$ that satisfy the properties required for the pairs presented in \cite{SchTh} and we establish that both $K_\alpha$ and $K'_\alpha$ are in thin position. For such a pair, it follows that $w(K_\alpha\#K'_\alpha)= max\{w(K_\alpha),w(K'_\alpha)\}$. Figure \ref{fig:counter} depicts such a pair where one knot is the trefoil. The figure demonstrates a projection of $K\#trefoil$ that has the same width as the given projection of $K$. We will show that this projection of $K$ is of minimal width.

\begin{figure}
\begin{center} \includegraphics[scale=.3]{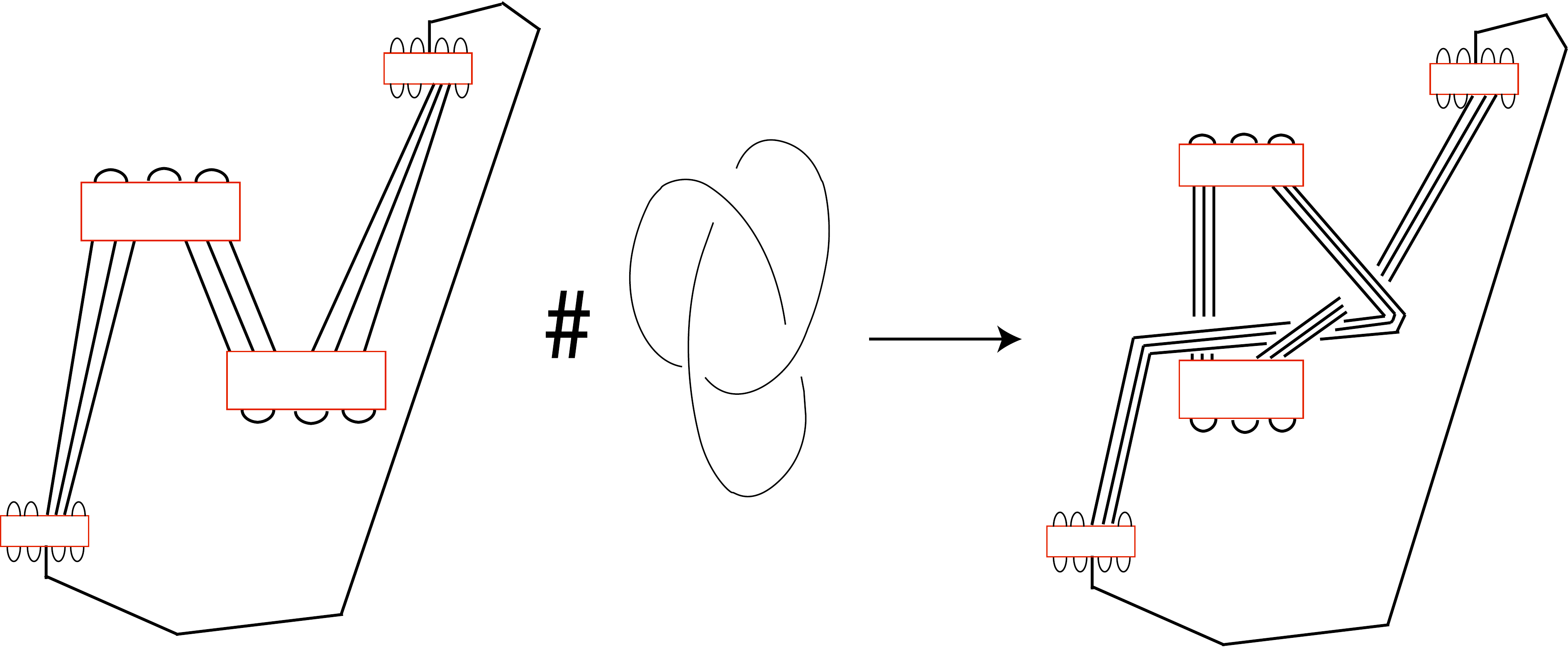}
\end{center}
\caption{} \label{fig:counter} \end{figure}

The paper is organized as follows. After some preliminary definitions in Section \ref{Vertical cut-disks}, we show that if a knot is in thin position and $P$ is a thin level sphere, then cut-disks for $P$ that do not intersect any other thin spheres can be isotoped to be vertical. This allows us to use the results about vertical cut-disks developed in \cite{T2}. In Section \ref{sec:bridgesurfaces}, we review results found in \cite{JT} about the behavior of a second bridge surface for a tangle that has a high distance bridge sphere. In Section \ref{sec:knotting}, we use a theorem of Schubert to construct tangles with high distance properties so that the numerator closures of certain subtangles yield non-trivial knots. In Section \ref{sec:esssurface}, we construct a three strand tangle for which we can classify all essential meridional surfaces of Euler characteristic greater than $-12$. In Section \ref{sec:construction}, we construct the candidate knots $K_\alpha$. Each of these knots has the general schematic introduced in \cite{SchTh}, see Figure \ref{fig:counter}. In Section \ref{sec:properties}, we establish some of the properties of the knots. The proofs of these properties depend on the results in the previous sections. In Section \ref{sec:meress}, we classify all essential meridional spheres in the complement of $K_\alpha$ that have fewer than 14 punctures. In Section \ref{sec:bridgenumber}, we determine that bridge position and thin position for $K_\alpha$ do not coincide. In Section \ref{sec:lemmas}, we introduce some additional lemmas. Finally, in Section \ref{sec:134}, we show that the projection given in Figure \ref{fig:counter} of each of the knots $K_\alpha$ is thin by checking the widths of a relatively small number of possible thin positions.

\section{Definitions and Preliminaries}

Let $K$ be a knot embedded in $S^3$. We will denote a regular neighborhood of $K$ by $\eta(K)$. An {\em essential meridional surface} in the knot complement is a surface with meridional boundary that does not have any compressing disks in the knot complement and that is not boundary parallel in $S^3 -\eta(K)$.  Let $F$ be a meridional surface embedded in the complement of $K$. A {\em cut-disk} for $F$ is a disk $D^c \subset S^3$ such that $D^c \cap F =\bdd D^c$, $|D^c \cap K|=1$ and the annulus $D^c-\eta(K)$ is not parallel  in the knot complement to a subset of $F-\eta(K)$. In particular, if $K$ is prime, the curve $\bdd D^c$ is not parallel to a boundary component of $F-\eta(K)$. We use the term {\em c-disk} to refer to either a compressing or a cut-disk.

Consider the standard height function $h: S^3 \to \mathbb{R} \cup \{- \infty, + \infty\}$ and suppose $K$ is in general position with respect to $h$. If $t$ is a regular
value of $h|_K$, $h^{-1}(t)$ is called a {\em level sphere} with width
$w(h^{-1}(t))=|K\cap h^{-1}(t)|$. If $c_{0}<c_{1}<...<c_{n}$ are all the
critical values of $h|_K$, choose regular values
$r_{1},r_{2},...,r_{n}$ such that $c_{i-1}<r_{i}<c_{i}$. Then the
{\em width of $K$ with respect to $h$} is defined by $w(K,h)=\sum
w(h^{-1}(r_{i}))$. The {\em width} of $K$, $w(K)$, is the minimum of $w(K',h)$
over all knots $K'$ isotopic to $K$. We say that $K$ is in {\em thin position} if $w(K,h)=w(K)$.  Note that by removing a neighborhood of the north and south pole, we can assume $K \subset S^2 \times I$ and define width there. We will switch between these two ambient spaces freely during this discussion. More details about thin position and basic results can be found in \cite{Schar1}.

A level sphere $h^{-1}(t)$ is called {\em thin}
if the highest critical point for $K$ below it is a maximum and the
lowest critical point above it is a minimum. If the highest critical point for $K$ below $h^{-1}(t)$ is a minimum and the
lowest critical point above it is a maximum, the level sphere is called {\em thick}. As the lowest critical point of $K$ is a minimum and the highest is a maximum, a thick level sphere can always be found. It is possible that the knot does not have any thin spheres with respect to some height function. When this occurs the unique thick sphere is called a {\em bridge sphere} and the knot is said to be in {\em bridge position}.

We will use the following result found in \cite{SchSch} to simplify our computations.

\begin{lemma}\cite[Lemma 6.2]{SchSch}
Let $K$ be an embedding of a knot in $S^3$ and let $h:S^3 \rightarrow \mathbb{R} \cup \{- \infty, + \infty\}$ be the standard height function on $S^3$. If $\{a_i\}$, $i=0,...n$ and $\{b_j\}$, $j=0,...n+1$ are the widths of all thin and all thick spheres respectively, then $$w(K)=\frac{\Sigma_{j=0}^{n+1} b_j^2-\Sigma_{i=0}^n a_i^2}{2}.$$
\end{lemma}

Unless otherwise stated, we will always consider all level spheres to lie in the knot complement, i.e., they are always meridional surfaces. A key ingredient to our proofs is the behavior of c-disks for the thin level spheres. We review some already known results here and then develop some new results in the next section.

\begin{theorem} \cite{Wu} \label{thm:thinincomp}
Suppose $K$ is a prime knot in thin position and let $P$ be the thin sphere of lowest width. Then $P$ is incompressible.
\end{theorem}

\begin{theorem} \cite[Theorem 8.1]{T1} \label{thm:sixincomp}
Suppose $K$ is a prime knot in thin position and $P$ is a minimal width thin sphere. If $P'$ is a thin sphere so that $w(P')=w(P)+2$, then $P'$ is incompressible.
\end{theorem}

\begin{defin}A {\em tangle}, $\mathcal{R}$, is a tuple, $(B_R,R)$, where $B_R$ is a 3-ball or $S^2\times I$ and $R$ is a collection of mutually disjoint properly embedded knots and arcs in $B_R$. $\mathcal{R}$ is an {\em $n$-strand tangle} if $R$ contains no knots and exactly $n$ arcs.
\end{defin}

Given any tangle $\mathcal{T}$ in a ball or in $S^2 \times I$ and a height function $h$, the critical points of the tangle can be organized into braid boxes as follows: suppose $t_1$ and $t_2$ are adjacent thin levels. Then a {\em braid box} $B_{[t_1,t_2]}\subset h^{-1}[t_1,t_2]$ is a ball containing  $T\cap h^{-1}[t_1,t_2]$. In this ball the tangle has some number of minima of $T$ followed by some number of maxima of $T$. Given a tangle $\mathcal{S}$ in $S^2 \times I$ and a c-disk $D^c$ for $S^2 \times \{0\}$, then $D^c$ naturally decomposes $\mathcal{S}$ into two subtangles $\mathcal{S}_{\alpha}$ and $\mathcal{S}_{\beta}$. These subtangles can be decomposed into braid boxes $\{B_{[a_i^-,a_i^+]}\}$ and $\{B_{[b_j^-,b_j^+]}\}$. For more details see \cite{T2}.

\begin{theorem}\cite[Lemma 9.1]{T2} \label{thm:disjoint}
Let $L$ be a prime tangle embedded in $S^2\times I$, let $P$ be
a level sphere for $L$ and let $D^*$ be a c-disk for $P$ that does not have any saddles with respect to the usual height function on $S^2\times I$. Then there exists
a horizontal isotopy $\nu$ which keeps $D^*$ fixed such that if $\{B_{[a_i^-,a_i^+]}\}$ and $\{B_{[b_j^-,b_j^+]}\}$ are the collections of braid boxes for the proper tangles $\alpha$ and $\beta$ respectively, for any $i$ and $j$ the intervals $[a_i^-,a_i^+]$ and $[b_j^-,b_j^+]$ are disjoint.\end{theorem}

\begin{rmk} As $\nu$ is a horizontal isotopy it does not change the total number or the heights of the thin spheres for $L$. \end{rmk}

\begin{cor}\label{cor:altthin}
Let $L$ be a prime knot embedded in $S^2\times I$, let $P$ be
a thin sphere for $L$, let $D^*$ be a vertical c-disk for $P$ above it and let $P'$ be the thin sphere directly above $D^*$. Then either there are some thin spheres between $P$ and $P'$ or all critical points for $L$ between $P$ and $P'$ are on the same side of $D^*$. In particular, in the latter case c-compressing $P$ along $D^*$ results in a component parallel to $P'$.
\end{cor}

\begin{proof}
Suppose $P$ and $P'$ are adjacent thin spheres and suppose $L$ has critical points on both sides of $D^*$. By Theorem \ref{thm:disjoint}, we may assume that the braid boxes for the two sides are disjoint. However, a level sphere that is disjoint from all braid boxes is necessarily thin and therefore $P$ and $P'$ are not adjacent.
\end{proof}

\section{Vertical cut-disks}\label{Vertical cut-disks}

Let $K$ be a knot in $S^3$ and let $h:S^3 \rightarrow \mathbb{R} \cup \{- \infty, + \infty\}$ be the standard height function on $S^3$. Suppose that $K$ is in thin position with respect to $h$. Let $P=h^{-1}(r)$ be a level sphere and suppose $P$ has a c-disk, $C$, that lies above it.

We first introduce some notation and definitions. Figure \ref{fig:labels1.eps} illustrates all of the terminology outlined below. Let $\digamma_{C}$ be the singular foliation on the
c-disk $C$ induced by $h|_{C}$. A {\em saddle} is any leaf of this foliation homeomorphic to the wedge of two circles. By standard position, we can assume that all saddles of $\digamma_{C}$ are disjoint from $K$.

Given a saddle $\sigma = s_{1}^{\sigma} \vee s_{2}^{\sigma}$ in a level sphere $S_{\sigma}=(h^{-1}\circ h)(\sigma)$, let $D_{1}^{\sigma}$ be the closure of the component of $S_{\sigma}-s_{1}^{\sigma}$ that is disjoint from $s_{2}^{\sigma}$ and $D_{2}^{\sigma}$ be the closure of the component of $S_{\sigma}-s_{2}^{\sigma}$ that is disjoint from $s_{1}^{\sigma}$.

A subdisk $D$ in $\digamma_{C}$ is monotone if its boundary is entirely contained in a leaf of $\digamma_{C}$ and the interior of $D$ is disjoint from every saddle in $\digamma_{C}$. In practice, we will use the term subdisk in a slightly broader sense, allowing $\partial D$ to be immersed in $C$, where if $\partial D$ is immersed, then $\partial D$ is a saddle. We say a monotone disk is {\em outermost} if its boundary is $s_{i}^{\sigma}$ for some saddle $\sigma$ and label the disk $D_{\sigma}$. Similarly, if some $s_{i}^{\sigma}$ bounds an outermost disk $D_{\sigma}$, we say $\sigma$ is an outermost saddle. It will usually be the case that only one of $s_{1}^{\sigma}$ and $s_{2}^{\sigma}$ is the boundary of an outermost disk, so, our convention is to relabel so that $\partial D_{\sigma} = s_{1}^{\sigma}$.

Suppose $\sigma$ is an outermost saddle. The level sphere $S_{\sigma}$ cuts $S^{3}$ into two 3-balls. The ball that contains $D_{\sigma}$ is again cut by $D_{\sigma}$ into two 3-balls $B_{\sigma}$ and $B'_{\sigma}$. We choose the labeling of $B_{\sigma}$ and $B'_{\sigma}$ so that $\partial B_{\sigma} =D_{1}^{\sigma} \cup D_{\sigma}$.

We say $\sigma$ is an {\em inessential saddle} if $\sigma$ is an outermost saddle and $D_{\sigma}$ is disjoint from $K$. An {\em n-punctured disk} denotes a disk embedded in $S^{3}$
that meets $K$ transversely in exactly $n$ points. An embedded simple closed
curve in a c-disk $C$ is {\em c-inessential} if it bounds
a 1-punctured disk in $C$. Similarly, $\sigma$ is a {\em c-inessential saddle} if $\sigma$ is an outermost saddle and $D_{\sigma}$ meets $K$ exactly once. We say $\sigma$ is a {\em removable saddle} if
$\sigma$ is an outermost saddle where $D_{\sigma}$ has a unique maximum (minimum) and $h|_{K \cap B_{\sigma}}$ has a
local end-point maximum (minimum) at every point of $K \cap D_{\sigma}$. See Figure \ref{fig:removeable.eps}.

\begin{figure}[h]
\centering \scalebox{1}{\includegraphics{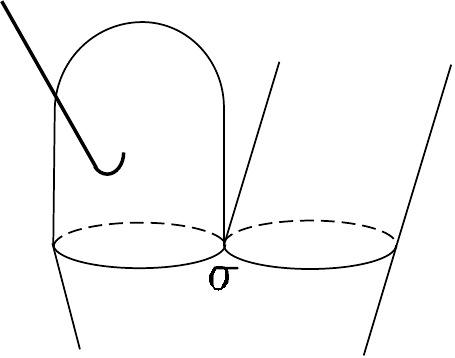}}
\caption{}\label{fig:removeable.eps}
\end{figure}

We say a saddle $\sigma$ in $\digamma_{C}$ is {\em standard} if there is a monotone disk $E_{\sigma}$ such that $\partial(E_{\sigma}) = \sigma$. If $\sigma$ is a standard saddle, $A_{\sigma}$ is the 3-ball with boundary $E_{\sigma} \cup D_{1}^{\sigma} \cup D_{2}^{\sigma}$ such that $A_{\sigma}\cap S_{\sigma} = D_{1}^{\sigma} \cup D_{2}^{\sigma}$.

By general position arguments, we can assume every saddle $\sigma$
in $\digamma_{C}$ has a bicollared neighborhood in $C$ that is
disjoint from $K$ and all other singular leaves of $\digamma_{C}$.  The boundary of this bicollared neighborhood
consists of three circles $c_{1}^{\sigma}$, $c_{2}^{\sigma}$, and $c_{3}^{\sigma}$ where
$c_{1}^{\sigma}$ and $c_{2}^{\sigma}$ are parallel to $s_{1}^{\sigma}$ and $s_{2}^{\sigma}$
respectively. We can assume $c_{1}^{\sigma}$, $c_{2}^{\sigma}$, and $c_{3}^{\sigma}$ are level
with respect to $h$ and that $c_{1}^{\sigma}$ and $c_{2}^{\sigma}$ lie in the same
level surface. The terminology for this section is summarized in Figure \ref{fig:labels1.eps}.

\begin{figure}[h]
\centering \scalebox{1.2}{\includegraphics{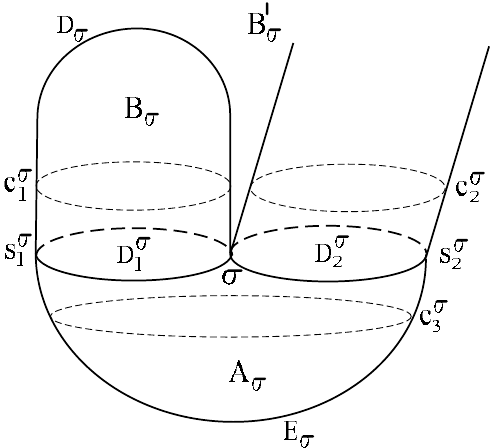}}
\caption{}\label{fig:labels1.eps}
\end{figure}

\begin{defin} A c-disk for a level sphere is {\em vertical} if it does not have any saddles with respect to $h$.

\end{defin}

In \cite{Wu}, it is shown that any compressing disk is isotopic to a vertical compressing disk by an isotopy that does not change the width of the knot. This is generally not true of cut-disks, but we will now show that under certain conditions cut-disks can also be isotoped to be vertical. Many of the arguments in this section are extensions of techniques developed by Schultens in \cite{Schu2} and extended by Blair in \cite{Bl}. We will need the following definitions.

 \begin{defin} A c-disk, $C$, for a level surface $P$ is taut with respect to $h$ if the number of saddles in $\digamma_{C}$ is minimal subject to the condition that $K$ is a minimal width embedding and $P$ is a level surface.\end{defin}

\begin{defin}
Following \cite{Kob}, a sphere $P$ in $S^3$ is called {\em bowl-like} with respect to a height function $h$ if it can be decomposed into two disks, $E_1$ and $E_2$, glued along their boundary such that $E_1$ is contained in a level surface for $h$ and $E_2$ is a monotone disk disjoint from $K$.
\end{defin}

\begin{lemma}\label{inesssaddle} Assume $P$ and $P'$ are adjacent thin spheres with $P'$ above $P$ and $C$ is a cut-disk for $P$ above it but disjoint from $P'$. We allow the special case where $P$ is the highest thin sphere and $P'$ is a level sphere above it disjoint from $K$. If $\digamma_{C}$ contains an inessential saddle, then $C$ is not taut.\end{lemma}

\begin{proof}
Suppose $\sigma$ is an inessential saddle in $\digamma_{C}$.

\medskip

\noindent \textbf{Case 1:} Suppose $D_{\sigma}$ contains a unique minimum. Call this point $a$.

\noindent{\em Case 1a:} Additionally, suppose $B_{\sigma}$ does not contain $-\infty$. Use the ``Pop out Lemma" \cite[Lemma 2]{Schu2} to eliminate $\sigma$, see Figure \ref{fig:popout.eps}. This isotopy fixes all of $S^3$ below $D_{\sigma}$ and above $P'$. However, $D_{\sigma}$ is contained strictly above the level sphere $P$. Hence, this isotopy eliminates $\sigma$ while fixing $K$ below $P$ and above $P'$ and not creating any new critical points for $h|_K$. Since all maxima of $h|_K$ are above all minima of $h|_K$ in the region between $P$ and $P'$, altering the relative heights of the critical points without creating any new critical points can only decrease the width of $K$. This is a contradiction to $C$ being taut.

\begin{figure}[h]
\centering \scalebox{0.6}{\includegraphics{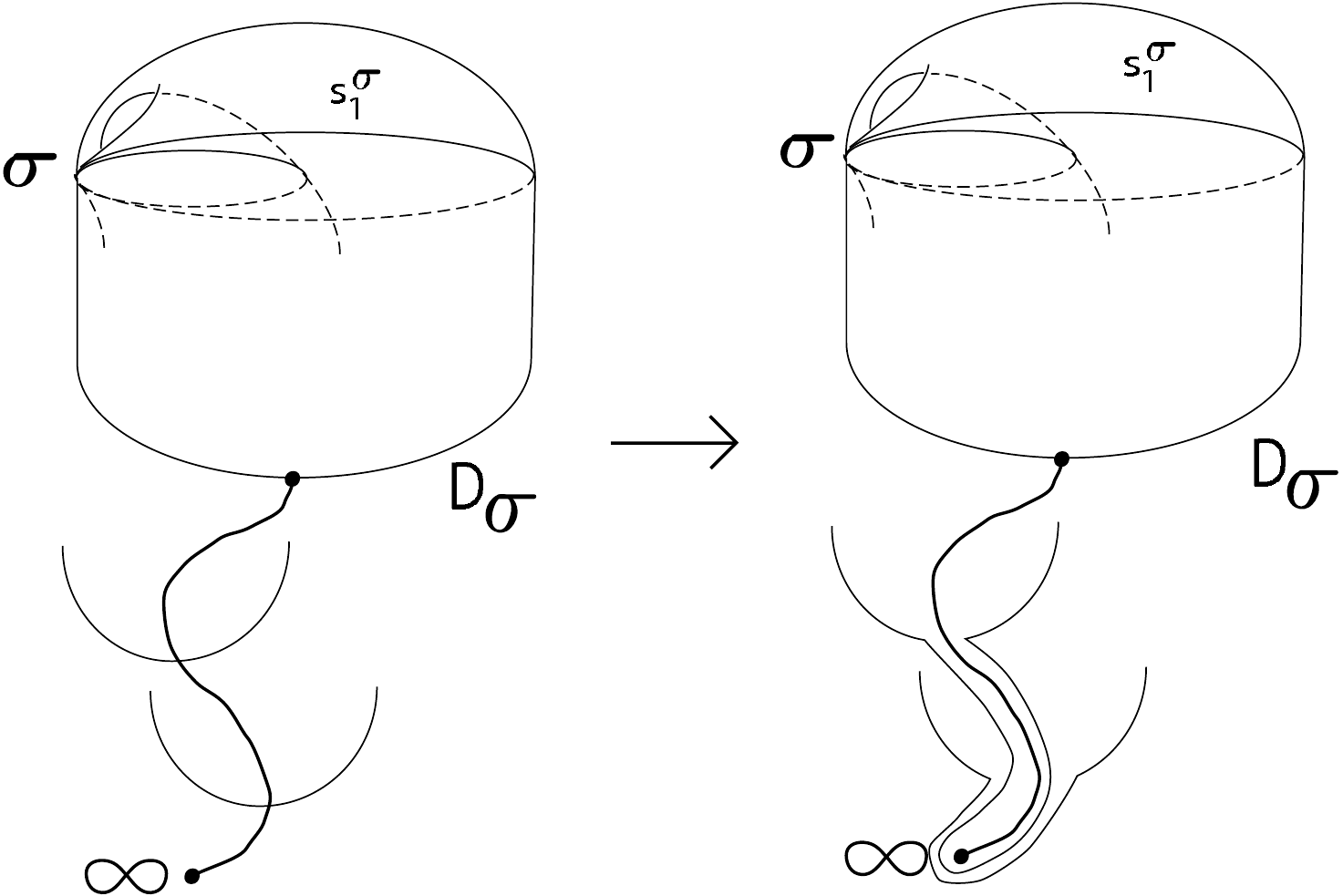}}
\caption{}\label{fig:popout.eps}
\end{figure}

\noindent{\em Case 1b:} Additionally, suppose $B_{\sigma}$ contains $-\infty$. We will describe a sequence of isotopies that allows us to decrease the number of saddles, see Figure \ref{fig:c-disk2}.

Let $\alpha$ be a monotone
arc with endpoints $a$ and $-\infty$ that misses $K$ and intersects
$C$ only at local minima.  Label the points of $\alpha \cap C$ in
order of decreasing height with $a, a_{1},..., a_{n}$. Since $C$ lies above $P$, $\alpha$ meets $P$ in a single point $b$ where $h(a_n) > h(b)$. See Figure \ref{fig:c-disk2}. Again by general position, we can assume none of the $a_{i}'s$ or $b$ lie on $K$. The following isotopy is a modification of the isotopy presented in \cite[Lemma 1]{Schu2}.

Let $S_{+}$ be a level sphere contained in a small neighborhood of
$- \infty$ such that $S_{+}$ does not meet $K$ or $P$. Let
$\alpha_{b}$ be a subarc of $\alpha$ with endpoints $b$ and $-\infty$.  Enlarge $\alpha_b$ slightly to be a vertical solid
cylinder $V$ such that $\partial V$ consists of a small neighborhood of $b$ in $P$, a small disk in $S_{+}$ and a vertical annulus, $A$. Replacing $P$ with the isotopic surface $(C-V) \cup A \cup (S_{+}-V)$ represents an isotopy of
$P$ in $S^{3}-K$ that fixes $K$ and results in $P$ being bowl-like, but not level.

Let
$\alpha_{n}$ be a subarc of $\alpha$ with endpoints $a_{n}$ and $-\infty$.  Enlarge $\alpha_{n}$ slightly to be a vertical solid
cylinder $V$ such that $\partial V$ consists of a small neighborhood of $a_{n}$ in $C$, a small disk in $S_{+}$ and a vertical annulus, $A$. Replacing $C$ with the
cut-disk $(C-V) \cup A \cup (S_{+}-V)$ represents an isotopy of
$C$ in $S^{3}-K$ that fixes $K$, does not change the number of saddles of $\digamma_C$ and preserves $P$ as bowl-like.

By induction on $n$, we can assume $\alpha$ is disjoint from $C$ and $P$
except at the point $a$.  By isotopying $D_{\sigma}$ to a new disk
$D^{*}_{\sigma}$ in the manner described above, we have enlarged
$B'_{\sigma}$ to contain $- \infty$ and shrunk
$B_{\sigma}$ so that it is disjoint from $- \infty$.  After
a small tilt so that $h$ again restricts to a Morse function on
$D^{*}_{\sigma}$, $\digamma_{D^{*}_{\sigma}}$ is a collection of
circles and one maximum. The resulting cut-disk $C^{*}$ is
isotopic to $C$ via an isotopy that leaves $\sigma$ and $K$ fixed and does not change the number of saddles of $\digamma_{C}$.

By the ``Pop out Lemma" \cite[Lemma 2]{Schu2}, we can eliminate $\sigma$ without introducing any new maxima to $h_{K}$ or new saddles to $\digamma_C$ and while preserving $P$ as bowl-like.

Since $P$ is now bowl-like, it can be decomposed into two disks $E_1$ and $E_2$ as in the definition of bowl-like. Let $a$ be the unique minimum on $E_2$. Again choose a monotone arc $\alpha$ with endpoints $a$ and $-\infty$ that misses $K$ and intersects
$C$ only at local minima. The arc $\alpha$ is disjoint from $P$ except at $a$. Label the points of $\alpha \cap C$ in
order of decreasing height with $a_{1},..., a_{n}$. Again by general position, we can assume none of the $a_{i}'s$ lie in $K$. Repeat the above argument to produce an isotopic copy of $C$ with the same number of saddles that is disjoint from $\alpha$. Horizontally shrink and vertically lower $P$ until it is strictly below all of $C$. Please see the last isotopy in Figure \ref{fig:c-disk2}. Let $S_-$ be 2-sphere boundary of a regular neighborhood of $-\infty$ so that $S_-$ is disjoint from $K$, $P$ and $C$. After lowering $P$ into the neighborhood of $-\infty$ and expanding $P$ to fill the neighborhood, we have isotoped $P$ to $S_-$ while preserving the width of $K$ below $P$ and above $P'$. Since we have produced an isotopy that decreases the number of saddles of $\digamma_C$ while not introducing any new maxima to $h|_K$ and fixing $K$ below $P$ and above $P'$, then $C$ is not taut.

\medskip

\noindent\textbf{Case 2:} Suppose $D_{\sigma}$ contains a unique maximum. The argument is symmetric to the one in case 1 above. If necessary, isotope $P'$ to be bowl-like to guarantee that $B_{\sigma}$ does not contain $\infty$, then apply the ``Pop out Lemma" to reduce the number of saddles for $C$. Finally, restore $P'$ to be level. As in Case 1, these isotopies do not affect the width of $K$ below $P$ and above $P'$ and do not introduce new critical points for $K$, so they do not increase the width of the knot.

\end{proof}

\begin{figure}[h]
\centering \scalebox{0.4}{\includegraphics{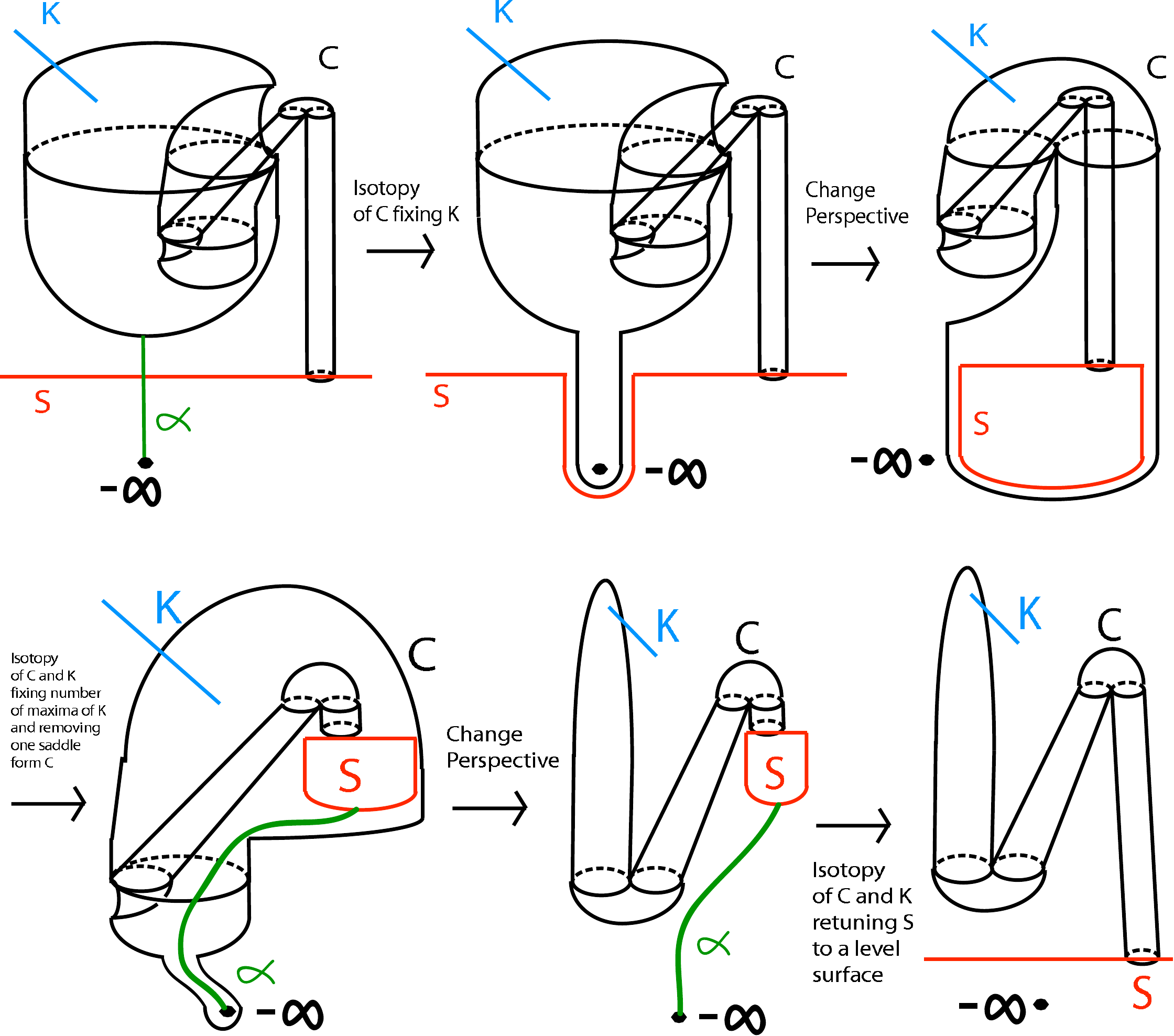}}
\caption{}\label{fig:c-disk2}
\end{figure}

\begin{cor}\label{non-standard} Assume $P$ and $P'$ are adjacent thin spheres with $P'$ above $P$ and $C$ is a cut-disk for $P$ above it but disjoint from $P'$. If $\digamma_C$ contains a non-standard saddle $\sigma$, then $C$ is not taut.
\end{cor}

\begin{proof}
Suppose $\digamma_C$ contains a non-standard saddle $\sigma$. By definition of non-standard, $c^{\sigma}_{3}$ does not bound a monotone disk in $C$. Since $c^{\sigma}_1$, $c^{\sigma}_2$ and $c^{\sigma}_3$ are the boundary components of an embedded pair of pants in $C$, then two of these curves, assume $c^{\sigma}_1$ and $c^{\sigma}_3$, bound, possibly punctured, disks in $C$ denoted $E^*_1$ and $E^*_3$ respectively. Both $E^*_1$ and $E^*_3$ are disjoint from $c^{\sigma}_1 \cup c^{\sigma}_2 \cup c^{\sigma}_3$. Either $\sigma$ is outermost or a saddle in $\digamma_{E^*_1}$ is outermost. By hypothesis, $E^*_3$ contains a saddle. Hence, $E^*_3$ contains an outermost saddle. Since $\digamma_C$ contains two outermost saddles and $C$ meets $K$ exactly once, one of these outermost saddles is inessential. By Lemma \ref{inesssaddle}, $C$ is not taut. The argument follows similarly if $c^{\sigma}_1$ and $c^{\sigma}_2$ or $c^{\sigma}_2$ and $c^{\sigma}_3$ bound disks in $C$.
\end{proof}

The cut-disk $C$ decomposes the $3$-ball above $P$ into two $3$-balls $B_{1}$ and $B_{2}$. Let $\sigma$ be a saddle in $\digamma_{C}$ and
$Q$ be the level sphere either just above or just below $\sigma$ that contains $c_{1}^{\sigma}$ and $c_{2}^{\sigma}$.
The surface $Q-(c_{1}^{\sigma} \cup c_{2}^{\sigma})$ is composed of two disks and an annulus $A$.
If a collar of $\partial A$ in $A$ is contained in $B_{1}$, then we
say $\sigma$ is {\em unnested} with respect to $B_{1}$. If not, we say
$\sigma$ is nested with respect to $B_{1}$. We define nested and
unnested with respect to $B_{2}$ similarly.  Note that nested with
respect to $B_{1}$ is the same as unnested with respect to $B_{2}$
and nested with respect to $B_{2}$ is unnested with respect to
$B_{1}$.

 Two saddles $\sigma = s_{1}^{\sigma}
\vee s_{2}^{\sigma}$ and $\tau =s_{1}^{\tau} \vee s_{2}^{\tau}$ in $\digamma_{C}$ are
{\em adjacent} if, up to subscript labels, $s_{1}^{\sigma}$ and $s_{1}^{\tau}$
cobound an annulus in $C$ that is disjoint from $s_{2}^{\sigma}$, $s_{2}^{\tau}$, all other saddles,
and $K$. Recall that, if $\sigma$ is a standard saddle, $E_{\sigma}$ is the monotone disk in $C$ with boundary $\sigma$.

\begin{lemma}\label{nested} Assume $P$ and $P'$ are adjacent thin spheres with $P'$ above $P$ and $C$ is a cut-disk for $P$ above it but disjoint from $P'$. If $\sigma$ and $\tau$ are adjacent saddles in $\digamma_C$
such that $\sigma$ and $\tau$ are nested with respect to
different 3-balls, then $C$ is not taut.\end{lemma}

\begin{proof}
Assume $\sigma$ and $\tau$ are adjacent saddles in $\digamma_C$
such that $\sigma$ and $\tau$ are nested with respect to
different 3-balls. By Corollary \ref{non-standard}, we can assume both $c_{3}^{\sigma}$ and $c_{3}^{\tau}$ bound monotone disks $E_{\sigma}$ and $E_{\tau}$ respectively.

Let $A$ be the monotone annulus in $C$ with boundary $s_{1}^{\sigma} \cup s_{1}^{\tau}$. If $K$ meets $A \cup E_{\sigma} \cup E_{\tau}$ (the annulus in $C$ with boundary $s_{2}^{\sigma} \cup s_{2}^{\tau}$), then one of $s^{2}_{\sigma}$ or $s^{2}_{\tau}$ bounds a disk $E^*$ in $C$ that is disjoint from $K$ and an outermost saddle of $\digamma_{E^*}$ is inessential. By Lemma \ref{inesssaddle}, $C$ is not taut. Hence, we can assume $K$ is disjoint from $A \cup E_{\sigma} \cup E_{\tau}$.

Without loss of generality, suppose $\sigma$ lies above $\tau$. Let $B$ be the $3$-ball in $S^3$ with boundary $D_{1}^{\tau} \cup A \cup E_{\sigma} \cup D_{2}^{\sigma}$. Use the isotopy constructed in \cite[Lemma 3]{Schu2} to eliminate $\tau$ without introducing any new saddles to $\digamma_C$ and without introducing any new maxima to $h|_K$. See Figure \ref{fig:nested}. This isotopy is supported in a neighborhood of $B$, hence, everything below $P$ and everything above $P'$ is fixed. Thus, $C$ is not taut.

\end{proof}

\begin{figure}[h]
\centering \scalebox{.5}{\includegraphics{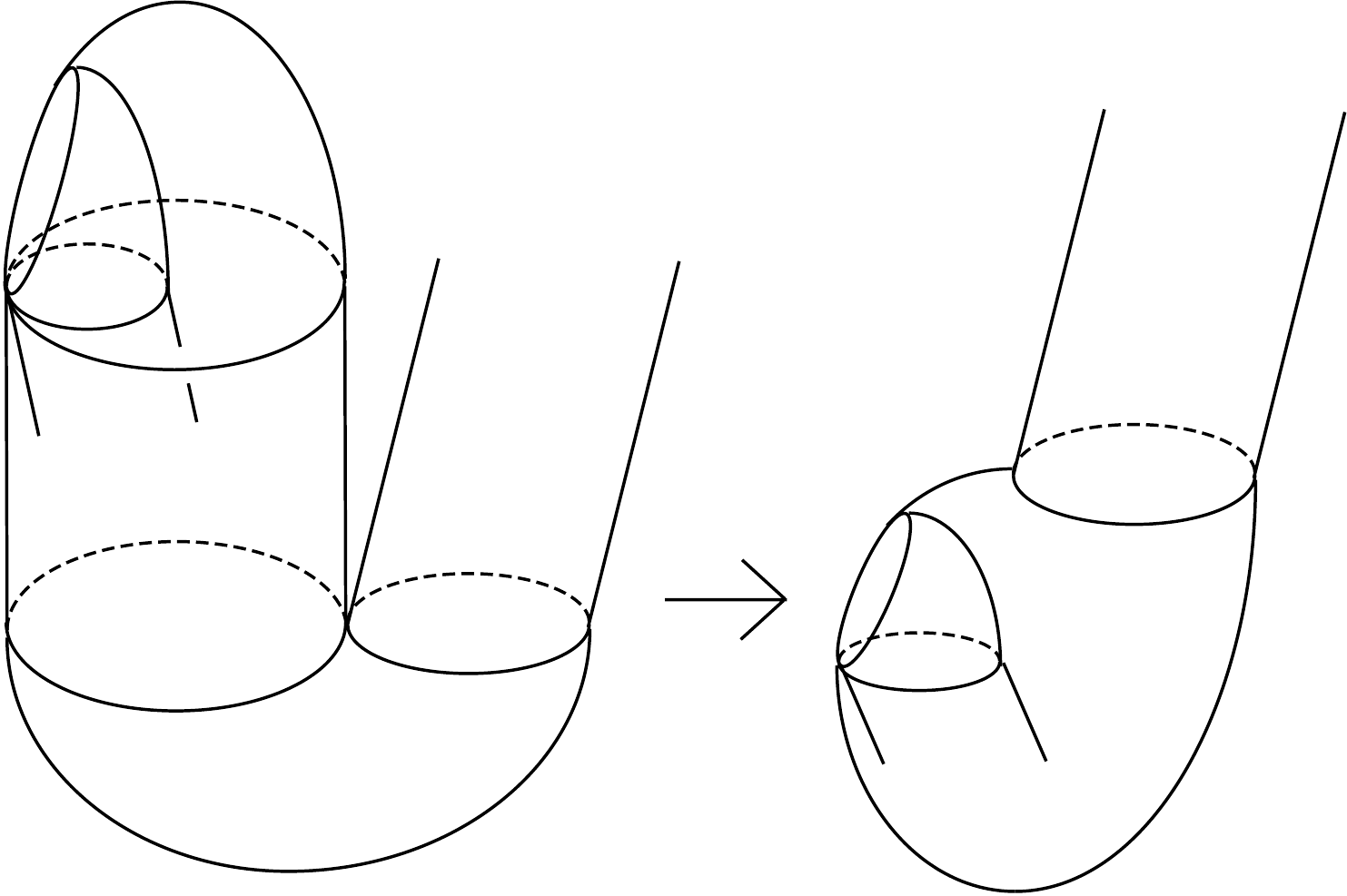}}
\caption{}\label{fig:nested}
\end{figure}

\begin{lemma}\label{C-out} Assume $P$ and $P'$ are adjacent thin spheres with $P'$ above $P$ and $C$ is a cut-disk for $P$ above it but disjoint from $P'$. If $\sigma$ is an outermost saddle in $\digamma_{C}$ such that $B_{\sigma}\cap C \neq \emptyset$, then $C$ is not taut.\end{lemma}

\begin{proof}Assume $\sigma$ is nested with respect to $B_{1}$. Since all saddles in $\digamma_C$ are standard, by Corollary \ref{non-standard}, and no saddles in $\digamma_C$ are inessential, by Lemma \ref{inesssaddle}, then there is a labeling of the saddles of $\digamma_C$ as $\sigma = \sigma_1$, ..., $\sigma_n$ such that $\sigma$ is an outermost c-inessential saddle and $\sigma_i$ is adjacent to $\sigma_{i+1}$ for each $i\in \{1,...,n-1\}$. Inductively, by Lemma \ref{nested}, $\sigma$ being nested with respect to $B_1$ implies all saddles in $\digamma_C$ are nested with respect to $B_1$. Hence, $B_2$ can be decomposed into $B_{\sigma}$ together with a collection of vertical solid cylinders and solid elbows. Thus, $B_{\sigma}$ is disjoint from $C$, contradicting the hypothesis.
\end{proof}

\begin{lemma}\label{cinesssaddle} Assume $P$ and $P'$ are adjacent thin spheres with $P'$ above $P$ and $C$ is a cut-disk for $P$ above it and disjoint from $P'$. If $\digamma_{C}$ contains a c-inessential saddle, then $C$ is not taut.\end{lemma}

\begin{proof}
Suppose $\sigma$ is a c-inessential saddle in $\digamma_{C}$.

\medskip
The isotopy utilized in the following claim was originally described in \cite[page 5]{Schu2}.

\textbf{Claim:} If $D_{\sigma}$ has a minimum and is punctured by $K$, we may assume that $h|_{K \cap B_{\sigma}}$ also has a local minimum at $K \cap D_{\sigma}$. Symmetrically, if $D_{\sigma}$ has a maximum and is punctured by $K$, we may assume that $h|_{K \cap B_{\sigma}}$ also has a local maximum at $K \cap D_{\sigma}$.

{\em Proof of claim:} Suppose $D_{\sigma}$ has a minimum and $h|_{K \cap B_{\sigma}}$ has a local maximum at $p_{\sigma} = K \cap D_{\sigma}$. Let $x$ be the minimum of $K$ that is nearest $p_{\sigma}$ and inside $B_{\sigma}$. Let $\alpha$ be the monotone subarc of $K$ inside $B_{\sigma}$ with boundary points $p_{\sigma}$ and $x$. Let $\beta$ be a monotone arc in $D_{\sigma}$ with endpoints $p_{\sigma}$ and $y$ such that $h(y)=h(x)$. Let $\delta$ be a level arc contained in $B_{\sigma}$ connecting $x$ to $y$. Let $E^*$ be the vertical disk with boundary $\alpha \cup \beta \cup \delta$ that is embedded in $B_{\sigma}$. We can assume the interior of $E^*$ meets $K$ transversely in a collection of points $k_1, ..., k_n$ where $h(k_1)>h(k_2)> ... >h(k_n)$. It is important to note that if $C$ meets the interior of $E^*$ then $C$ is not taut, by Lemma \ref{nested}. Hence, we can assume $C$ is disjoint from $E^*$. Let $\mu_i$ be the arc corresponding to a small neighborhood of $k_i$ in $K \cap B_{\sigma}$ for each $i$.

Replace $\mu_n$ with a monotone arc which starts at an end point of $\mu_n$ , runs
parallel to $E^*$ until it nearly reaches $D_{\sigma}$, travels along $D_{\sigma}$ until it returns
to the other side of $E^*$, travels parallel to $E^*$ (now on the opposite side)
and connects to the other end point of $\mu_n$. The result is isotopic to $K$, does not change the number of maxima of $h|_K$ and reduces $n$. By
induction on $n$, we may assume that $K \cap E^* = \emptyset$. Isotope $\alpha$ along $E^*$
until it lies just outside of $D_{\sigma}$ except where it intersects $D_{\sigma}$ exactly at
the point $y$. After a small tilt of $K$, $h|_{K\cap B_{\sigma}}$ now has a local minimum at $p_{\sigma}$.
$\square$
\medskip

After applying the isotopy given by the claim, we can repeat the arguments in Lemma \ref{inesssaddle} to remove this saddle. We give a very brief summary here.

\medskip

\noindent\textbf{Case 1:} Suppose $D_{\sigma}$ contains a unique maximum. If necessary, isotope $C$, $P'$ and all level spheres above $P'$ so that $B_{\sigma}$ is disjoint from $+\infty$. This isotopy replaces $P'$ and all level spheres above it with bowl-like spheres. Next, use the isotopy in \cite[Lemma 3]{Bl} to eliminate $\sigma$, see Figure \ref{fig:removeableiso}. Finally isotope $C$, $P'$ and all the other spheres that used to be level, to be level again. This can be done without introducing any new saddles or critical points for $K$. This isotopy fixes all of $S^3$ below $h(\sigma)$ and above $P'$. Additionally, this isotopy removes at least one saddle of $C$ and does not create any new critical points for $h|_K$ in the thick region between $P$ and $P'$. Since all maxima of $h|_K$ are above all minima of $h|_K$ in the region between $P$ and $P'$, altering the relative heights of the critical points without creating any new critical points can only decrease the width of $K$. Hence, this is a contradiction to $C$ being taut.

\begin{figure}[h]
\centering \scalebox{.7}{\includegraphics{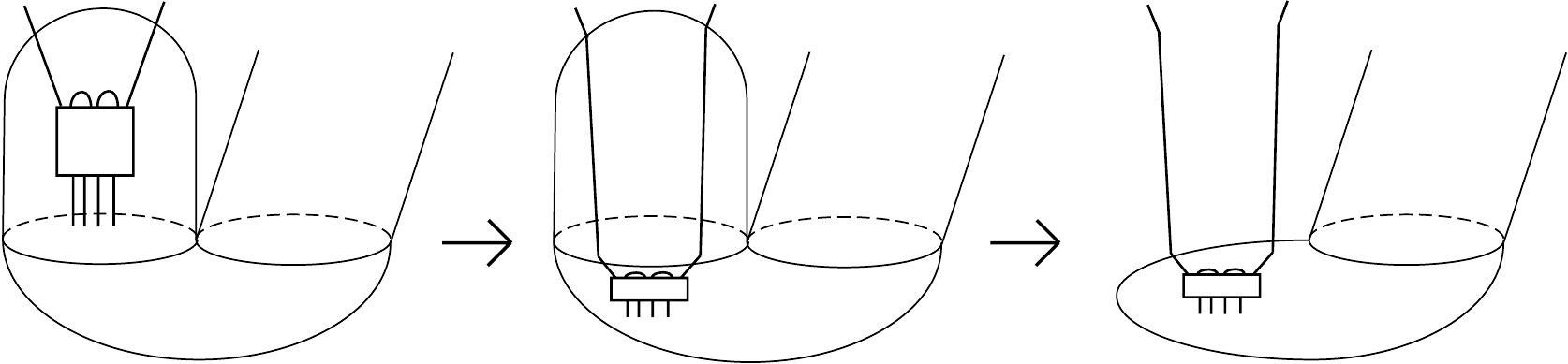}}
\caption{}\label{fig:removeableiso}
\end{figure}

\medskip

\noindent\textbf{Case 2:} Suppose $D_{\sigma}$ contains a unique minimum.

Note that if $B_{\sigma}$ contains $-\infty$, then $B_{\sigma}$ contains $P$ and $B_{\sigma}
\cap C \neq \emptyset$. By Lemma \ref{C-out}, $C$ is not taut. Therefore, we may assume that $B_{\sigma}$ does not contain $-\infty$. By the Claim, we may also assume that $h|_{K \cap B_{\sigma}}$ has a local minimum at
$p_{\sigma} = K \cap D_{\sigma}$. We can then use the isotopy from \cite[Lemma 3]{Bl} to eliminate $\sigma$. As in the previous case this isotopy can only decrease the width of $K$ and decreases the number of saddles. Hence, this is a contradiction to $C$ being taut.

 \end{proof}

\begin{theorem} \label{thm:diskisvertical}
 Let $K$ be a knot in $S^3$ in thin position and let $P$ and $P'$ be two adjacent thin spheres or let $P$ be the highest thin sphere for $K$. Suppose $P$ is cut-compressible with the cut-disk $D^c$ such that $D^c \cap P'=\emptyset$. Then there is an isotopy of $K$ supported between $P$ and $P'$ that does not change the width of $K$ after which $D^c$ is vertical.
\end{theorem}

\begin{proof} We can assume that we have isotoped $C$ to be taut. By Lemma \ref{inesssaddle} and Lemma \ref{cinesssaddle}, $\digamma_C$ has no inessential and no c-inessential saddles. Hence, $C$ is vertical.
\end{proof}

The above theorem allows us use the following previously known result.

\begin{theorem}\label{thm:not compressible on one side}\cite{T2}
Let $K$ be a prime knot in thin position and suppose $P$ is a thin sphere. Let $D^*$ be a compressing disk or a vertical cut-disk for $P$, say above it. Then there is a thin sphere above $D^*$ and if $S_0$ is the lowest such thin level sphere, then $w(S_0)<w(P)$.

\end{theorem}

Combining Theorem \ref{thm:diskisvertical} and Theorem \ref{thm:not compressible on one side} we obtain the following corollary.

\begin{cor}\label{cor:thinabove}

Let $K$ be a prime knot in thin position and suppose $P$ is a thin sphere. Let $D^*$ be a c-disk for $P$, say above it. Then there is a thin sphere above $P$ and if $P'$ is the lowest such thin sphere, then either $D^* \cap P' \neq \emptyset$ or $w(P')<w(P)$.

\end{cor}

\section{Bridge surfaces}\label{sec:bridgesurfaces}

In the previous section, we focused on results pertaining to thin position for knots. Here we review results about a tangle in bridge position.   We begin with a brief review of the definition of a bridge surface and its distance. For more details see \cite{T3}.

Suppose $M$ is a $3$--manifold homeomorphic either to $S^3$, to a ball, or to $S^2 \times I$ and containing a properly embedded collection of knots and arcs, $T$. A sphere $\Sss$ is a {\em bridge sphere} for $M$ if $M-\Sss$ has two components each of which is either a ball or is homeomorphic to $S^2\times I$ and each component of $T-\Sss$ is boundary parallel in $M-\Sss$ to $\Sss$ or it is a vertical arc in $S^2 \times I$. Let $H^+$ and $H^-$ be the components of $M-\Sss$ and let $\tau^{\pm}=T \cap H^{\pm}$. We say that $(\Sss,(H^-, \tau^-), (H^+, \tau^+))$ is a {\em bridge splitting} for $(M,T)$.

If $B$ is a ball in $(M,T)$, under certain conditions $B \cap \Sss$ induces a bridge sphere for $(B, B\cap T)$ as described in the following lemma.

\begin{lemma}\label{lem:newbridge}
Suppose $P$ and $P'$ are two adjacent thin spheres with $P'$ above $P$. Suppose $D^*$ is a vertical c-disk for $P$ lying above it and disjoint from $P'$. Let $B$ be the ball cobounded by $D^*$ and $P$ disjoint from $P'$ and let $T=K\cap B$. Then there exists a thick sphere between $P$ and $P'$ such that the disk $\Delta=B \cap \Sigma$ together with the possibly once punctured disk that $\bdd \Delta$ bounds in $D^*$ is a bridge sphere for $(B, T)$.
\end{lemma}

\begin{proof}
Since $D^*$ is vertical we can lower all of the minima of $K$ in the region bounded by $P$ and $P'$ until each of the minima is below $K\cap D^*$. This isotopy fixes $D^*$ and alters $K$ only in a small neighborhood of its minima. After this isotopy, any level sphere between the highest minimum of $K$ in the region bounded by $P$ and $P'$ and $K\cap D^*$ is a thick sphere. Denote one such thick sphere by $\Sigma$. Since $D^*$ is vertical, $\beta=D^* \cap \Sigma$ is a single essential simple closed curve. Note that $\beta$ bounds a disk in $D^*$ that meets $K$ in at most one point, and an annulus that is disjoint from $K$.

Let $E$ be a bridge disk for some bridge contained in $T$. By redefining $E$, we may assume that $E \cap D^*$ consists only of arcs and these arcs necessarily have both of their endpoints in $\beta$. Let $\alpha$ be an outermost arc of intersection so that the disk $F$ that $\alpha$ cobounds with a segment of $\beta$ in $D^*$ is disjoint from $E$ and does not contain the puncture of $D^*$. We can redefine $E$ by replacing the subdisk $\alpha$ cuts off in $E$ with the disk $F$. This reduces $D^* \cap E$. Therefore, we may assume that all bridge disks for bridges inside $B$ are contained inside $B$. It follows that the sphere $R$ obtained by compressing or cut-compressing $\Sigma$ along the c-disk that $\beta$ bounds in $D^*$ is a bridge sphere for $T$.
\end{proof}

The {\em curve complex}, $\mathcal{C}(\Sigma,T)$, is
a graph with vertices corresponding to isotopy classes of
essential simple closed curves in $\Sigma-\eta(T)$. Two vertices are adjacent
in $\mathcal{C}(\Sigma,T)$ if their corresponding classes of curves
have disjoint representatives.

 Let $\mathcal{V}^+$ (respectively $\mathcal{V}^-$) be the set of all essential
  simple closed curves in $\Sigma-\eta(T)$ that bound disks in
  $H^+ -\eta(T)$ (respectively $H^- -\eta(T)$). Then the distance of the bridge splitting, $d(\Sigma,T)$, is defined to be the minimum distance between a vertex in $\mathcal{V}^+$ and a vertex in $\mathcal{V}^-$ measured in
  $\mathcal{C}(\Sigma,T)$ with the path metric.

We will need the following special case of a result proven in \cite{JT}.

 \begin{theorem}\cite[Theorem 4.4]{JT} \label{thm:boundbridge}
Suppose $N$ is a 3-sphere, a 3-ball or $S^2\times I$ containing a properly embedded collection of knots and arcs, $K$. Let $M$ be a submanifold homeomorphic to a 3-ball or to $S^2 \times I$ such that $T=K \cap M$ is a properly embedded tangle. Let $\Sigma$ be a bridge sphere of $(M,T)$ and let $\Sigma'$ be a bridge sphere of $(N,K)$. Then one of the following holds:

\begin{itemize}
\item  There is an isotopy of $\Sss'$ followed by some number of compressions and cut-compressions of $\Sigma'\cap M$ in $M$ giving a compressed surface $\Sigma''$ such that $\Sigma'' \cap M$ is parallel to $\Sigma$,
\item $d(\Sigma,T) \leq 2-\chi(\Sigma')$,
\item $\chi(\Sigma)\geq -3$.
\end{itemize}
\end{theorem}

The following result can be easily obtained by a simplified version of the proof of Theorem \ref{thm:boundbridge} given in \cite{JT} so we will not prove it here.

\begin{theorem}\label{thm:boundess}
Suppose $N$ is a manifold containing a containing a properly embedded collection of knots and arcs, $K$. Let $M$ be a submanifold homeomorphic to a ball or to $S^2 \times I$  such that $T=K \cap M$ is a collection of knots and arcs properly embedded in $M$. Let $\Sigma$ be a bridge surface for $(M,T)$ and $F$ be an essential separating surface in $N$. Then one of the following holds:

\begin{itemize}

\item $d(\Sigma,T) \leq 2-\chi(F)$,
\item $\chi(\Sigma)\geq -3$,
\item Each component of $F\cap (M-\eta(T))$ is boundary parallel in $M-\eta(T)$.
\end{itemize}
\end{theorem}

\section{Knotting}\label{sec:knotting}

The goal of this section is to produce 3-tangles so that the strands of the tangles are knotted in some sense while also controlling the distance between certain curves in the curve complex of the 6-punctured sphere. We will later show how these tangles can be inserted in the boxes in Figure \ref{fig:counter} to construct a knot that is in thin position. We begin by reviewing some definitions.

\begin{defin}
 An $n$-strand tangle,
$\mathcal{R}$, is {\em rational} if $R$ does not have any closed components and all arcs of $R$ can be simultaneously isotoped into $\partial(B_R)$. A tangle $\mathcal{T}$ is an {\em induced sub-tangle} of a tangle $\mathcal{R}$ if $B_T = B_R$ and $T \subseteq R$.
\end{defin}

\begin{defin}\label{def:BuildingTangles}Given a rational tangle $\mathcal{R}$ and a simple closed curve $\epsilon$ in $\partial(B_R)$, $(\mathcal{R},\epsilon)$ is an {\em equatorial pair} if $\epsilon$ is disjoint from $R$ and no arc in $R$ has both of its endpoints on the same component of $\partial(B_R)-\epsilon$. Given an equatorial pair $(\mathcal{R},\epsilon)$, an {\em equatorial sub-pair} is an equatorial pair $(\mathcal{T},\epsilon)$ where $\mathcal{T}$ is an induced sub-tangle of $\mathcal{R}$.\end{defin}

\begin{defin}\label{projection} Given an $n$-strand tangle, $\mathcal{R}$, and an equatorial pair, $(\mathcal{R},\epsilon)$, embed $B_R$ as the unit ball in $R^3$ such that $\epsilon$ is mapped to the unit circle in the xy-plane and all points of $R \cap \partial(B_R)$ lie on the unit circle in the xz-plane. A {\em projection} of $(\mathcal{R},\epsilon)$ is a projection of such an embedding into the xz-plane. If $\mathcal{R}$ is a 2-strand tangle, the {\em numerator closure} of $(\mathcal{R},\epsilon)$ is the knot obtained by connecting the endpoints of $R$ via two arcs in the unit circle in the xz-plane so that each of the arcs is disjoint from the unit circle in the xy-plane.\end{defin}

In the remainder of the section we will heavily rely on standard results about rational tangles. In particular, recall that each 2-strand rational tangle can be represented by a fraction $\frac{p}{q}$ where $(p,q)=1$. See \cite{Kauffman}, or \cite{Conway} for a detailed treatment of the subject.

\begin{theorem}\label{Conway} \cite{Conway} Two rational tangles are properly isotopic if and only if they have the same fraction.\end{theorem}

\begin{theorem}\cite{Schubert1956}\label{schubert} Consider two rational tangles with fractions $\frac{p}{q}$ and $\frac{p'}{q'}$. If $K(\frac{p}{q})$ and $K(\frac{p'}{q'})$ denote the corresponding rational knots obtained by taking the numerator closures of these tangles, then $K(\frac{p}{q})$ and $K(\frac{p'}{q'})$ are isotopic if and only if

\begin{itemize}

\item $p=p'$ and

\item either $q \equiv q'$ $mod$ $p$ or $qq' \equiv 1$ $mod$ $p$.
\end{itemize}
\end{theorem}

By Theorem \ref{schubert}, if $\mathcal{R}$ is a 2-strand tangle, $(\mathcal{R},\epsilon)$ has numerator closure the unknot if and only if $(\mathcal{R},\epsilon)$ has a projection as a $\frac{1}{q}$ rational tangle. We will call such an equatorial pair the {\em unpair}.

\begin{theorem}\label{inducedknotting}
If $\mathcal{R}$ is a rational 3-strand tangle and $(\mathcal{R},\epsilon)$ is an equatorial pair, then there exists an equatorial pair $(\mathcal{R},\delta)$ such that no 2-strand equatorial sub-pair is the unpair and $d(\epsilon,\delta)\leq 4$ (where $d$ is the distance function for the curve complex of the 6-punctured sphere).
\end{theorem}

\begin{proof}
Suppose $R$ consists of three arcs $\alpha$, $\beta$ and $\gamma$. The equatorial pair $(\mathcal{R},\epsilon)$ has three 2-strand equatorial sub-pairs $(\mathcal{R}_1,\epsilon)$, $(\mathcal{R}_2,\epsilon)$ and $(\mathcal{R}_3,\epsilon)$ where $R_1$ contains $\alpha$ and $\beta$, $R_2$ contains $\beta$ and $\gamma$, and $R_3$ contains $\gamma$ and $\alpha$. Let $\frac{p_1}{q_1}$, $\frac{p_2}{q_2}$ and $\frac{p_3}{q_3}$ be the fractions corresponding to the projections of $(\mathcal{R}_1,\epsilon)$, $(\mathcal{R}_2,\epsilon)$ and $(\mathcal{R}_3,\epsilon)$ respectively. In Definition \ref{projection}, we are free to choose how the points of $R \cap \partial(B_R)$ are mapped to the unit circle in the xz-plane. In particular, we are free to twist pairs of points of $\partial R$ that get mapped to the southern hemisphere of the unit sphere. In Conway's notation, this twisting corresponds multiplying each fraction $\frac{p_1}{q_1}$, $\frac{p_2}{q_2}$ and $\frac{p_3}{q_3}$ by $\frac{1}{m}$ where $m$ corresponds to the number of twists. Since $m$ can be chosen to be arbitrarily large, we can assume that each of the rational numbers $\frac{p_1}{q_1}$, $\frac{p_2}{q_2}$ and $\frac{p_3}{q_3}$ lie strictly between $-1$ and $1$.

If none of $(\mathcal{R}_1,\epsilon)$, $(\mathcal{R}_2,\epsilon)$ or $(\mathcal{R}_3,\epsilon)$ are the unpair, then we are done.

Suppose $(\mathcal{R}_1,\epsilon)$ is the unpair. By Theorem \ref{schubert}, $(\mathcal{R}_1,\epsilon)$ is a $\frac{1}{n}$ rational tangle. Isotope $\mathcal{R}$ as in Figure \ref{fig:equaaug2.eps}. This isotopy alters $\alpha$ but fixes $\beta$ and $\gamma$. Note that, by choosing $m$ sufficiently large, we can assume $n$ is positive. After this isotopy, $(\mathcal{R},\delta_1)$ is an equatorial pair such that the numerator closure of $(\mathcal{R}_1,\delta_1)$ is the twisted Whitehead double of the unknot.

Since $\epsilon$ is isotopic to $\delta_1$ in $\partial(B_R)-\partial(R_2)$, then $(\mathcal{R}_2,\epsilon) = (\mathcal{R}_2,\delta_1)$. However, $(\mathcal{R}_3,\epsilon) \neq (\mathcal{R}_3,\delta_1)$. Since the projection of $(\mathcal{R}_3, \epsilon)$ is a $\frac{p_3}{q_3}$ tangle and the isotopy in Figure \ref{fig:equaaug2.eps} corresponds to Conway sum of $(\mathcal{R}_3,\epsilon)$ and a $\frac{2}{1}$-tangle, then $(\mathcal{R}_3,\delta_1)$ is a $\frac{p_3+2q_2}{q_2}$ tangle. However, $\frac{p_3}{q_3}$ was assumed to be strictly between $-1$ and $1$. Hence, $\frac{p_3+2q_2}{q_2}$ is strictly between $2$ and $3$ and cannot be of the form $\frac{1}{r}$ for any integer $r$. Thus, by Theorem \ref{schubert}, $(\mathcal{R}_3,\delta_1)$ is not the unpair.

As illustrated in Figure \ref{fig:gammaarc2.eps}, there is an essential simple closed curve in $\partial(B_R)-\partial(R)$ that is disjoint from both $\epsilon$ and $\delta_1$. Thus, $d(\epsilon, \delta_1) \leq 2$. If $(\mathcal{R}_2,\delta_1)$ remains an unpair, repeat this process once more to construct a equatorial pair $(\mathcal{R},\delta)$ such that no 2-strand equatorial sub-pair is the unpair and $d(\epsilon,\delta)\leq 4$.
\end{proof}

\begin{figure}[h]
\centering \scalebox{.6}{\includegraphics{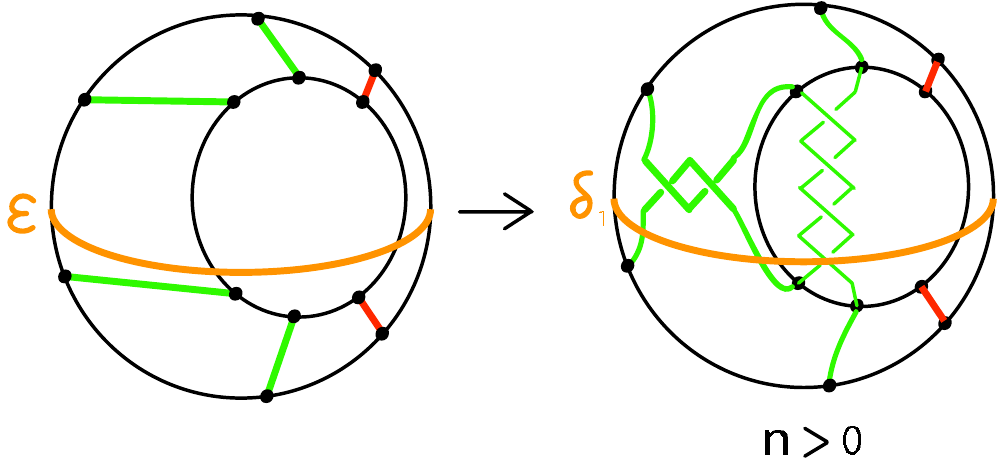}}
\caption{}\label{fig:equaaug2.eps}
\end{figure}

\begin{figure}[h]
\centering \scalebox{.6}{\includegraphics{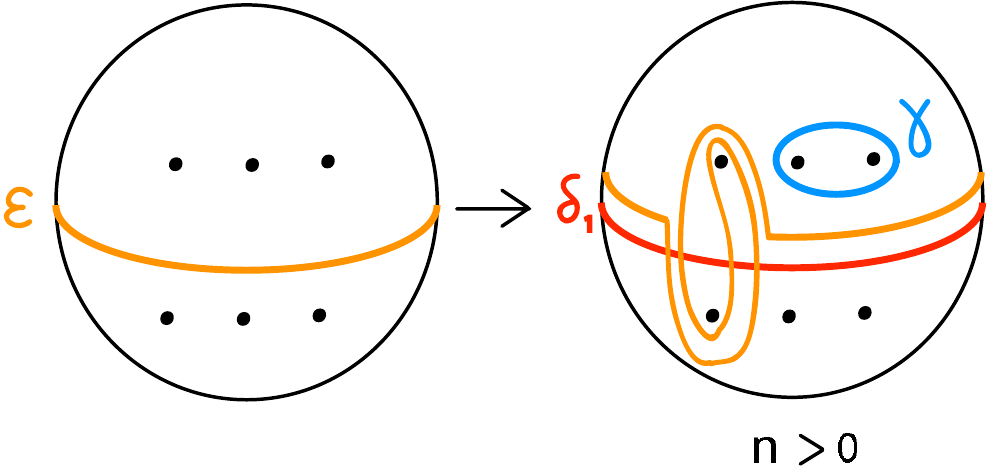}}
\caption{}\label{fig:gammaarc2.eps}
\end{figure}

\section{Essential surfaces in high distance tangles}\label{sec:esssurface}

In this section, we lay the foundations that will eventually allow us to construct a tangle with only one essential surface with large Euler characteristic. We begin with some definitions which build on the definitions introduced in the last section.

\begin{rmk}\label{tangleT}
Let $\mathcal{R}$ and $\mathcal{Q}$ be rational $n$-strand tangles. Let $\mathcal{C}_{n}$ be the curve complex for the $2n$-punctured sphere. Let $\mathcal{V}_R$ be the set of all isotopy classes of essential simple closed curves in $\partial(B_R)-R$ that bound disks in $B_R-R$. Define $\mathcal{V}_Q$ analogously.

Let $\gamma_{R}$ (respectively $\gamma_{Q}$) be an essential curve on $\partial(B_R)$ (respectively $\partial(B_Q)$) such that $\gamma_{R}$ (respectively $\gamma_{Q}$) bounds a $k$-punctured disks $D_{R}$ (respectively $D_{Q}$) in $\partial(B_R)$ (respectively $\partial(B_Q)$). Create a new $(2n-k)$-strand tangle $\mathcal{T}$ by identifying $D_{R}$ and $D_{Q}$ in the disjoint union of $B_R$ and $B_Q$ via a homeomorphism $\psi$ of the $k$-punctured disk. The resulting tangle $\mathcal{T}$ depends on $D_{R}$, $D_{Q}$ and $\psi$. Let $D$ be the image of $D_{R}$ and $D_{Q}$ in $B_T$.
\end{rmk}

\begin{defin}
Let $F$ be a properly embedded surface or a sub-manifold in a 3-manifold $M$ and let $K$ be a collection of properly embedded arcs and knots in $M$. If $\eta(K)$ is a regular open neighborhood of $K$ in $M$, define $F_{K}$ to be $F-\eta(K)$
\end{defin}

\begin{defin}
Given a rational tangle $\mathcal{R}$, let $B$ be a small ball in the interior of $B_R$ disjoint from $R$. Then $B_R-B$ is homeomorphic to $S^2 \times I$. We may isotope $R$ so each strand of $R$ has exactly one critical point with respect to the induced height function. Connect each of these critical points to $\bdd B$ with an unknotted arc $\tau_i$. Let $\Gamma_R=B \cup (\cup_1^k\tau_i)$. Then $\Gamma_R$ is the {\em spine} of $\mathcal{R}$. Note that $(B_R, R)-\Gamma_R$ is homeomorphic to $(S^2 \times I, R^-)$ where $R^-$ is a collection of properly embedded vertical arcs.
\end{defin}

\begin{lemma}\cite[Lemma 2.9]{T3}\label{essinrational}
If $F_{K}$ is a connected incompressible surface in a rational tangle $\mathcal{R}$, then one of the following holds

(1) $F_{K}$ is a sphere bounding a ball,

(2) $F_{K}$ is a twice punctured sphere bounding a ball containing an unknotted arc, or

(3) $F_{K} \cap \partial(B_R) \neq \emptyset$
\end{lemma}

\begin{defin}
A tangle $\mathcal{T}$ is {\em prime} if every embedded 2-punctured sphere in $B_T$ bounds a 3-ball containing an unknotted arc.
\end{defin}

\begin{theorem}\label{disjointfromspine} Let $\mathcal{T}$ be an irreducible, prime $3$-strand tangle as described in Remark \ref{tangleT}. Suppose $F_{K}$ is a properly embedded connected c-incompressible surface in $\mathcal{T}$ with $\partial F$ a possibly empty collection of curves isotopic in $\partial(B_R)-R$ to $\partial D$. If $F_{K}$ can be isotoped to be disjoint from a spine of $\mathcal{R}$ and a spine of $\mathcal{Q}$, then $F_{K}$ is one of the following:

(1) $F_{K}$ is a sphere bounding a ball,

(2) $F_{K}$ is a twice punctured sphere bounding a ball containing a unknotted arc,

(3) $F_{K}$ is isotopic to $\partial(B_R)_{K}-int(D)$,

(4) $F_{K}$ is isotopic to $\partial(B_Q)_{K}-int(D)$,

(5) $F_{K}$ is isotopic to $\partial(B_T)_{K}$,

(6) $F_{K}$ is isotopic to $D_{K}$,

(7) $F_{K}$ is a $\partial(B_T)_{K}$-parallel annulus.
\end{theorem}

\begin{proof}
Let $F^{R}=F \cap B_R$ and $F^{Q}=F \cap B_Q$ be $F \cap B_R$, and let $\Gamma_R$ be the spine of the rational tangle $\mathcal{R}$ such that $F \cap \Gamma_R = \emptyset$. Let $M_R$ be the complement of an open neighborhood of $\Gamma_R$ in $B_R$. Then $M_R$ is homeomorphic to $(S^{2} \times I)$ with $\partial(M_R)= \partial_{+}(M_R) \sqcup \partial_{-}(M_R)$ so that $\partial_{+}(M_R)$ is the boundary of a regular neighborhood of $\Gamma_R$ and $\partial_{-}(M_R)=\partial(B_R)$. Hence, $F^{R}$ is properly embedded in $M_R$, is disjoint from $\partial_{+}(M_R)$ and meets $\partial_{-}(M_R)$ only in $D$. It will be convenient to refer to the height function, $h_R$, on $M_R$ obtained from the natural projection of $S^2 \times I$ onto its $I$ factor.

As $\mathcal{R}$ is a rational tangle, $K \cap M_R$ is isotopic to a collection of arcs, $\{x_1, ... , x_6\}$ that are monotone with respect to $h_R$ and where $x_4$, $x_5$ and $x_6$ are the unique arcs of $K \cap M_R$ that meet $D$. Let $E$ be an embedded vertical rectangle in $M_R$ that is disjoint from $D$ and contains $x_{1}$, $x_{2}$ and $x_{3}$ such that $\partial E$ is the end point union of $x_{1}$, $\gamma$, $x_{3}$ and $\gamma'$ where $\gamma$ is an arc in $\partial_{-}(M_R)$ and $\gamma'$ is the image of $\gamma$ under the natural projection from $M_R$ onto $\partial_{+}(M_R)$. By assumption, any arc of $E \cap F$ is disjoint from both $\gamma$ and $\gamma'$. Suppose $\alpha \in F \cap E$ is any arc or simple closed curve.

\medskip

\noindent \textbf{Claim 1:} If $\alpha$ meets $x_i$ more than once, then we have conclusion (2).

{\em Proof:} If $\alpha$ meets some $x_i$ more than once, then for some $x_j$ there is a bigon $H'$ in $E$ cobounded by a subarc of $x_j$ and a subarc of $\alpha$ such that the interior of $H'$ is disjoint from $\alpha$ and from $x_1$, $x_2$ and $x_3$. $F$ meets the boundary of a closed regular neighborhood of $H'$, $\partial(\eta(H'))$, in a single simple closed curve. Since $H'$ is a disk, $\partial(\eta(H'))$ is a 2-sphere. The closure of each component of $\partial(\eta(H'))-F$ is a disk. One of these disks, together with a 2-punctured subdisk of $F$ cobound a 3-ball containing an unknotted subarc of $K$. As $F_K$ is incompressible we must have conclusion (2).$\square$

\medskip

Suppose that $\alpha$ is a simple closed curve. If $\alpha$ is disjoint from $x_{2}$, it can be eliminated by using an innermost disk argument that appeals to the incompressibility of $F_{K}$ and the irreducibility of $M_R-K$. If $\alpha$ intersects $x_2$, then it must intersect it at least twice and, therefore, , by Claim 1 we have conclusion (2). Hence, we can assume $F\cap E$ contains no simple closed curves.

Suppose that $\alpha \in F\cap E$ is an arc. By Claim 1, one endpoint of $\alpha$ must be in $x_1$ and the other in $x_3$. Furthermore, by Claim 1, we can assume $E \cap F$ consists of arcs that meet each of the strands $x_1$, $x_2$ and $x_3$ in exactly one point each. After an isotopy, we can assume that each curve in $E \cap F$ is level with respect to $h_R$. If $\eta(E)$ is a regular open neighborhood of $E$, then let $N_R=M_R-\eta(E)$. Hence, we can assume that $F^R$ meets $M_R$ outside of $N_R$ in a, possibly empty, collection of level disks each meeting $K$ in three points. By repeating the same argument, we can assume that $F^Q$ meets $M_Q$ outside of $N_Q$ in a, possibly empty, collection of level disks each meeting $K$ in three points.

\begin{figure}[h]
\centering \scalebox{.5}{\includegraphics{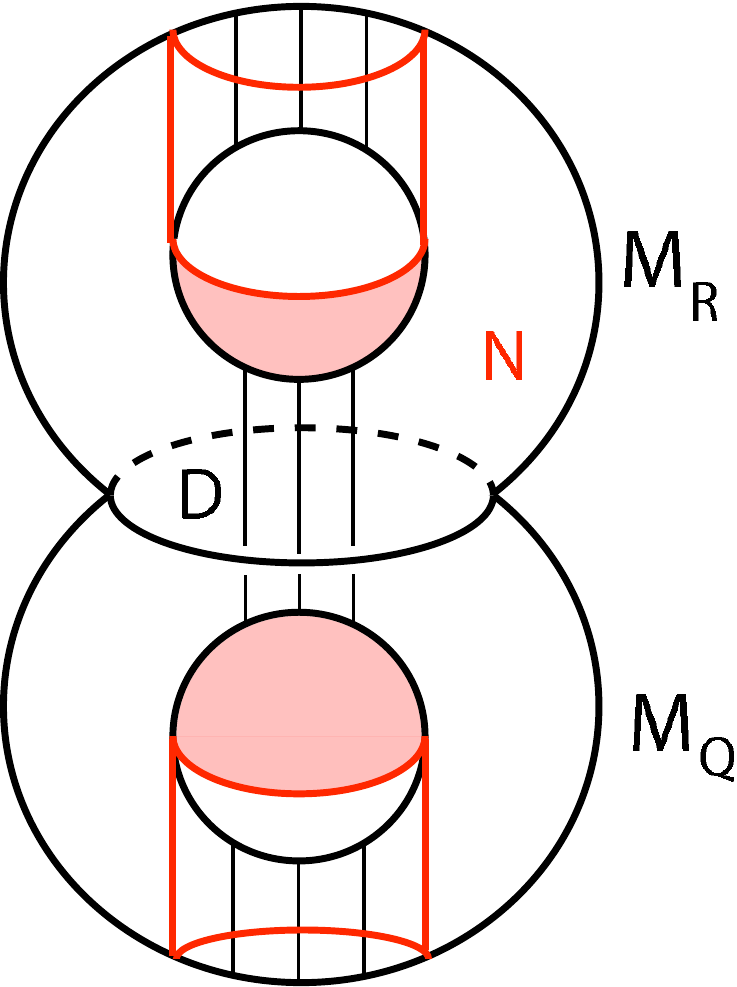}}
\caption{}\label{fig:MRandMQ}
\end{figure}

Let $N$ be the union of $N_R$ and $N_Q$ in $B_T$, see Figure \ref{fig:MRandMQ}. Outside of $N$, $F$ is a collection of $3$-punctured disks that are level with respect to $h_R$ or $h_Q$. Call the boundary curves of these two collections $\mathcal{D}_R$ and $\mathcal{D}_Q$ respectively. Since $F$ meets $\partial(B_T)$ in a collection of curves, $\mathcal{D}$, parallel to $\partial D$, then $F \cap \partial N = \mathcal{D} \cup \mathcal{D}_R \cup \mathcal{D}_Q$. Additionally, any two curves in $\mathcal{D} \cup \mathcal{D}_R \cup \mathcal{D}_Q$ are isotopic in $\partial(N)_K$. $N$ is homeomorphic to $D^2\times I$ where $D^2 \times \{1\} = \partial_{+}(M_R)-\eta(E)$ and $D^2 \times \{0\} = \partial_{+}(M_Q)-\eta(E)$. Hence, $N$ has a natural height function, $h_N$, induced by projection onto the $I$ factor. In particular, $h_N$ can be chosen so that every curve in $\mathcal{D} \cup \mathcal{D}_R \cup \mathcal{D}_Q$ is level and each arc of $K\cap N$ is monotone. Let $H$ be a properly embedded vertical disk in $N$ that contains all three strands of $N \cap K$. Note that $\partial H$ meets every curve in $\mathcal{D} \cup \mathcal{D}_R \cup \mathcal{D}_Q$ in exactly two points and each component of $N-H$ is a $3$-ball disjoint from $K$. See Figure \ref{fig:NandH.eps}.

Suppose $H \cap F = \emptyset$. Since every curve in $F \cap \partial N$ meets $H$, then $F \cap \partial N = \emptyset$ and $F$ is contained in $N-H$. Hence, we have conclusion (1).

Suppose $H \cap F \neq \emptyset$. If any component of $F \cap H$ is a closed curve disjoint from $K$, then, by the incompressibility of $F_K$ and the irreducibility of $N-\eta(K)$, we can remove it via an isotopy of $F_K$ supported in $N_K$. If any component of $F \cap H$ is a closed curve not disjoint from $K$, then, by appealing to the argument in Claim 1, we have conclusion (2). Hence, every component of $F \cap H$ is an arc. Label the endpoints of an outermost such arc as in Figure \ref{fig:NandH.eps}, where the $a^{\pm}_{i}$'s lie on $\mathcal{D}_R$, the $b^{\pm}_{i}$'s lie on $\mathcal{D}_Q$, and the $c^{\pm}_{i}$'s lie on $\mathcal{D}$. There is an outermost arc in $F \cap H$ with one of the following endpoint labels:

\begin{figure}[h]
\centering \scalebox{.5}{\includegraphics{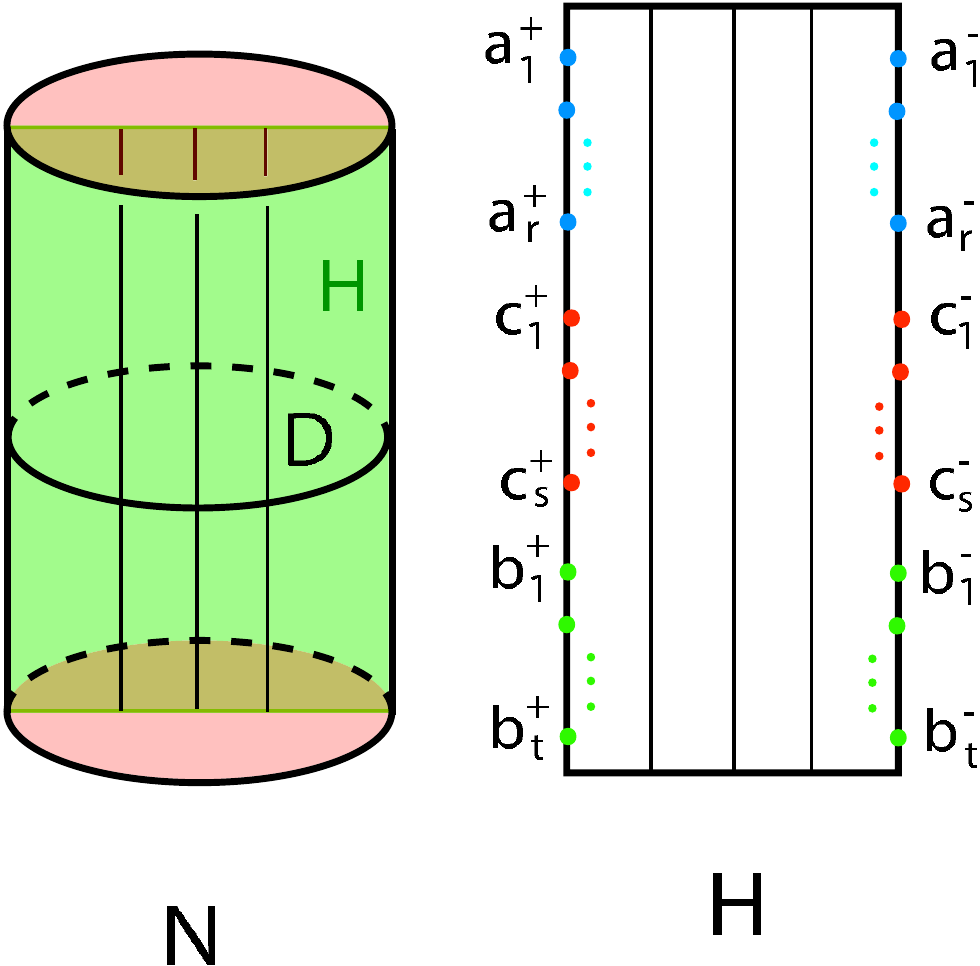}}
\caption{}\label{fig:NandH.eps}
\end{figure}

(1) $a^{-}_{1}$ and $a^{+}_{1}$ (similarly, $b^{-}_{t}$ and $b^{+}_{t}$)

(2) $a^{-}_{i}$ and $a^{-}_{i+1}$ (similarly, $b^{-}_{i}$ and $b^{-}_{i+1}$, $a^{+}_{i}$ and $a^{+}_{i+1}$, or $b^{+}_{i}$ and $b^{+}_{i+1}$)

(3) $c^{-}_{i}$ and $c^{-}_{i+1}$ (similarly, $c^{+}_{i}$ and $c^{+}_{i+1}$)

(4) $a^{+}_{r}$ and $c^{+}_{1}$ (similarly, $a^{-}_{r}$ and $c^{-}_{1}$, $b^{+}_{1}$ and $c^{+}_{s}$, or $b^{-}_{1}$ and $c^{-}_{s}$)

(5) $c^{-}_{1}$ and $c^{+}_{1}$ (similarly, $c^{-}_{s}$ and $c^{+}_{s}$)

(6) $a^{+}_{r}$ and $b^{+}_{1}$ (similarly, $a^{-}_{r}$ and $b^{-}_{1}$)

Let $y_1$, $y_2$ and $y_3$ be the three strands of $K$ in $N$. Let $\alpha$ be an outermost arc of $F \cap H$ in $H$. If $\alpha$ meets one of $y_1$, $y_2$ and $y_3$ in more than one point, then as in Claim 1, we have conclusion (2). Hence, we can assume that $\alpha$ meets each of $y_1$, $y_2$ and $y_3$ in at most one point.

\medskip

\noindent{\bf Case 1:} Suppose $\alpha$ is an outer most arc of $F \cap H$ with endpoints $a^{-}_{1}$ and $a^{+}_{1}$. Since $\alpha$ meets each of $y_1$, $y_2$ and $y_3$ in at most one point, $\alpha \cap K$ consists of exactly three points. Let $L$ be the disk in $M_R$ that contains $a^{-}_{1}$ and $a^{+}_{1}$. The disk $L$ together with a neighborhood of $\alpha$ in $F_R$ is a $6$-punctured annulus in $M_R$. Both boundary components of this annulus are contained in the interior of $M_R$ and bound disks in $M_R$ disjoint from both $F$ and $K$. See Figure \ref{fig:Spine1.eps}. By incompressibility of $F_{K}$ and irreducibility of $M_R-K$, both boundary components of this annulus bound disks in $F_{K}$. Hence, $F_R$ is a $6$-punctured sphere in $M_R$ isotopic to $\partial(B_R)_{K}$. This is a contradiction to $F_{K}$ being incompressible, so such an outermost arc must not exist.

\begin{figure}[h]
\centering \scalebox{.5}{\includegraphics{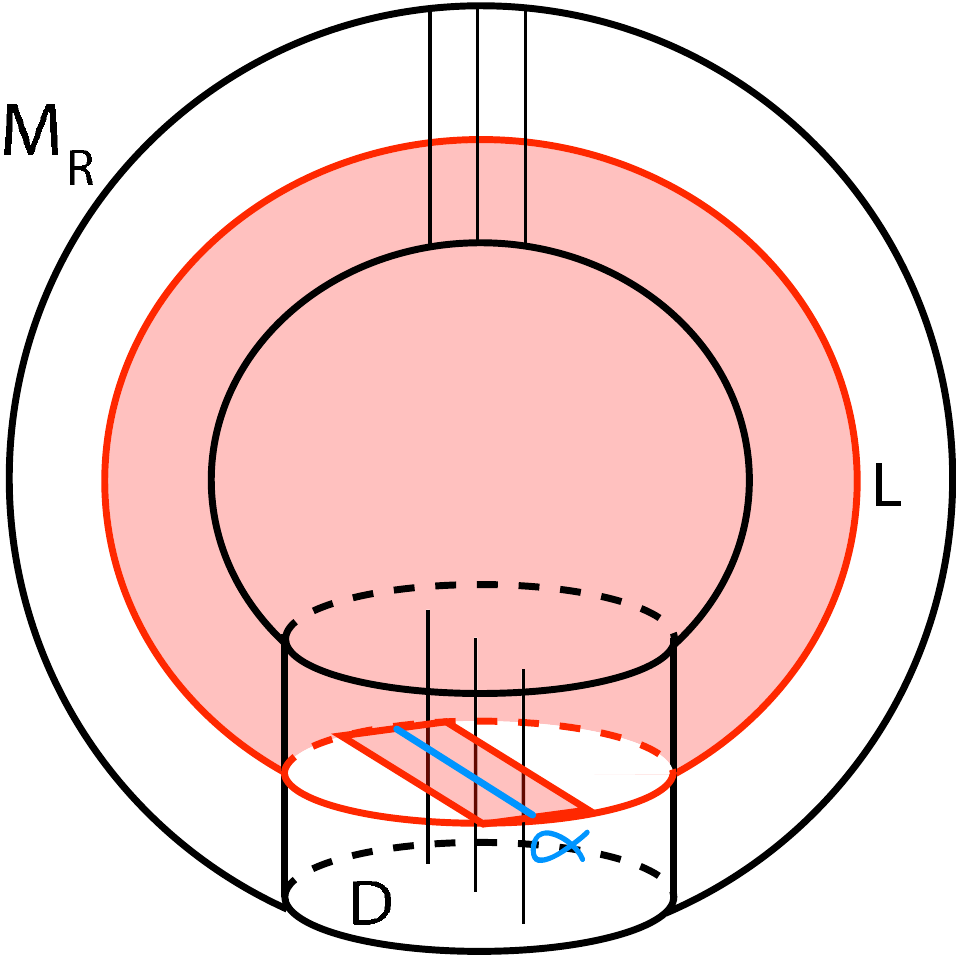}}
\caption{}\label{fig:Spine1.eps}
\end{figure}

\medskip

\noindent{\bf Case 2:} Suppose $\alpha$ is an outermost arc of $F \cap H$ with endpoints $a^{-}_{i}$ and $a^{-}_{i+1}$. Since $\alpha$ meets each of $y_1$, $y_2$ and $y_3$ in at most one point, $\alpha \cap K = \emptyset$. Let $L_{1}$ and $L_{2}$ be the two $3$-punctured disks in $\mathcal{D}_R$ that contain $a^{-}_{i}$ and $a^{-}_{i+1}$ in there respective boundaries. Since $\alpha$ is outermost, we can isotope it to be monotone with respect $h_N$. Let $x$ be one of the strands of $K$ in $M_R-N$. As in Figure \ref{fig:Spine2.eps}, there is a disk $G$ in $M_R-R$ that is vertical with respect to $h_R$ and illustrates a parallelism between a sub arc of $x_i$, $i \in \{1,2,3\}$, and an arc in $F$. As in the proof of Claim 1, we have conclusion (2).

\begin{figure}[h]
\centering \scalebox{.5}{\includegraphics{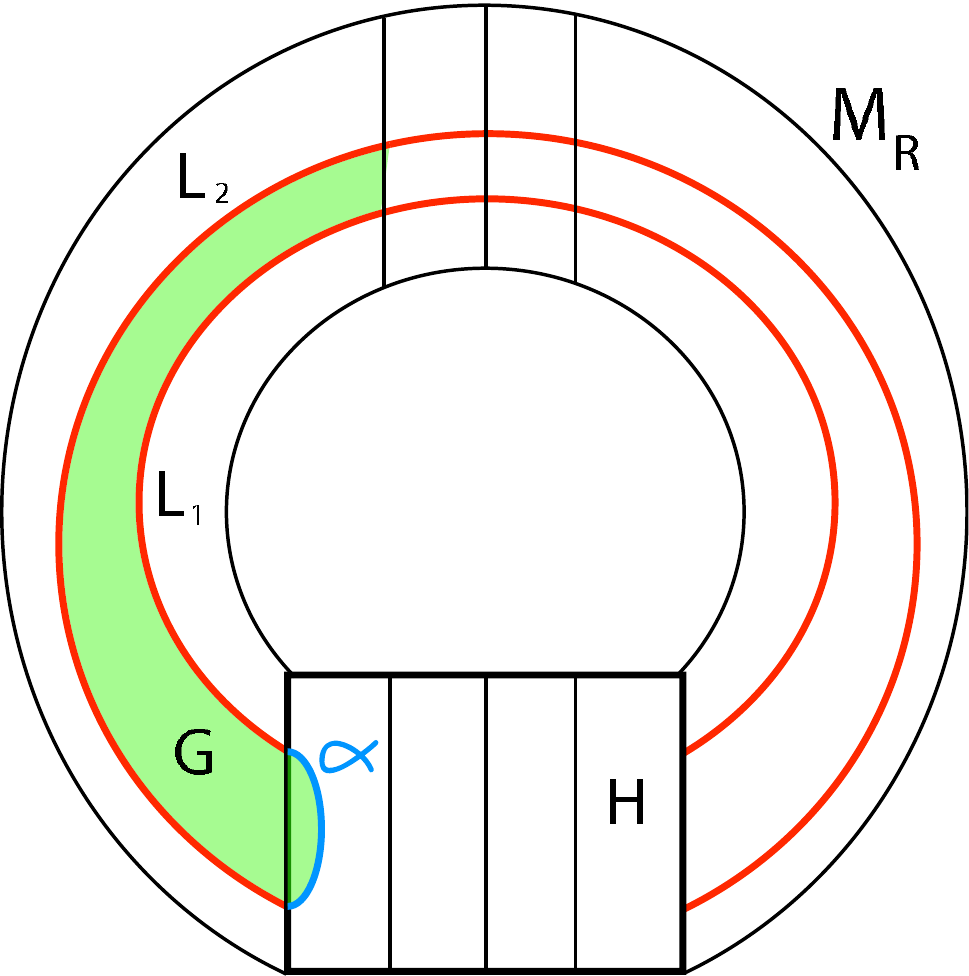}}
\caption{}\label{fig:Spine2.eps}
\end{figure}

\medskip

\noindent{\bf Case 3:} Suppose $\alpha$ is an outermost arc of $F \cap H$ with endpoints $c^{-}_{i}$ and $c^{-}_{i+1}$. Since $\alpha$ meets each of $y_1$, $y_2$ and $y_3$ in at most one point, $\alpha \cap K = \emptyset$. Let $\gamma_{1}$ and $\gamma_{2}$ be the curves in $\mathcal{D}$ that contain $c^{-}_{i}$ and $c^{-}_{i+1}$ respectively. Let $C_{1}$ and $C_{2}$ be small closed neighborhoods of $\gamma_{1}$ and $\gamma_{2}$ in $F$ respectively. A neighborhood of $\alpha$ in $F_{K}$ together with $C_{1}$ and $C_{2}$ form a connected subsurface of $F_{K}$ with three boundary components $\gamma_{1}$, $\gamma_{2}$ and $\beta$. The arc $\beta$ bounds a disk in $N$ disjoint from $F$ and $K$. By incompressibility of $F_{K}$ and irreducibility of $N-K$, $\beta$ bounds a disk in $F_{K}$. Hence, we have conclusion (7).

\medskip

\noindent{\bf Case 4:} Suppose $\alpha$ is an outermost arc of $F \cap H$ with endpoints $a^{+}_{r}$ and $c^{+}_{1}$. Since $\alpha$ meets each of $y_1$, $y_2$ and $y_3$ in at most one point, $\alpha \cap K = \emptyset$. Let $L$ be the disk in $F$ with boundary in $\mathcal{D}_R$ that contains $a^{+}_{r}$. Let $\gamma$ be the curve in $\mathcal{D}$ that contains $c^{-}_{i}$ and let $C_{1}$ be a small closed neighborhood of $\gamma$ in $F$. A neighborhood of $\alpha$ in $F$ together with $E$ and $C_{1}$ forms a $3$-punctured annulus subsurface of $F_{K}$ with boundary components $\gamma$ and $\beta$. The curve $\beta$ bounds a disk in $N$ disjoint from $F$ and $K$. By incompressibility of $F_{K}$ and irreducibility of $N-K$, $\beta$ bounds a disk in $F_{K}$. Hence, we have conclusion (3). By a similar argument, if $\alpha$ has endpoints $b^{+}_{1}$ and $c^{+}_{s}$ or $b^{-}_{1}$ and $c^{-}_{s}$ we have conclusion (4).

\medskip

\noindent{\bf Case 5:} This case can only occur if there are no $a^{\pm}_{i}$'s (i.e., $\mathcal{D}_R$ is empty). Suppose $\alpha$ is an outermost arc of $F \cap H$ with endpoints $c^{-}_{1}$ and $c^{+}_{1}$. Since $\alpha$ meets each of $y_1$, $y_2$ and $y_3$ in at most one point, $\alpha \cap K$ consists of exactly three points. Let $\gamma$ be the curve in $\mathcal{D}$ that contains both $c^{-}_{1}$ and $c^{+}_{1}$. Let $C$ be a small closed neighborhood of $\gamma$ in $F$. $C$ together with a regular neighborhood of $\alpha$ in $F$ is a $3$-punctured subsurface of $F_{K}$ with three boundary components $\gamma$, $\beta_{1}$ and $\beta_{2}$. However, $\beta_{1}$ and $\beta_{2}$ each bound disks in $N-K$ disjoint from $F$ and $K$. By incompressibility of $F_{K}$ and irreducibility of $N-K$, both $\beta_{1}$ and $\beta_{2}$ bound disks in $F_{K}$. Hence, we have conclusion (6).

\medskip

\noindent{\bf Case 6:} This case can only occur if there are no $c^{\pm}_{i}$'s (i.e., $\mathcal{D}$ is empty). Suppose $\alpha$ is an outermost arc of $F \cap H$ with endpoints $a^{+}_{r}$ and $b^{+}_{1}$. Since $\alpha$ meets each of $y_1$, $y_2$ and $y_3$ in at most one point, $\alpha \cap K = \emptyset$. Let $L_R$ be the disk in $F$ with boundary in $\mathcal{D}_R$ such that $a^{+}_{r} \in \partial(L_R)$ and let $L_Q$ be the disk in $F$ with boundary in $\mathcal{D}_Q$ such that $b^{+}_{1}\in \partial(L_Q)$. A neighborhood of $\alpha$ in $F$ together with $L_R$ and $L_Q$ is a $6$-punctured disk subsurface of $F_{K}$. The boundary of this disk bounds a disk in $N-K$ disjoint from $F$ and $K$. By incompressibility of $F_{K}$ and irreducibility of $N-K$, the boundary of this $6$-punctured disk bounds a disk in $F_{K}$. Hence, we have conclusion (5).
\end{proof}

\begin{lemma}\label{compressibleinR}
If $D_{K}$ is compressible in $B_R-R$ then $d(\partial D,\mathcal{V}_R) \leq 1$. If $D_{K}$ is compressible in $B_Q-Q$ then $d(\partial D,\mathcal{V}_Q) \leq 1$
\end{lemma}

\begin{proof}
Let $E \subseteq B_R-R$ be a compressing disk for $D_{R}$. Since $\partial E$ is essential in $D_{R}$ and disjoint from $\partial(D_R)$, then $\partial E \epsilon \mathcal{V}_R$ and $d(\partial D, \mathcal{V}_R) \leq 1$.
\end{proof}

\begin{lemma}\label{compressible}

If $D_K$ is compressible in $B_T-T$ then $d(\partial D,\mathcal{V}_R) \leq 1$ or $d(\partial D,\mathcal{V}_Q) \leq 1$.

\end{lemma}

\begin{proof}
If $D_K$ is compressible in $B_T-T$, then $D_{K}$ is compressible in $B_R-R$ or $D_{K}$ is compressible in $B_Q-Q$. By Lemma \ref{compressibleinR}, $d(\partial D,\mathcal{V}_R) \leq 1$ or $d(\partial D,\mathcal{V}_Q) \leq 1$.
\end{proof}

\begin{lemma}\cite[Proposition 4.1]{T3}\label{ccompressibleinR}
If $D_{R}$ is cut-compressible in $B_R-R$ then $d(\partial D,\mathcal{V}_R) \leq 2$. Similarly, if $D_{Q}$ is cut-compressible in $B_Q-Q$ then $d(\partial D,\mathcal{V}_Q) \leq 2$\end{lemma}

\begin{proof}
Let $\alpha$ be the arc of $K$ in $B_R$ that punctures the cut-compressing disk, $C$, for $D$. Let $B$ be the disk in $B_R$ illustrating the boundary parallelism of $\alpha$. After perhaps an isotopy of $B$, $B \cap C$ is a single arc $\beta$ that separates $B$ into two disks $B_{1}$ and $B_{2}$. Consider a regular neighborhood of $C \cup B_{1}$. Its boundary contains a disk that intersects $\partial(B_R)_{K}$ in an essential curve $\gamma$ and does not intersect $\partial C$. Hence,  $d(\partial D,\mathcal{V}_R) \leq d(\partial D,\partial C) +d(\partial C,\gamma) \leq 2$.
\end{proof}

\begin{lemma}\label{ccompressible}

If $D_K$ is c-compressible in $B_T-T$ then $d(\partial D ,\mathcal{V}_R) \leq 2$ or $d(\partial D,\mathcal{V}_Q) \leq 2$.

\end{lemma}

\begin{proof}
If $D_K$ is compressible in $B_T-T$ then, by Lemma \ref{compressible}, $d(\partial D,\mathcal{V}_R) \leq 1$ or $d(\partial D,\mathcal{V}_Q) \leq 1$. If $D_K$ is cut-compressible, then $D_K$ is cut-compressible in $B_R-R$ or $D_{K}$ is cut-compressible in $B_Q-Q$. By Lemma \ref{ccompressible}, $d(\partial D,\mathcal{V}_R) \leq 2$ or $d(\partial D,\mathcal{V}_Q) \leq 2$
\end{proof}

\begin{lemma}\label{reducible}
If $B_T-T$ is reducible then $d(\partial D,\mathcal{V}_R) \leq 1$ or $d(\partial D,\mathcal{V}_Q) \leq 1$.
\end{lemma}

\begin{proof}
Let $S$ be a reducing sphere for $B_T-T$. Since $B_R-R$ and $B_Q-Q$ are irreducible, $S$ cannot be isotoped to be disjoint from $D$. Isotope $S$ so that $|S \cap D|$ is minimal. If an innermost curve of $D \cap S$ in $S$ is essential in $D_K$, then $D_K$ is compressible and, by Lemma \ref{compressible} $d(\partial D,\mathcal{V}_R) \leq 1$ or $d(\partial D,\mathcal{V}_Q) \leq 1$.

Suppose $\alpha$ is a curve in $D \cap S$ that is innermost in $S$ and bounds a subdisk $D'$ in $D$. Since $\alpha$ is innermost in $S$, $\alpha$ bounds a subdisk $S'$ of $S$ that is disjoint from $D$ except in its boundary. After pushing $D'$ slightly off of $D$ toward $S'$, $D' \cup S'$ is a $2$-sphere embedded in either $B_R-R$ or $B_Q-Q$. Since both $B_R-R$ and $B_Q-Q$ are irreducible, $D' \cup S'$ bounds a $3$-ball disjoint from $K$. We can use this $3$-ball to construct an isotopy of $S$ that reduces $|S \cap D|$. However, this contradicts the assumption that $|S \cap D|$ is minimal. Hence, $\alpha$ must be essential in $D_K$.
\end{proof}

\begin{lemma}\label{prime}
 If $B_T-T$ contains an essential $2$-punctured sphere then $d(\partial D,\mathcal{V}_R) \leq 2$ or $d(\partial D,\mathcal{V}_Q) \leq 2$.
\end{lemma}

\begin{proof}
Let $S$ be an essential $2$-punctured sphere in $B_T-T$. Since $B_R-R$ and $B_Q-Q$ are prime, $S$ cannot be isotoped to be disjoint from $D$. Isotope $S$ so that $|S \cap D|$ is minimal. If an innermost curve of $D \cap S$ in $S$ bounds a disk in $S$ and is essential in $D_K$, then $D_K$ is compressible and, by Lemma \ref{compressible}, $d(\partial D,\mathcal{V}_R) \leq 1$ or $d(\partial D ,\mathcal{V}_Q) \leq 1$.

If an innermost curve of $D \cap S$ in $S$ bounds a cut-disk in $S$ and is essential in $D_K$, then $D_K$ is cut-compressible. By Lemma \ref{ccompressible}, $d(\partial D,\mathcal{V}_R) \leq 2$ or $d(\partial D,\mathcal{V}_Q) \leq 2$.

Suppose $\alpha$ is an innermost curve of $D \cap S$ in $S$ that bounds a c-disk $S'$ in $S$ and is inessential in $D_K$. Since $\alpha$ is inessential in $D_K$, $\alpha$ bounds a c-disk $D'$ in $D$. After pushing $D'$ slightly off of $D$ toward $S'$, $D' \cup S'$ is a sphere or a $2$-punctured sphere embedded in either $B_R-R$ or $B_Q-Q$. Since both $B_R-R$ and $B_Q-Q$ are irreducible and prime, $D' \cup S'$ bounds a $3$-ball or a $3$-ball containing an unknotted arc of $K$. We can use this $3$-ball to construct an isotopy of $S_{K}$ that reduces $|S \cap D|$. However, this contradicts the assumption that $|S \cap D|$ is minimal. Hence, no such $\alpha$ exists.
\end{proof}

\begin{lemma}\label{S4insideincomp} If $\partial(B_T)-T$ is compressible in $B_T-T$, then $d(\partial D,\mathcal{V}_R) \leq 2$ or $d(\partial D,\mathcal{V}_Q) \leq 2$.\end{lemma}

\begin{proof}
To form a contradiction, assume $\Delta$ is a compressing disk for $\partial(B_T)-T$ in $B_T-T$. Isotope $\Delta$ so that $|\Delta \cap D_K|$ is minimal. By Lemma \ref{compressible}, we can assume that $D_K$ is incompressible. By Lemma \ref{reducible}, we can assume that $B_T-T$ is irreducible. Since $D_K$ is incompressible and $B_T-T$ is irreducible we can assume $\Delta \cap D_K$ consists only of arcs and no simple closed curves. Let $\alpha$ be an outermost arc of $\Delta \cap D_K$ in $\Delta$ and let $F$ be the subdisk of $\Delta$ that $\alpha$ cobounds with an arc in $\partial \Delta$ such that the interior of $F$ is disjoint from $D_K$. Without loss of generality, $F$ is properly embedded in $B_R$. If $\partial F$ is inessential in $\partial(B_R)-R$, then $F$ is boundary parallel in $B_R-R$ and this boundary parallelism can be used to construct an isotopy of $\Delta$ that decreases $|\Delta \cap D_K|$, a contradiction. Hence, $F$ is a compressing disk for $B_R-R$ that intersects $\partial D$ in exactly two points. It follows that $d(\partial D,\mathcal{V}_R) \leq 2$.
\end{proof}

\begin{lemma}\label{essinboth}
If $F_K$ is a c-incompressible surface in $B_T-T$ with $\partial F$ consisting of a (possibly empty) collection of simple closed curves isotopic to $\partial D$ in $\partial(B_T)-T$, $d(\partial D,\mathcal{V}_R) \geq 3$ and $d(\partial D,\mathcal{V}_Q) \geq 3$, then, after an isotopy of $F_K$, $F \cap D$ is a (possibly empty) collection of simple closed curves all of which are c-essential in both $F_K$ and $D_K$.
\end{lemma}

\begin{proof}
Since all components of $\partial F$ are isotopic to $\partial D$ in $\partial(B_T)-T$, there is an isotopy of $F$ supported in a small neighborhood of $\partial(B_T)$ in $B_T$ resulting in $\partial F \cap \partial D = \emptyset$. Suppose the interior of $F_K$ has been isotoped to minimize $|F \cap D|$ and suppose $\alpha$ is a curve in $D \cap F$ which is c-inessential in $D_K$. By appealing to an innermost such $\alpha$, we can assume that the c-disk, $D'$, that $\alpha$ bounds in $D$ is disjoint from $F$ except in its boundary. Since $F_K$ is c-incompressible, $\alpha$ bounds a c-disk, $F'$, in $F$. By Lemma \ref{reducible} and Lemma \ref{prime}, $D' \cup F'$ cobound a 3-ball or a 3-ball containing an unknotted arc that gives rise to an isotopy which reduces the number of components of $D \cap F$. Hence, if $\alpha$ is c-inessential in $D_K$, then we contradict the minimality of $|D \cap F|$. By Lemma \ref{ccompressible}, $D_K$ is c-incompressible so a similar argument implies that $\alpha$ is also c-essential in $F_K$.
\end{proof}

\begin{theorem}\cite[Proposition 4.3]{T3}\label{tomova}
Suppose $(F,\partial F)\subset(B_R,\partial(B_R))$ is a properly embedded surface transverse to $K$ that satisfies all of the following conditions
\begin{enumerate}
\item $F_{K}$ has no disk components,

\item $F_{K}$ is c-incompressible,

\item $F_{K}$ intersects every spine $\Gamma_{R}$ of $\mathcal{B}_R$,

\item all curves of $F \cap \partial(B_R)$ are essential on $\partial(B_R)-R$

\end{enumerate}
Then there is at least one curve $f \in F \cap \partial(B_R)$ that is essential in $\partial(B_R)-R$ such that $d(\mathcal{V}_R,f) \leq 1 - \chi(F_{K})$ and every $g \in F \cap \partial(B_R)$ that is essential on $\partial(B_R)-R$ for which the inequality does not hold lies in the boundary of a $(\partial(B_R)-R)$-parallel annulus component of $F_{K}$.
\end{theorem}

\begin{theorem}\label{thm:Tess}
Let $\mathcal{T}$ be a tangle as described in Remark \ref{tangleT}. In addition, choose $D_{R}$ and $D_{Q}$ such that $d(\partial(D_{R}),\mathcal{V}_R) \geq 3$ and $d(\partial(D_{Q}),\mathcal{V}_Q) \geq 3$. If $F_{K}$ is a properly embedded connected c-incompressible surface in $B_T-T$ with $\partial F$ a (possibly empty) collection of curves isotopic to $\partial D$ in $\partial(B_T)-K$, then one of the following holds

\begin{enumerate}
\item $F_{K}$ is a sphere bounding a ball,

\item $F_{K}$ is a twice punctured sphere bounding a ball containing a unknotted arc,

\item $F_{K}$ is isotopic to $\partial(B_R)_{K}-int(D)$,

\item $F_{K}$ is isotopic to $\partial(B_Q)_{K}-int(D)$,

\item $F_{K}$ is isotopic to $\partial(B_T)_{K}$,

\item $F_{K}$ is isotopic to $D_{K}$,

\item $F_{K}$ is a $\partial(B_T)_{K}$-parallel annulus,

\item $d(\partial D,\mathcal{V}_R) \leq 2 - \chi(F_{K})$,

\item $d(\partial D,\mathcal{V}_Q) \leq 2 - \chi(F_{K})$.
\end{enumerate}
\end{theorem}

\begin{proof}
\noindent{\bf Case 1:} Suppose $F_{K}$ can be isotoped to be disjoint from $D$. Then $F$ is properly embedded in one of $B_R$ or $B_Q$ with boundary (if non-empty) isotopic to parallel copies of $\partial D$. Without loss of generality, assume $F$ is contained in $B_R$. Note that since every component of $\partial F$ is essential in $B_R-R$, no component of $F_K$ is a boundary parallel disk in $B_R-R$. If $F_{K}$ is a disk, then $d(\partial D,\mathcal{V}_R) = 0$ contradicting the hypothesis of the theorem. If $\partial F$ is empty, then conclusions (1) or (2) hold, by Lemma \ref{essinrational}. Hence, we can assume that $F_{K}$ is a c-incompressible properly embedded surface with no disk components and non-trivial boundary. If, in addition, $F_{K}$ intersects every spine $\Gamma_R$, then the hypotheses of Theorem \ref{tomova} are satisfied and there is at least one curve $f \in F_{K} \cap (\partial(B_R)-R)$ that is essential in $\partial(B_R)-R$ such that $d(\mathcal{V}_R,f) \leq 1 - \chi(F_{K})$. Since all curves in $F_{K} \cap \partial(B_R)-R$ are parallel to $\partial D$, then conclusion (9) holds. If $F_{K}$ is disjoint from some spine $\Gamma_R$, then the hypotheses of Theorem \ref{disjointfromspine} are satisfied and one of conclusions (1) to (7) holds.

\medskip

\noindent{\bf Case 2:} Suppose $F_{K}$ cannot be isotoped to be disjoint from $D$.
If $F_{K}$ can be isotoped to be disjoint from some spine $\Gamma_R$ and some spine of $\Gamma_Q$ then, by Theorem \ref{disjointfromspine}, one of conclusions (1) to (7) holds.

Suppose $F_{K}$ cannot be isotoped to be disjoint from any spine of $\mathcal{R}$. (The case where $F_{K}$ cannot be isotoped to be disjoint from any spine of $\mathcal{Q}$ is proven analogously.) In particular, $F^{R}_K= F_K \cap B_R$ can not be isotoped to be disjoint from any spine of $\mathcal{R}$. To apply Theorem \ref{tomova} to $F^{R}_K$ we need to verify the remaining three hypotheses.

(1) By Lemma \ref{ccompressible}, $d(\partial(D_{R}),\mathcal{V}_R) \geq 3$ and $d(\partial(D_{Q}),\mathcal{V}_Q) \geq 3$ imply that $D_K$ is c-incompressible. By Lemma \ref{reducible}, $d(\partial(D_{R}),\mathcal{V}_R) \geq 3$ and $d(\partial(D_{Q}),\mathcal{V}_Q) \geq 3$ imply that $B_T-T$ is irreducible. By c-incompressibility of $D$ and irreducibility of $B_T-T$ we can assume that $F_K-D$ contains no disk components. Hence, $F^{R}_{K}$ contains no disk components.

(2) $F_K$ is assumed to be c-incompressible in $B_T-T$. Hence, $F^{R}_K$ is c-incompressible in $B_R-R$.

(3) We have assumed that $F_K$ intersects every spine of $\mathcal{R}$.

(4) By Lemma \ref{essinboth}, $D_K$ and $F_K$ can be isotoped to intersect in a non-empty collection of closed curves that are essential in each surface. In particular, every component of $F^{R}_{K} \cap D_K$ is essential in $D_K$.

Since the hypotheses for Theorem \ref{tomova} are satisfied, there exists a curve $f \in F^{R}_{K} \cap D_K$ that is essential on $\partial(B_R)-R$ and such that $d(\mathcal{V}_R,f) \leq 1 - \chi(F^{R}_K)$. Since both $F^{R}_K$ and $F^{Q}_K$ are planar surfaces containing no disk or sphere components, $\chi(F^{R}_K) \leq 0$ and $\chi(F^{Q}_K) \leq 0$. As $\chi(F_{K})=\chi(F^{R}_{K})+\chi(F^{Q}_{K})$, $\chi(F^{R}_K) \leq 0$ and $\chi(F^{Q}_K) \leq 0$, it follows that $1-\chi(F_{K}) \geq 1-\chi(F^{R}_K)$ and $d(\mathcal{V}_R,f) \leq 1 - \chi(F_{K})$. Since $d(f,\partial D) \leq 1$, we conclude that $d(\partial D,\mathcal{V}_R) \leq 2 - \chi(F_{K})$.

\end{proof}

\section{Constructing the example}\label{sec:construction}

\begin{figure}
\begin{center} \includegraphics[scale=.4]{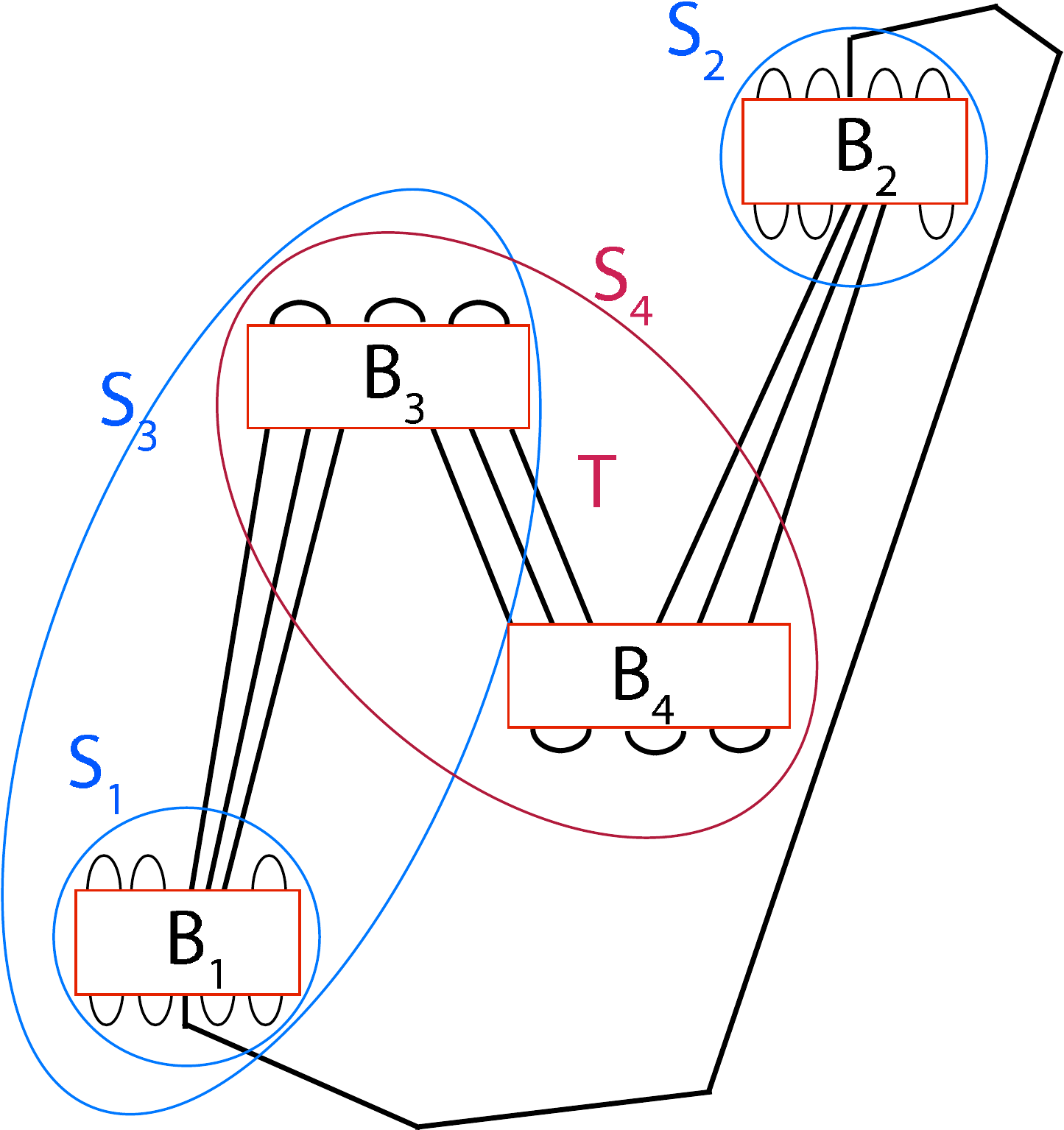}
\end{center}
\caption{} \label{fig:counterexample} \end{figure}

We will now construct a knot $K$ as in Figure \ref{fig:counterexample}. We begin with the schematic in that figure and we substitute each of the balls $B_1,...,B_4$ with tangles satisfying particular properties. In the schematic $S_1,..,S_4$ are spheres, $B_1$ and $B_2$ are the disjoint balls bounded by $S_1$ and $S_2$ containing tangles $\mathcal{T}_1=(B_1,T_1)$ and $\mathcal{T}_2=(B_2,T_2)$. The unique strand of $K$ that connects $S_1$ and $S_2$ but is disjoint from $S_4$ will be labeled $\epsilon$. The sphere $S_4$ bounds a ball $B_T$ disjoint from $B_1$ and $B_2$ containing a 3-strand tangle $\mathcal{T}=(B_T, T)$. Thus the knot $K$ is naturally decomposed into three tangles - the two 2-strand tangles $\mathcal{T}_1$ and $\mathcal{T}_2$ and the 3-strand tangle $\mathcal{T}$.

	We will use the basic construction in Definition \ref{def:BuildingTangles} to construct each of $\mathcal{T}_1$, $\mathcal{T}_2$ and $\mathcal{T}$. To construct $(B_1,T_1)$, begin with two rational tangles $\mathcal{R}$ and $\mathcal{Q}$ with $5$ and $6$ strands respectively. Choose $D_R$ to be any 9-punctured disk in $\partial(B_R)$ and $D_Q$ to be a 9-punctured disk in $\partial(B_Q)$ so that exactly three strands of $Q$ have both of their endpoints in $D_Q$. The resulting 2-strand tangle has a natural bridge surface $\Sigma_1$ which is parallel to $\bdd B_R$. By choosing $\psi$ appropriately, we may simultaneously assume that $d(\Sigma_1, T_1)$ is arbitrarily high and $\mathcal{T}_1$ does not have any closed components. In our construction, we will require that $d(\Sigma_1, T_1)\geq 25$. The tangle $(B_2, T_2)$ will be identical to the tangle $(B_1, T_1)$

The tangle $\mathcal{T}$ can be similarly constructed taking each of $\mathcal{R}$ and $\mathcal{Q}$ to be a three-strand rational tangle. Let $\mathcal{V}_R$ (respectively $\mathcal{V}_Q$) be the set of all essential simple closed curves in $\partial(B_R)$ (respectively $\partial(B_Q)$) that bound disks in $B_R-K$ (respectively $B_Q-K$). We require that the equatorial pairs $(\mathcal{R},\partial D)$ and $(\mathcal{Q},\partial D)$ have the property that $\partial D$ is far from $\mathcal{V}_R$ in the curve complex of the 6-punctured sphere. Similarly, we require that $\partial D$ is far from $\mathcal{V}_Q$ in the curve complex of the 6-punctured sphere. Specifically, we will require that $d(\partial D,\mathcal{V}_R)\geq 25$ and $d(\partial D,\mathcal{V}_Q)\geq 25$. By Theorem \ref{inducedknotting}, we may simultaneously require that no 2-strand equatorial sub-pair of $(\mathcal{R},\partial D)$ or $(\mathcal{Q},\partial D)$ is the unpair. Furthermore, we may assume that $\mathcal{T}$ is symmetric with respect to the disk $D$.

It is important to note that there is sufficient latitude in the choice of the tangles $\mathcal{T}_1$ and $\mathcal{T}_2$ so that we can assume $K$ is a knot.

\section{Properties of $K$}\label{sec:properties}

In this section, we establish some of the properties of $K$. The first property is based on an easy computation for the schematic in Figure \ref{fig:counterexample}.

\begin{pro}\label{pro:width}
$w(K)\leq 134$.
\end{pro}

The next property follows directly from the construction.

\begin{pro}\label{pro:symmetry}
$K$ is symmetric with respect to $S_3$.
\end{pro}

\begin{pro}\label{pro:different}
The spheres $S_1$, $S_2$ and $S_3$ are all distinct. \end{pro}

\begin{proof}
In search of contradiction, suppose that $S_1$ is isotopic to $S_3$. By symmetry, $S_2$ is also isotopic to $S_3$ and, consequently, $S_1$ is isotopic to $S_2$. Hence, it suffices to draw a contradiction to the assumption that $S_1$ is isotopic to $S_2$. Let $N$ be the three manifold homeomorphic to $S^2 \times I$ cobounded by $S_1$ and $S_2$. Then $N \cap K$ is a collection of vertical strands and $S_4 \subset N$ and so $S_4$ is visibly compressible into $B_T$. However, this is not possible, by Lemma \ref{S4insideincomp}, since we have assumed $d(\partial D,\mathcal{V}_R) \geq 25$ and $d(\partial D,\mathcal{V}_R) \geq 25$.
\end{proof}

\begin{pro}\label{pro:distance}
Any bridge surface for the tangle $(B_1, T_1)$ meets $T_1$ in at least 10 points, and similarly for $(B_2, T_2)$.
\end{pro}

\begin{proof}
By construction, $(B_1, T_1)$ has a bridge surface $\Sigma_1$ of distance 25. Suppose $\Sigma'$ is another bridge surface for $(B_1, T_1)$. By Theorem \ref{thm:boundbridge}, with $B_1$ playing the roles of both $M$ and $N$, it follows that either $\chi(\Sigma')\leq -23$ or, after some c-compressions, $\Sigma'$ is parallel to $\Sigma_1$. In either case, $T_1$ intersects $\Sigma'$ at least as many times as it intersects $\Sigma_1$, i.e., at least 10 times.
\end{proof}

\begin{pro}\label{S1S2cessential}Both $S_1$ and $S_2$ are c-incompressible in $S^3-K$.
\end{pro}

\begin{proof} In search of contradiction, suppose that one of $S_1$ or $S_2$ is c-compressible with c-disk $\Delta'$. By taking $\Delta$ to be the disk bounded by an innermost curve of $\Delta' \cap (S_1\cup S_2)$ in $\Delta'$, we can assume that $\Delta$ is a c-disk for $S_1$ and the interior of $\Delta$ is disjoint from both $S_1$ and $S_2$. The curve $\partial \Delta$ separates $S_1$ into two disks $D_1$ and $D_2$.

If $\Delta$ is contained in $B_1$, then, up to relabeling $D_1$ and $D_2$, $\Delta \cup D_1$ is a 2-punctured sphere in $B_1$. If $\Delta \cup D_1$ is an inessential 2-punctured sphere, then $\Delta$ is boundary parallel, contradicting the fact that $\Delta$ is a c-disk, or $(B_1, T_1)$ is a rational tangle, a contradiction to Property \ref{pro:distance}. Hence, we can assume that $\Delta \cup D_1$ is an essential 2-punctured sphere. Maximally compress and cut-compress $\Delta\cup D_1$ in $B_1$ and let $F$ be one of the resulting components. Note that $F$ is an essential 2-punctured sphere. By Theorem \ref{thm:boundess}, either $F$ can be isotoped to be disjoint from $\Sigma_1$, $\Sigma_1$ has four punctures or $25\leq d(\Sigma_{1})\leq 2-\chi(F) = 2-\chi(\Delta\cup D_1)$. It is easy to show that there are no essential surfaces in the complement of a tangle that are disjoint from its bridge surface. By the construction of $B_1$, $\Sss_1$ has 10 punctures and therefore we conclude that $\chi(\Delta\cup D_1)=\chi(F) \leq -23$. This contradicts the fact that $F$ is an essential 2-punctured sphere.

Suppose $\Delta$ in not contained in $B_1$ or $B_2$. Let $\epsilon$ be the arc in $K-(S_1\cup S_2)$ that is disjoint from $B_T$ and connects $S_1$ to $S_2$.
As $S_4$ is isotopic the boundary of the neighborhood of $S_1\cup S_2 \cup \epsilon$, it follows that $\partial D$ is isotopic to a meridional curve for $\epsilon$. Additionally, if $\Delta$ is contained strictly between $S_1$ and $S_2$ and is disjoint from $\epsilon$, then $\Delta$ is a c-disk for $S_4$ with boundary disjoint from $\partial D$. By Lemma \ref{ccompressible}, $D$ is c-incompressible, so we can assume that $\Delta$ is disjoint from $D$. Since $\Delta$ is disjoint from $D$ then $\Delta$ is a c-disk for $\partial(B_R)-R$. If $\Delta$ is a compressing disk, then $d(\partial D, \mathcal{V}_R)\leq 1$, a contradiction to the construction of the tangle $\mathcal{T}$. If $\Delta$ is a cut-disk and $E$ is the bridge disk for the strand of $R$ that meets $\Delta$, then, possibly after an isotopy of $E$, we can assume that $\Delta$ meets $E$ in a single arc. The boundary of a regular neighborhood of $E\cup \Delta$ contains a compressing disk, $\Delta''$, for $\partial(B_R)-R$ that is disjoint from $\Delta$. Since $\partial \Delta $ is disjoint from $\partial D$ and $\partial(\Delta'')$ is disjoint from $\partial \Delta$ then $d(\partial D, \mathcal{V}_R)\leq 2$. This is a contradiction to the construction of $\mathcal{T}$.

If $\Delta$ is contained strictly between $S_1$ and $S_2$ and meets $\epsilon$, then we can assume that $D_1$ meets $K$ in one point and $D_2$ meets $K$ in three points. If $\epsilon$ has an endpoint in $D_1$, then $D_1 \cup \Delta$ bounds a 3-ball containing a unknotted arc, since $\epsilon$ is unknotted. Hence, $\Delta$ is boundary parallel, a contradiction. Therefore, we can assume that $D_1$ is disjoint from $\epsilon$. The twice punctured sphere $\Delta\cup D_1$ is disjoint from $S_1\cup S_2$ and meets $\epsilon$ in a single point. Therefore, after isotopying $S_4$ to be the boundary of a regular neighborhood of $S_1 \cup S_2 \cup \epsilon$, $(\Delta\cup D_1)\cap B_T$ is a cut-disk for $\partial(B_T)-T$ with boundary parallel to a meridional curve of $\epsilon$. As noted before, such a meridional curve is isotopic to $\partial D$. As argued above, this implies $d(\partial D, \mathcal{V}_R)\leq 2$, a contradiction to the construction of $\mathcal{T}$.
\end{proof}

\begin{pro}\label{pro:esssurfaces}
Let $F_K$ be a connected, meridional, non-boundary parallel, c-incompressible surface in $S^3-\eta(K)$. Then one of the following holds:
\begin{enumerate}

\item $F_K$ can be isotoped to be disjoint from $B_T$.

\item $F_K$ is isotopic to $S_3$.

\item $F$ has at least 14 punctures.
\end{enumerate}
\end{pro}

\begin{proof}
Isotope $F_K$ so that $F_K \cap (S_1 \cup S_2)$ is minimal. Suppose $F \cap S_1$ is non-empty. Since $S_1$ is c-incompressible, by Lemma \ref{S1S2cessential}, and $K$ is a knot, then minimality of $F_K \cap (S_1 \cup S_2)$ implies $F-S_1$ contains no disk or 1-punctured disk components. Let $F^{1}_{K}$ be a component of $B_1 \cap F_K$. By minimality of $|F_K \cap (S_1 \cup S_2)|$, $F^{1}_{K}$ is non-boundary parallel. Since $F_K$ is c-incompressible, so is $F^{1}_K$. So, $F^{1}_K$ can not be isotoped to be disjoint from $\Sigma_1$ and it is not boundary parallel in $B_1-\eta(T_1)$. By Theorem \ref{thm:boundess}, $F^1$, and thus $F$, has at least 14 punctures. Hence, we can assume that $F$ is disjoint from $S_1$ and $S_2$.

Let $M$ be the $S^2 \times I$ region in $S^3$ with boundary $S_1$ and $S_2$. Since $F$ is disjoint from $S_1$ and $S_2$, $F$ is contained in the interior of $M$. Recall that $\epsilon$ is the strand of $K \cap M$ that is disjoint from $B_T$. Let $\eta$ be a small open neighborhood of $S_1 \cup S_2 \cup \epsilon$ in $M$. By transversality, $F$ meets $\eta$ in a possibly empty collection of parallel, disjoint, 1-punctured disks. Since $B_T$ is isotopic to $M-\eta$ then $F_K$ can be isotoped to meet $\partial(B_T)$ in a collection of curves parallel to $\partial D$. Since $F$ is planar and $F \cap \eta(\epsilon)$ is a collection of once-punctured disks, $F^T_{K}= F_K \cap B_T$ is connected. If $F^T_K$ is a disk, then $F_K$ is a 1-punctured sphere in $S^3$, a contradiction. Since $F_K$ is c-incompressible, so is $F^T_K$. By Theorem \ref{thm:Tess}, $F^T_K$ is isotopic to one of ten surfaces. Conclusion (1) and (2) cannot occur since $F_K$ was assumed to be essential. If conclusion (3), (4), (5) or (7) holds, then we can isotope $F_K$ to be disjoint from $B_T$. If conclusion (6) holds, $F_K$ is the boundary union of $D_K$ and a 1-punctured disk that meets $\epsilon$. In this case, $F_K$ is isotopic to $S_3$. If conclusion (8) or (9) holds then, since $d(\partial D,\mathcal{V}_R)\geq 25$ and $d(\partial D,\mathcal{V}_Q)\geq 25$, we can conclude that $F_K$ has at least 14 punctures.
\end{proof}

We will be particularly interested in surfaces obtained by tubing two spheres with a tube that runs along an arc of the knot connecting these spheres. The following definition describes this construction.

\begin{defin}
Let $F$ and $G$ be disjoint embedded spheres in $S^3$ with the property that $F \cap K \neq \emptyset$ and $G \cap K \neq \emptyset$ and let $\alpha$ be a component of $K-(F\cup G)$ with an endpoint in each of $F$ and $G$. Then the boundary of a regular neighborhood of $F \cup \alpha \cup G$ has three components. Let $F\sharp_{\alpha}G$ be the component that is not parallel to $F$ or to $G$.

Equivalently, $F\sharp_{\alpha}G$ is the embedded connected sum of $F$ and $G$ obtained by replacing a neighborhood of $\partial(\alpha)$ in $F$ and $G$ with an annulus that runs parallel to $\alpha$.
\end{defin}

\begin{pro}\label{pro:6punctured}  Let $B_{i,j}$ be the ball bounded by $S_i\sharp_{\alpha} S_j$ with $i,j \in \{1,2,3\}$ and $i \neq j$ that is disjoint from $B_1$ and $B_2$ and let $T_{i,j}=K \cap B_{i,j}$. If $S_i\sharp_{\alpha} S_j$ is incompressible, then any bridge sphere $\Sigma_{i,j}$ for $(B_{i,j}, T_{i,j})$ has at least 10 punctures.

In the special case where $i=1$, $j=2$ and $S_1$ is tubed to $S_2$ along a strand $\alpha$ that intersects $S_4$, then the bridge sphere $\Sigma_{1,2}$ for $(B_{1,2}, T_{1,2})$ has at least 14 punctures.
\end{pro}

\begin{proof}

\noindent{\bf Case 1:} Suppose $S=S_1\sharp_{\epsilon}S_2$. In this case, $S$ is isotopic to $S_4$ and the tangle under consideration is the tangle $\mathcal{T}$. This tangle contains three arcs $\alpha$, $\beta$ and $\gamma$. By construction, $\partial D$ is a simple closed curve in $S_4$ and each arc of $T$ has an endpoint in each of the two components of $S_4-\partial D$. For any two arcs in $T$, say $\alpha$ and $\beta$ define $K_{\alpha,\beta}$ to be the knot obtained by connecting the endpoints of $\alpha$ and $\beta$ via two arcs in $S_4$ so that each of these arcs is disjoint from $\partial D$. Under such restrictions, the knot type of $K_{\alpha,\beta}$ is well defined. As illustrated in Figure \ref{fig:s4.eps}, $K_{\alpha,\beta}$ can also be constructed by taking the connected sum of some numerator closure of an equatorial sub-pair of $(\mathcal{R},\partial D)$  with some numerator closure of an equatorial sub-pair of $(\mathcal{Q},\partial D)$ both of which are knotted by construction. Since $K_{\alpha,\beta}$ is the connected sum of two knots, then the bridge number of $K_{\alpha,\beta}$ is at least 3, by \cite{Sch}. Hence, one of $\alpha$ or $\beta$ meets the bridge sphere of $T$ in at least 4 points, the other meets the sphere in at least 2 points. By examining $K_{\alpha,\gamma}$ and $K_{\beta,\gamma}$ we conclude that one of $\alpha$ or $\gamma$ meets the bridge sphere in at least 4 points and one of $\beta$ or $\gamma$ meets the bridge sphere in at least 4 points. Hence, two of the arcs $\alpha$, $\beta$, or $\gamma$ meet the bridge sphere in at least 4 points and the third meets it in at least 2 points. Thus, the bridge sphere for $\mathcal{T}$ has at least 10 punctures.

\begin{figure}[h]
\centering \scalebox{.5}{\includegraphics{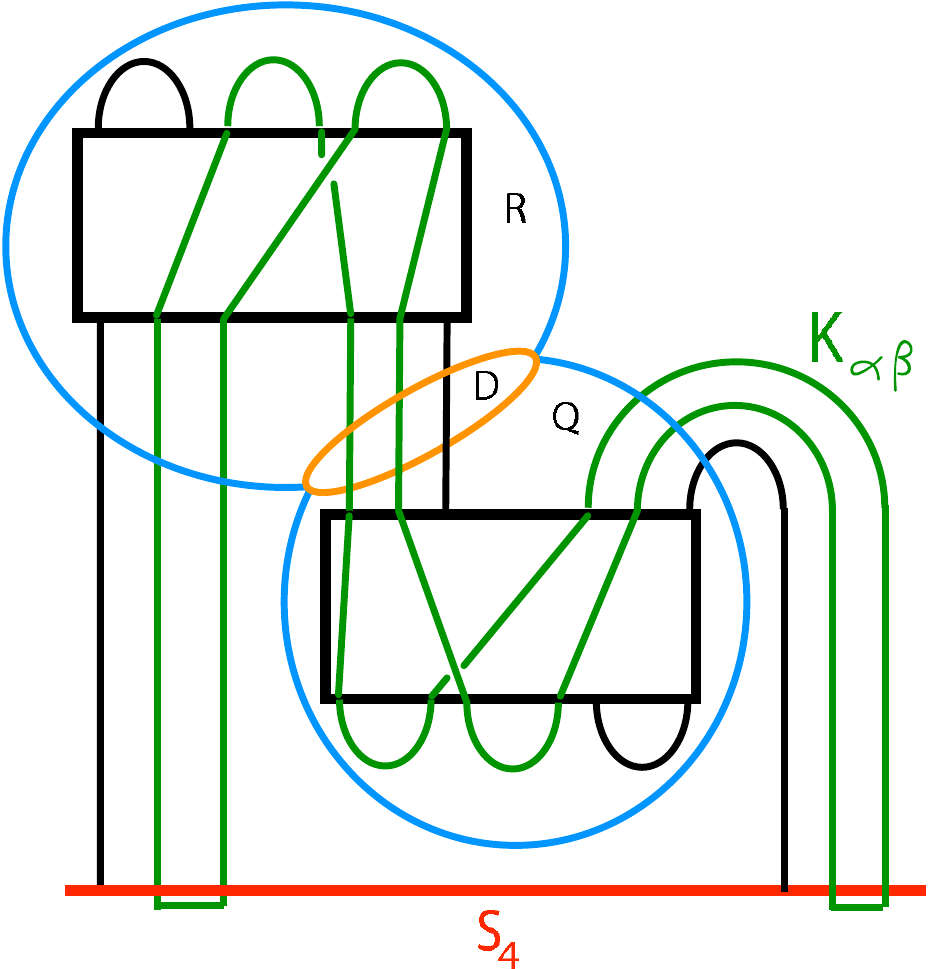}}
\caption{}\label{fig:s4.eps}
\end{figure}

\medskip

\noindent{\bf Case 2:} Suppose $S$ is isotopic to $S_1\sharp_{\alpha}S_2$ where $\alpha$ is one of the three strands connecting $S_1$ to $S_2$ that intersects $S_4$ and therefore passes through both $B_R$ and $B_Q$. Let $\gamma$, $\delta$ be the other two strands of $T$ and, thus, $\epsilon$, $\gamma$ and $\delta$ are the three strands of the tangle $T_{1,2}$ under consideration. By connecting the endpoints of $\gamma$ with an arc contained in $S$, we create a knot, $K_{\gamma}$. The knot-type of $K_{\gamma}$ is well-defined independent of how we connect the points in $\partial \gamma$. We define $K_{\delta}$ and $K_{\epsilon}$ similarly. As illustrated in Figure \ref{fig:s1ts2.eps}, $K_{\gamma}$ can also be constructed by taking the connected sum of some numerator closure of an equatorial sub-pair of $(\mathcal{R},\partial D)$ with some numerator closure of an equatorial sub-pair of $(\mathcal{Q},\partial D)$. As before, both of these are knotted by construction. Since $K_{\gamma}$ is the connected sum of two knots, then the bridge number of $K_{\gamma}$ is at least 3, by \cite{Sch}. Hence $\gamma$ meets the bridge sphere in at least 6 points. A similar argument reveals that $\delta$ meets the bridge sphere in at least 6 points. Additionally, $\epsilon$ must meet this bridge sphere in at least 2 points. Hence, the bridge sphere for $T_{1,2}$ must have at least 14 punctures.

\begin{figure}[h]
\centering \scalebox{.5}{\includegraphics{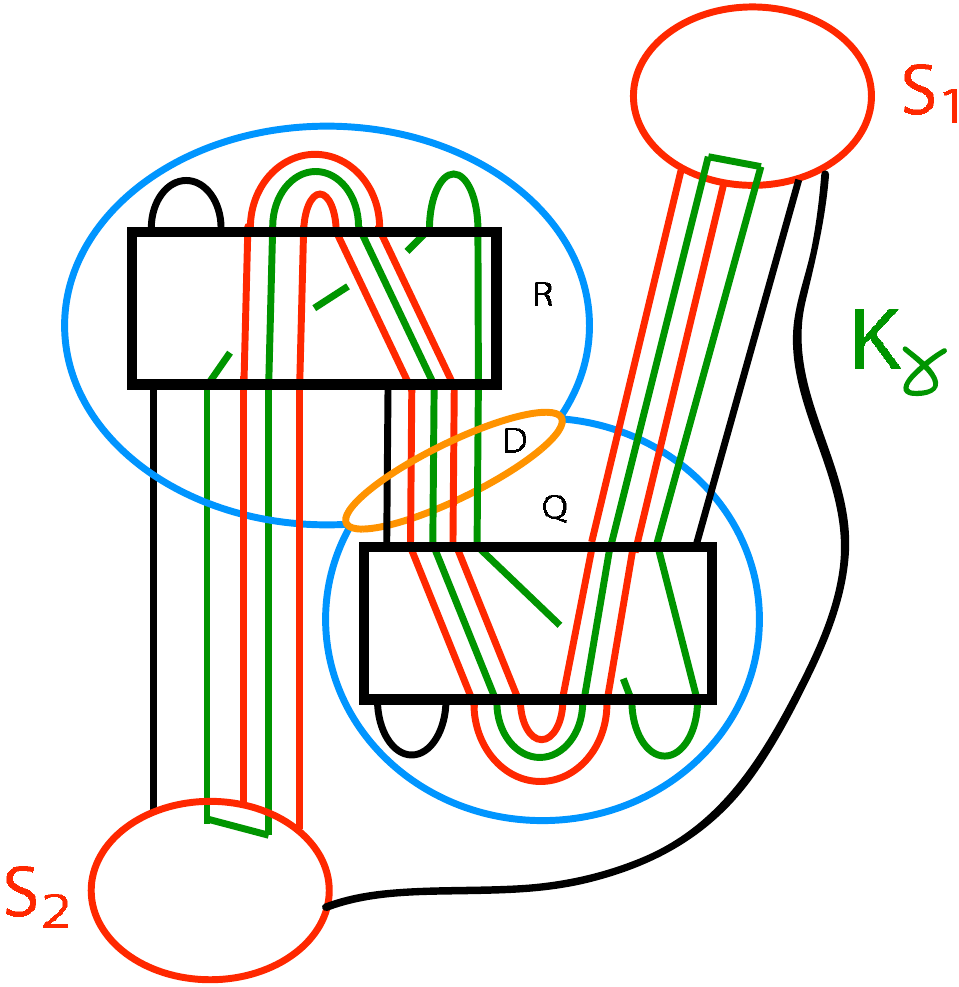}}
\caption{}\label{fig:s1ts2.eps}
\end{figure}

\medskip

\noindent{\bf Case 3:} Suppose $S$ is isotopic to $S_1\sharp_{\alpha}S_3$ or $S_2\sharp_{\alpha}S_3$. Without loss of generality, we may assume that $S$ is isotopic to $S_1\sharp_{\alpha}S_3$. There are four components of $K-(S_1\cup S_3)$ that have a boundary point on each of $S_1$ and $S_3$. One of these components is $\epsilon$. Call the other three components $\beta$, $\gamma$, $\delta$. If $\alpha = \epsilon$, then $S_1\sharp_{\alpha}S_3$ is isotopic to $\partial(B_R)$, and therefore compressible. Hence, we can assume that $\alpha$ passes through $B_R$. Without loss of generality, assume $\alpha = \beta$. In this case, the strands $\gamma$, $\delta$ and $\epsilon$ become the three strands of the tangle $T_{1,3}$. By connecting the endpoints of $\gamma$ with an arc contained in $S$, we create a knot, $K_{\gamma}$. The knot-type of $K_{\gamma}$ is well-defined independent of how we connect the points in $\partial \gamma$. We define $K_{\delta}$ and $K_{\epsilon}$ similarly. As illustrated in Figure \ref{fig:s1ts3.eps}, $K_{\gamma}$ can also be constructed by taking a numerator closure of an equatorial sub-pair of $(\mathcal{R},\partial D)$, which is knotted by construction. Since $K_{\gamma}$ is knotted, the bridge number of $K_{\gamma}$ is at least 2. Hence, $\gamma$ meets any bridge sphere in at least 4 points. A similar argument reveals that $\delta$ meets any bridge sphere in at least 4 points. Additionally, $\epsilon$ must meet this thick sphere in at least 2 points. Hence, the bridge sphere for $T_{1,3}$ must have at least 10 punctures.

\end{proof}

\begin{figure}[h]
\centering \scalebox{.5}{\includegraphics{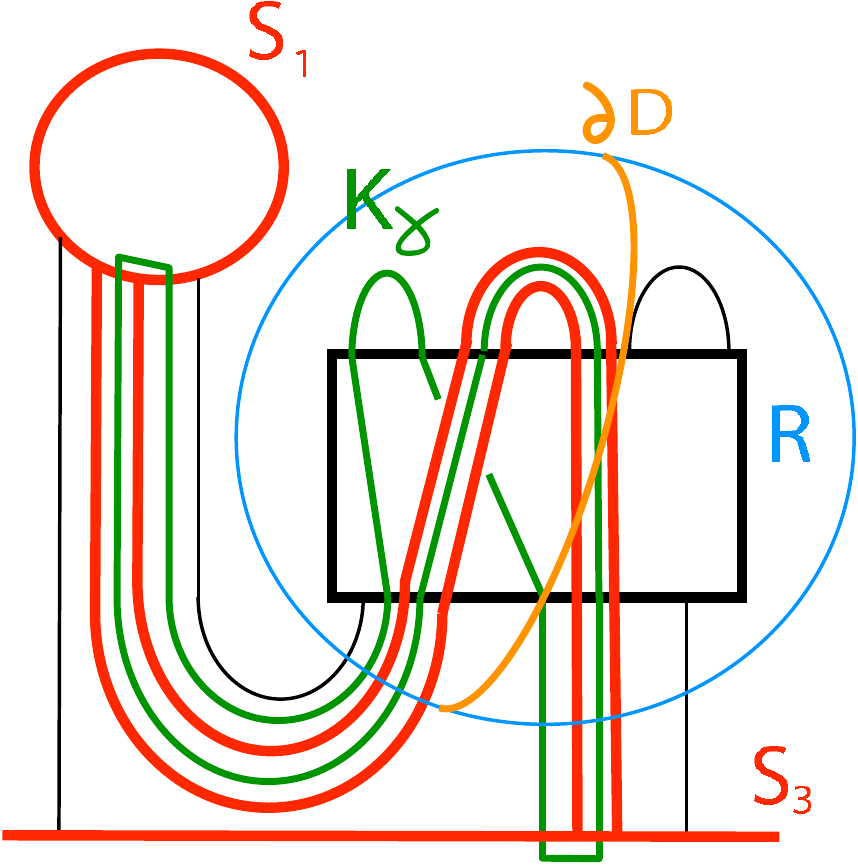}}
\caption{}\label{fig:s1ts3.eps}
\end{figure}

\begin{pro}\label{pro:10puncturedbridge} Let $\mathcal{T}$ be the tangle $[S^3-(B_1 \cup B_2) ]\cap K$ contained in a manifold homeomorphic to $S^2 \times I$ where $S^2\times \{0\}=S_1$ and $S^2\times \{1\}=S_2$. Then any bridge sphere for $\mathcal{T}$ must have at least 10 punctures.
\end{pro}

\begin{proof}
Let $M$ be the embedded copy of $S^2\times I$ in $S^3$ with boundary $S_1 \cup S_2$. The knot $K$ meets $M$ in four arcs $\alpha$, $\beta$, $\gamma$ and $\epsilon$, where $\epsilon$ is the unique arc not in $T$. For any two arcs in $M$, say $\alpha$ and $\beta$, define $K_{\alpha,\beta}$ to be the knot obtained by connecting the endpoints of $\alpha$ and $\beta$ via an arc in $S_1$ and an arc in $S_2$. The knot type of $K_{\alpha,\beta}$ is well defined under this construction. As illustrated in Figure \ref{fig:s1s2thin.eps}, $K_{\alpha,\beta}$ can also be constructed by taking the connected sum of some numerator closure of an equatorial sub-pair of $(\mathcal{R},\partial D)$ with some numerator closure of an equatorial sub-pair of $(\mathcal{Q},\partial D)$, both of which are knotted by construction. Since $K_{\alpha,\beta}$ is the connected sum of two knots, then the bridge number of $K_{\alpha,\beta}$ is at least 3, by \cite{Sch}. Similarly, each of $K_{\alpha,\gamma}$ and $K_{\beta,\gamma}$ has bridge number at least 3. Since $K_{\alpha,\beta} \cup K_{\alpha,\gamma} \cup K_{\beta,\gamma}$ meets any bridge sphere in at least 18 points, then any bridge sphere meets the union of $\alpha$, $\beta$ and $\gamma$ in at least 9 points. Additionally, $\epsilon$ must meet the thick level in at least one point. Thus, any bridge sphere for $\mathcal{T}$ must have at least 10 punctures.
\end{proof}

\begin{figure}[h]
\centering \scalebox{.5}{\includegraphics{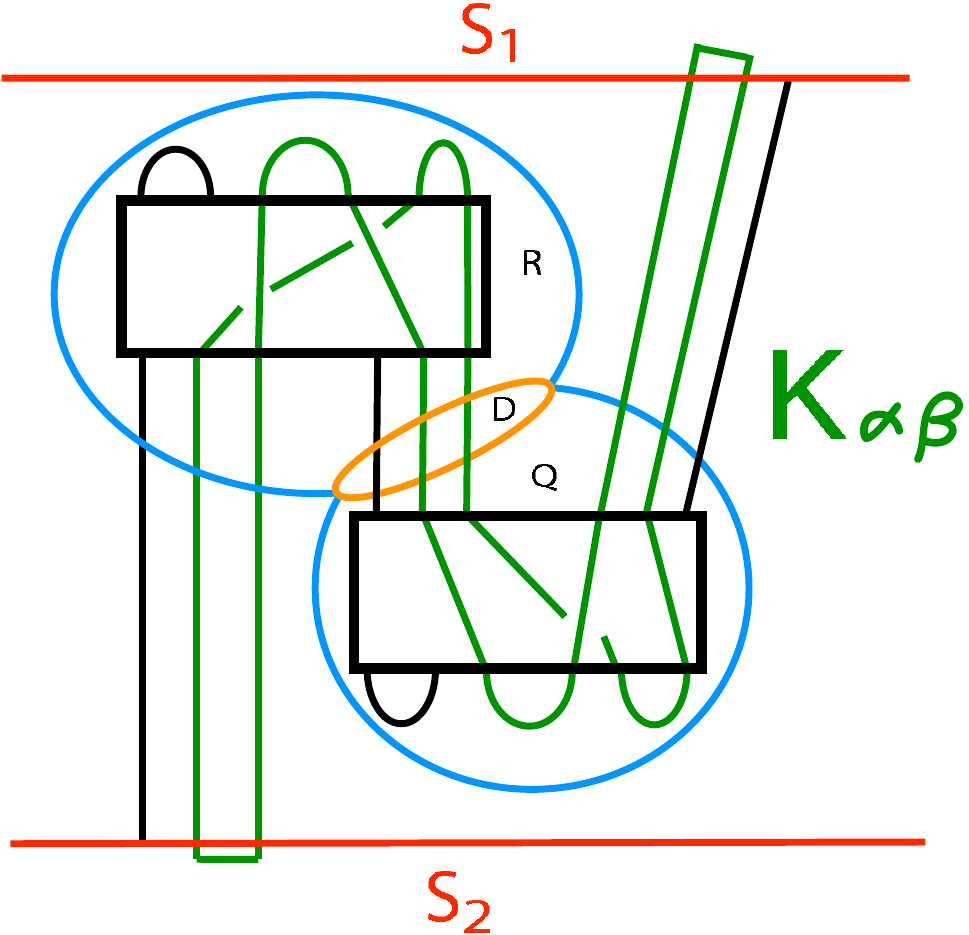}}
\caption{}\label{fig:s1s2thin.eps}
\end{figure}

\begin{pro}\label{pro:12puncturedbridge} Let $B'$ and $B''$ be the two balls bounded by $S_3$ and let $T'=K\cap B'$ and $T''=K\cap B''$. Then any bridge surface for each of the tangles $(B', T')$ and $(B'', T'')$ has at least 12 punctures.
\end{pro}

\begin{proof}
By Property \ref{pro:symmetry}, $(B', T')$ and $(B'', T'')$ are identical, so it suffices to show the result for $(B', T')$. Let $\Sigma'$ be a bridge surface for the tangle $(B', T')$. Without loss of generality, suppose $B_1$ is contained in $B'$.  By Theorem \ref{thm:boundbridge}, either $d(\Sigma_{1}) \leq 2-\chi(\Sigma')$, or $\Sigma' \cap B_1$ is a sphere parallel to $\Sigma_{1}$ with tubes attached. In the first case, $\Sigma'$ has at least 26 punctures. In the second case, if no tubes are attached, then $\Sigma'$ is isotopic to $\Sigma_1$ and so, by Property \ref{S1S2cessential}, $S_1$ is an essential surface contained in $B'$ but disjoint from $\Sss'$ which is not possible since $S_1$ is not isotopic to $S_3$. Therefore, to recover $\Sigma'$ from $\Sigma_1$ at least one tube must be attached. Since each additional tube adds at least two punctures, $\Sigma'$ has at least 12 punctures.

\end{proof}

\section{Meridional essential spheres in the complement of $K$}\label{sec:meress}

In this section, we will classify all essential spheres in the complement of $K$ with fewer than 14 punctures.

\begin{prop}\label{prop:nonein}
Any meridional essential sphere that is embedded in one of $B_1$,  $B_2$ or $B_T$ has at least $14$ punctures.
\end{prop}

\begin{proof}
Suppose $F$ is an essential sphere and $F \subset B_1$. Maximally cut-compress $F$ in $B_1$ and let $F'$ be one of the resulting components. Note that $F'$ is incompressible and $\chi(F')\geq \chi(F)$. By Theorem \ref{thm:boundess}, it follows that either $F'$ can be isotoped to be disjoint from $\Sss_1$, $\Sss_1$ has four punctures or $25\leq d(\Sigma_{1})\leq 2-\chi(F')\leq 2-\chi(F)$. It is easy to show that there are no essential surfaces in the complement of a tangle that are disjoint from its bridge sphere. By the construction of $B_1$, $\Sss_1$ has 10 punctures and therefore we conclude that $\chi(F)=\chi(F') \leq -23$. It follows that $F$ must have at least $25$ punctures. A symmetric argument produces the desired result if $F \subset B_2$.

Suppose $F$ is an essential sphere in $B_T$. After maximally cut-compressing $F$ in $B_T$, each component is a meridional, c-incompressible sphere in $B_T$. Let $F'_K$ be one such component. By Theorem \ref{thm:Tess}, we conclude that $d(\partial(D),\mathcal{V}_R)\leq 2-\chi(F'_K)$, or $d(\partial(D),\mathcal{V}_Q)\leq 2-\chi(F'_K)$, or $F'_K$ is isotopic to $\partial(B_T)-T$.
 In the first two case, since $d(\partial(D),\mathcal{V}_R) \geq 25$ and $d(\partial(D),\mathcal{V}_Q)\geq 25$, we conclude that $F'$, and thus $F$, has at least 14 punctures. Hence, we can assume that, after maximally cut-compressing $F$ in $B_T$, every component is isotopic to $\partial(B_T)-T$. If there is more than one such component, then reversing the cut-compression by attaching a tube between two parallel copies of $\partial(B_T)-T$ results in a compressible surface, a contradiction to the incompressibility of $F$. Hence, $F$ is isotopic to $\partial(B_T)-T$, a contradiction to $F$ being essential.
\end{proof}

We associate to the given projection of $K$ a graph $\Gamma$ embedded in $S^2$ where each vertex corresponds to one of the 4 balls $B_1,...,B_4$ and the edges correspond to strands of the knot connecting them. We may assume that $K$ lies in an arbitrarily small neighborhood of $S^2$. We will use this graph on several occasions.

\begin{lemma} \label{lem:nobigons} Suppose $F$ is an essential meridional sphere disjoint from $B_i$ for $i=1,..,4$, that is embedded so that $|F \cap S^2|$ is minimal. Then $F \cap S^2$ is a single simple closed curve $\tau$. Moreover there are no bigons in $S^2$ whose boundary is the endpoint union of a segment of an edge of $\Ggg$ and a segment of $\tau$.
\end{lemma}

\begin{proof} As $F$ is disjoint from all $B_i$, it intersects the 2-sphere containing $\Ggg$ in circles disjoint from the vertices of the graph. If any such circle bounds a disk in the 2-sphere disjoint from the graph, it can be removed by an isotopy using the incompressibility of $F$, contradicting the assumption that $|F \cap S^2|$ is minimal.

Suppose $\alpha$ and $\beta$ are two simple closed curves in $F \cap S^2$. By the above both curves contain points of intersection between $F$ and $K$. Consider a product neighborhood $S^2 \times I$ of $S^2$ so that $\bdd (S^2 \times I)$ is disjoint from $K$ and $F \cap (S^2 \times I)$ is a union of vertical fibers. Then $\alpha$ is parallel to two curves in the boundary of the product, one lying in $F \cap (S^2 \times \{0\})$ and the other in $F \cap (S^2 \times \{1\})$. Since $F$ is connected, one of these curves, $\alpha'$, separates $\alpha$ and $\beta$ in $F$. Let $D$ be the disk in $\bdd (S^2 \times I)$ that $\alpha'$ bounds and note that $D$ is disjoint from $K$. As $F$ is incompressible, after possibly an isotopy of $F$ we may assume that $F \cap int(D)=\emptyset$. As $\bdd D = \alpha'$ separates the punctures of $F$ lying in $\alpha$ from the punctures lying in $\beta$, the disk $D$ is a compressing disk for $F$, contradicting the hypothesis.

  Suppose a sub-arc of $\tau = S^2 \cap F$ cobounds with some edge of $\Ggg$ a bigon. Without loss of generality, we may assume that the interior of the bigon is disjoint from $\Ggg$ and $F$. A regular neighborhood of this bigon contains a compressing disk for $F$ unless $F$ is a twice punctured sphere parallel to a segment of the edge. As $F$ is essential, no such bigons can exist.
\end{proof}

\begin{prop}\label{prop:noneout}
Any essential meridional sphere, $F$, that is disjoint from $B_i$ for $i=1,..,4$ is parallel to one of $S_1$, $S_2$, $S_3$ or $S_4$. In particular, any such surface has at most 6 punctures.
\end{prop}

\begin{proof}

We will continue using the terminology developed in the proof of Lemma \ref{lem:nobigons}. By that lemma, it follows that if $\tau = S^2 \cap F$ intersects any edge in a triple of parallel edges of $\Ggg$, it must intersect all three of them. Thus, we can replace these triples with a single edge and assign it a weight of 3. The projection of $K$ is then modeled by a cycle graph on 4 vertices, $C_4$, where three of the edges have weight 3 and one has weight 1. The number of punctures of $F$ lying in $\tau$ is the sum of the weights of the edges $\tau$ intersects. By the Jordan-Brower theorem, $\tau$ has an even number of intersections with the edges of $\Gamma$ and, by Lemma \ref{lem:nobigons}, consecutive edges that $\tau$ intersects must be distinct.

Let $E$ be one of the two disks $C_4$ bounds in $S^2$ and let $l$ be an outermost arc of intersection of $\tau$ and $E$. Let $p_1$ and $p_2$ be the two endpoints of $l$. Choose $p'_1$ and $p'_2$ to be points in $\tau$ that are close to $p_1$ and $p_2$ respectively but not in $l$ so that the segment of $l'=\tau-(p'_1 \cup p'_2)$ containing $l$ only intersects $C_4$ in $p_1$ and $p_2$. Let $g$ be any embedded path in $S^2-(\tau \cup C_4)$ from $p'_1$ to $p'_2$. Consider a small regular neighborhood $S^2 \times I$ of $S^2$. Then there is a disk with boundary $(l'\times\{0\}) \cup (p'_1\times I) \cup (p'_2\times I)  \cup (l'\times\{1\})$ that contains two copies of $g$. This disk would be a compressing disk for $F$ unless $p_1 \cup p_2=\tau \cap C_4$. Therefore, we conclude that $\tau$ intersects exactly two of the edges of $C_4$ and intersects each of those exactly once. The result follows by considering all possible pairs of edges.
\end{proof}

\begin{prop}\label{prop:notwicepunctured}
There are no essential twice-punctured spheres that are disjoint from all of $S_1$, $S_2$ and $S_3$.
\end{prop}

\begin{proof}
Let $F$ be an essential twice punctured sphere in the complement of $K$ disjoint from $S_1$, $S_2$ and $S_3$. By Proposition \ref{prop:nonein}, $F$ is not embedded in $B_1$ or $B_2$. Without loss of generality, we may assume that $F$ is on the same side of $S_3$ as $B_1$ but outside of $B_1$. In this case, it is easy to see that $F$ can be isotoped to lie entirely in $B_T$, contradicting Lemma \ref{prime}.

\end{proof}

\begin{prop}\label{prop:cessential}
The spheres $S_1$, $S_2$ and $S_3$ are c-essential.
\end{prop}

\begin{proof}

The spheres $S_1$ and $S_2$ are c-essential by Property \ref{S1S2cessential}. Suppose $S_3$ has a c-disk $D^c$. We may assume that this c-disk is disjoint from $S_1 \cup S_2$. Let $\Delta$ be one of the two punctured disks $\bdd D^c$ bounds on $S_3$. Then $F=D^c \cup \Delta$ is a sphere that is disjoint from all $S_i$, $i=1,2,3$. Note that we may choose $\Delta$ so it has at most 2-punctures. As $D^c$ has at most one puncture and every sphere has an even number of punctures, $F$ is a twice punctured sphere. As $K$ is not a split link, $F$ is an essential twice punctured sphere disjoint from $S_1$, $S_2$ and $S_3$ thus contradicting Proposition \ref{prop:notwicepunctured}.

\end{proof}

\begin{cor} \label{cor:prime} $K$ is a prime knot.

\end{cor}

\begin{proof}
Suppose $F$ is a decomposing sphere for $K$. As $S_1$, $S_2$ and $S_3$ are c-essential by Proposition \ref{prop:cessential}, we can find such a sphere that is disjoint from $S_1 \cup S_2 \cup S_3$ contradicting Proposition \ref{prop:notwicepunctured}
\end{proof}

\begin{cor}\label{cor:fourpuncturedessential}
Every essential 4-punctured sphere is c-essential.
\end{cor}

\begin{pro}\label{pro:14puncturedbridge} Any bridge surface for the tangles $\mathcal{T}'=(B',K-T_1)$ and $\mathcal{T}''=(B'',K-T_2)$ has at least 14 punctures.
\end{pro}

\begin{proof}

By Property \ref{pro:symmetry}, it suffices to prove the result for $\mathcal{T}'$. Let $\Sigma$ be a bridge surface for $\mathcal{T}'$. By Theorem \ref{thm:boundbridge}, either $d(\Sigma_{2}) \leq 2-\chi(\Sigma)$, or $\Sigma \cap B_2$ is a sphere parallel to $\Sigma_{2}$ with tubes attached. In the first case, $\Sigma$ has at least 26 punctures, so suppose we are in the second case. If no tubes are attached, i.e., if $\Sigma$ is parallel to $\Sigma_2$, then $S_2$ is an essential sphere that is disjoint from the bridge surface of the tangle $T'$ which is not possible since $S_2$ is not isotopic to $S_3$. If more than one tube is attached, then $w(\Sigma)\geq 14$. Similarly, if a tube that corresponds to a compressing disk decomposes $\Sigma$ into a component parallel to $\Sigma_2$ and a sphere with more than two punctures or if a tube that corresponds to a cut-disk decomposes $\Sigma$ into a component parallel to $\Sigma_2$ and a sphere with more than four punctures, then $w(\Sigma)\geq 14$. Therefore, we may assume that $\Sigma$ is parallel to $\Sigma_2$ with exactly one tube attached. If this tube corresponds to a compressing disk, then it connects $\Sigma_2$ to a twice punctured sphere and if it corresponds to a cut-disk, then it connects $\Sigma_2$ to a four times punctured sphere.

\medskip

\textbf{Claim:} We may assume the tube corresponds to a compressing disk.

{\em Proof of claim:} Suppose the tube corresponds to a cut-disk $D^c$. Then $\bdd D^c$ bounds a three punctured disk $\Delta$ in $\Sigma$. Consider the arc $\kappa$ that has both of its endpoints in $\Delta$. A bridge disk for $\kappa$ can be isotoped to be disjoint from $D^c$ using an innermost disk/outermost arc argument. Let $E$ be the compressing disk for $\Sigma$ obtained by taking the boundary of a regular neighborhood of the bridge disk for $\kappa$. Note that $E$ is disjoint from $D^c$. Additionally, $D^c$, $E$ and a once punctured annulus $A$ in $\Sigma$ cobound a twice punctured sphere. By Corollary \ref{cor:prime}, $K$ and, thus, $\mathcal{T}'$ is prime. As $\mathcal{T}'$ is prime, it follows that $E\cup A$ is isotopic to $D^c$. Therefore, the surface $F$ obtained by cut-compressing $\Sigma$ along $D^c$ that does not contain $\Delta$ is isotopic to the surface $F'$ obtained from $\Sigma$ by compressing it along $E$ that contains $A$. As $F$ is parallel to $\Sigma_2$, so is $F'$. Therefore, we may replace the tube corresponding to $D^c$ with a tube corresponding to $E$.
\qed

\medskip

By the claim, we may assume that $\Sigma$ can be recovered by tubing $\Sigma_2$ to a twice punctured sphere $Q$ along a tube that is disjoint from the knot. There are 4 strands of $K$ that have one endpoint in $S_1$ and the other in $\Sigma_2$. By Property \ref{S1S2cessential}, $S_1$ and $S_2$ are c-essential, so $Q$ can be isotoped to be disjoint from both $S_1$ and $S_2$. Hence, the sphere $Q$ intersects at most one of these strands, so there are at least three strands that have one endpoint in $S_1$ and one point in $\Sigma$. At least two of these strands, $\alpha$ and $\beta$, intersect $B_T$. As $\Sigma$ is a bridge surface, $\alpha$ and $\beta$ are both vertical and, therefore, parallel to each other. Let $R$ be the rectangle they cobound. Connecting $\alpha\cap B_T$ and $\beta\cap B_T$ along the two arcs $R \cap S_4$ results in the unknot. However, this is the knot $K_{\alpha, \beta}$ described in Case 1 of the proof of Property \ref{pro:6punctured}. By the argument there, the bridge number of $K_{\alpha, \beta}$ is at least three, leading to a contradiction.
\end{proof}

\begin{prop}\label{prop:noc-essential}
Any c-essential meridional sphere with fewer than 14 punctures has exactly 4 punctures and is parallel to one of $S_1$, $S_2$ or $S_3$.
\end{prop}

\begin{proof}
Let $F$ be a c-essential sphere with fewer than 14 punctures. Isotope it to intersect $B_1 \cup B_2 \cup B_T$ minimally and suppose first that $F \cap B_1\neq \emptyset$. Note that no component of $F - S_1$ can be a c-disk as $S_1$ is c-incompressible. It follows that, if $F'$ is any component of $F \cap B_1$, $\chi(F') \geq \chi(F)$. Furthermore, $F'$ cannot be parallel to $\bdd (B_1-\eta(K))$ as in this case either $|F \cap B_1|$ can be reduced or $F$ has a cut-disk (this situation occurs if $F\cap B_1$ has a component that is an annulus parallel to a segment of $B_1 \cap K$). Recall that $\Sss_1$ has 10 punctures and there are no essential surfaces in the complement of the tangle that are disjoint from the bridge sphere of the tangle. By Theorem \ref{thm:boundess}, it follows that $d(\Sigma_1, K\cap B_1) \leq 2-\chi(F')$. As $d(\Sigma_1, K \cap B_1)\geq 25$, this implies that $\chi(F') \leq -23$ and therefore $\chi(F) \leq -23$. This contradicts the hypothesis that $F$ has at most 14 punctures and therefore it follows that $F \cap B_1=\emptyset$. Similarly, $F \cap B_2$ must also be empty.

By Property \ref{pro:esssurfaces}, it follows that if $F \cap B_T \neq \emptyset$, then $F$ is $S_3$ or has at least 14 punctures.
\end{proof}

\begin{prop} \label{prop:sixpunctured} The only incompressible 6-punctured spheres are the ones obtained by tubing two essential 4-punctured spheres that are not mutually parallel along a strand of $K$. In particular, if $P$ is an incompressible 6-punctured sphere, then its cut-disk, $B_1$, and $B_2$ are on the same side of $P$.\end{prop}
\begin{proof}

Suppose $G$ is an incompressible 6-punctured sphere. By Proposition \ref{prop:noc-essential}, it follows that $G$ is cut-compressible. As $K$ is prime, cut-compressing a 6-punctured sphere results in a pair of 4-punctured spheres. If the original sphere was incompressible, so are the two 4-punctured ones.

By Corollary \ref{cor:fourpuncturedessential}, all essential 4-punctured spheres are c-essential and, by Proposition \ref{prop:noc-essential}, all c-essential 4-punctured spheres in the complement of $K$ are parallel to one of $S_1$, $S_2$ or $S_3$. If two parallel copies of some 4-punctured sphere are tubed together along a single strand of the knot, the resulting 6-punctured sphere is always compressible. Therefore, any incompressible 6-punctured sphere is the result of tubing together two of $S_1$, $S_2$ or $S_3$.
\end{proof}

\section{Bridge number} \label{sec:bridgenumber}

In this section, we will show that thin and bridge position for $K$ do not coincide.

\begin{lemma}\label{cdiskbridgenotthin}Suppose $\Sigma$ is a minimal bridge sphere for $K$, and $C_1$ and $C_2$ are c-disks for $\Sigma$ above and below it respectively so that $\partial(C_1)\cap \partial(C_2) = \emptyset$. Then bridge position for $K$ is not thin position.
\end{lemma}

\begin{proof}
Let $\Delta_1$ and $\Delta_2$ be the disjoint disks $\bdd C_1$ and $\bdd C_2$ bound in $\Sigma$. The disk $\Delta_1$ must have at least two punctures as $\bdd C_1$ is essential in $\Sigma$. In particular, there is a strand $\kappa_1$ above $\Sigma$ that has both of its endpoint in $\Delta_1$ and is disjoint from $C_1$. As $\Sigma$ is a bridge sphere, $\kappa_1$ has a bridge disk $E_1$. Using the fact that there are no spheres that intersect the knot in exactly one point and an innermost disk argument, we can choose $E_1$ so that it intersects $C_1$ only in arcs. Using an outermost arc argument and the fact that we can always pick an outermost arc so that the disk it bounds in $C_1$ that is disjoint from $E_1$ does not contain the puncture, we can choose $E_1$ to be completely disjoint from $C_1$. In particular, $E_1 \cap \Sigma \subset \Delta_1$. Similarly, there is a bridge disk $E_2$ for some strand $\kappa_2$ below $\Sigma$ so that $E_2 \cap \Sigma \subset \Delta_2$. This pair of disks allows us to push the maximum of $\kappa_1$ below the minimum of $\kappa_2$, thus decreasing the width of $K$. Therefore, bridge position of $K$ is not thin position.
\end{proof}

We will rely heavily on the terminology introduced in Section \ref{Vertical cut-disks} and illustrated in Figure \ref{fig:labels1.eps}. In addition, we need the following definition first introduced in \cite{Bl}.

\begin{defin}\label{worm-like}
Let $S$ be a 2-sphere embedded in $S^3$ so that $S$ meets $K$ transversely in exactly 4 points. $S$ is {\em worm-like} if, for every saddle $\sigma=s_{1}^{\sigma} \vee s_{2}^{\sigma}$, each of  $s_{1}^{\sigma}$ and $s_{2}^{\sigma}$ cuts $S$ into two twice punctured disks and every saddle in $S$ is nested with respect to the same side of $S$.
\end{defin}

\begin{theorem}\cite[Theorem A]{Bl}\label{sphereworm-like} If $S$ is a c-incompressible 4-punctured sphere and bridge position for $K$ is thin position, then there is an isotopy of $S$ and $K$ resulting in $h|_{K}$ having $\beta(K)$ maxima and $S$ being worm-like.
\end{theorem}

\begin{theorem}\label{spherenice} Suppose $S$ is a c-incompressible 4-punctured sphere in $S^3$ bounding 3-balls $B_1$ and $B_2$ on opposite sides. Additionally, suppose bridge position for $K$ is thin position. Then, up to relabeling $B_1$ and $B_2$,  there is a minimal bridge sphere, $\Sigma$, such $B_1 \cap \Sigma$ is a collection of punctured disks and $S-\Sigma$ is a collection of annuli and 2-punctured disks.
\end{theorem}

\begin{proof}
By Theorem \ref{sphereworm-like}, we can assume that $S$ is worm-like. In particular, we can assume that all saddles in $S$ are nested with respect to $B_1$. Let $E_1$ and $E_2$ be the unique outermost disks of $S$. By definition of worm-like, $E_1$ and $E_2$ are 2-punctured disks. Since all saddles in $S$ are nested with respect to $B_1$, the interior of $A_{\sigma}$ is disjoint from $S$ for every saddle $\sigma$. In particular, $A_{\sigma}$ is disjoint from $A_{\tau}$ whenever $\sigma \neq \tau$. If $E_{\sigma}$ has a unique maximum, we can horizontally shrink and vertically lower $A_{\sigma}$ so that, after this isotopy, $A_{\sigma}$ is contained in an arbitrarily small neighborhood of the level sphere containing $\sigma$. If $E_{\sigma}$ has a unique minimum, we can horizontally shrink and vertically raise $A_{\sigma}$ so that after this isotopy $A_{\sigma}$ is contained in an arbitrarily small neighborhood of the level sphere containing $\sigma$. By Morse theory, we can assume that all saddles occur at distinct heights. Since $A_{\sigma}$ is disjoint from $A_{\tau}$ whenever $\sigma \neq \tau$ and all saddles occur at distinct heights, we can isotope $S$ and $K$ so that $h(A_{\sigma}) \cap h(A_{\tau}) =\emptyset$ whenever $\sigma \neq \tau$. This isotopy does not change the number of critical points of $h_K$, fixes the saddles of $S$ and leaves invariant the number of saddles of $S$.

Suppose that, after isotopying all the $A_{\sigma}$'s to occur at distinct heights, there exists an $A_{\sigma}$ with a unique minimum above an $A_{\tau}$ with a unique maximum. Each of $D^{\sigma}_1$, $D^{\sigma}_2$, $D^{\tau}_1$, $D^{\tau}_2$ meets $K$ as otherwise $S$ is compressible. Since each of $E_{\sigma}$ and $E_{\tau}$ are disjoint from $K$, $A_{\sigma}$ contains a minimum of $K$ and $A_{\tau}$ contains a maximum of $K$. Since $A_{\sigma}$ is contained completely above $A_{\tau}$, there is a minimum of $K$ above a maximum of $K$, a contradiction to the assumption that bridge position is thin position. Hence, we can assume that all $A_{\sigma}$'s with unique maxima lie above all $A_{\sigma}$'s with unique minima.

Fix $\sigma$ and $\tau$ as the two unique outermost saddles of $S$. Suppose $E_1=D_{\sigma}$ and suppose $E_1$ has a unique maximum. Let $\{x_1,x_2\} = K \cap E_1$ such that $h(x_2)<h(x_1)$.

The following claim is very similar to the claim in the proof of Lemma \ref{cinesssaddle}.
\medskip

\textbf{Claim 1:} After an isotopy that fixes $S$ and preserves that number of maxima of $h_K$, we can assume that $h_{K\cap B_{\sigma}}$ has a local maximum at $x_1$.

{\em Proof:} Suppose $h_{K\cap B_{\sigma}}$ has a local minimum at $x_1$. Let $y$ be the maximum of $K$ that is nearest $x_1$ and inside $B_{\sigma}$. Such a $y$ must exist since $K$ does not meet $E_1$ above $x_1$. Let $\alpha$ be the monotone subarc of $K$ with boundary $x_1 \cup y$. The arc $\alpha$ is completely contained in $B_{\sigma}$. Let $\beta$ be a monotone arc in $E_1$ with endpoints $x_1$ and $z$ such that $h(z)=h(y)$. Let $\delta$ be a level arc disjoint from $K$ and contained in $B_{\sigma}$ connecting $x$ to $y$. Let $E$ be a vertical disk with boundary $\alpha \cup \beta \cup \delta$ that is embedded in $B_{\sigma}$. We can assume the interior of $E$ meets $K$ transversely in a collection of points $k_1, ..., k_n$ where $h(k_1)>h(k_2)> ... >h(k_n)$. Let $\eta_i$ be the arc corresponding to a small neighborhood of $k_i$ in $K$ for each $i$.

Replace $\eta_1$ with a monotone arc which starts at an end point of $\eta_1$, runs
parallel to $E$ until it nearly reaches $E_1$, travels along $E_1$ until it returns
to the other side of $E$, travels parallel to $E$ (now on the opposite side)
and connects to the other end point of $\eta_1$. The result is isotopic to $K$, does not change the number of maxima of $h|_K$ and reduces $n$. By
induction on $n$, we may assume that $K \cap E = \emptyset$. Isotope $\alpha$ along $E$
until it lies just outside of $E_1$ except where it intersects $E_1$ exactly at
the point $y$. After a small tilt of $K$, we see that $x_1$ is now a local maximum of $h_{K\cap B_{\sigma}}$.$\square$

\medskip

By a symmetric argument, we conclude that if $E_1$ has a unique minimum and $\{x_1,x_2\} = K \cap E_1$ such that $h(x_2)<h(x_1)$, then, after an isotopy that fixes $S$ and preserves the number of maxima of $h_K$, we can assume that $h_{K\cap B_{\sigma}}$ has a local minimum at $x_2$.

Suppose $E_1$ has a unique maximum, $\{x_1,x_2\} = K \cap E_1$ such that $h(x_2)<h(x_1)$ and there exists an $A_{\varsigma}$ such that $A_{\varsigma}$ has a unique minimum and $h(\varsigma)>h(x_2)$. If $x_2$ is a maximum of $h_{B_{\sigma}\cap K}$ then $\sigma$ is a removable saddle. By \cite[Lemma 3.5]{Bl}, we can eliminate $\sigma$ while preserving the number of maxima of $h_K$. Hence, we can assume $x_2$ is a minimum of $h_{B_{\sigma}\cap K}$. Since $x_2$ is a minimum of $h_{B_{\sigma}\cap K}$, then either there is a maximum of $K$ inside $B_{\sigma}$ or $x_1$ is connected to $x_2$ via a monotone subarc of $K$. In the later case, examine a level disk in $B_{\sigma}$ with boundary in $E_1$ such that this disk is just below $x_2$. If $K$ meets this disk, then there is a maximum of $K$ in $B_{\sigma}$. If $K$ does not meet this disk then $S$ is compressible, a contradiction. Hence, we can assume that there is a maximum of $K$ in $B_{\sigma}$. As previously noted, $A_{\varsigma}$ contains a minimum of $K$.

Since $x_1$ is a local maximum of $h_{B_{\sigma}\cap K}$, we can horizontally shrink and vertically lower the portion of $B_{\sigma}$ above $x_2$ to within an arbitrarily small neighborhood of the level sphere containing $x_2$. Similarly, we can horizontally shrink and vertically raise $A_{\varsigma}$ to within an arbitrarily small neighborhood of the level sphere containing $\varsigma$. These isotopies do not change the number of critical points of $h_K$, however they do raise a minimum above a maximum, a contradiction to the assumption that bridge position coincides with thin position.

By a similar argument, we can eliminate the following possibilities.

\begin{enumerate}
\item $E_1$ has a unique minimum, $\{x_1,x_2\} = K \cap E_1$ such that $h(x_2)<h(x_1)$ and there exists an $A_{\varsigma}$ such that $A_{\varsigma}$ has a unique maximum and $h(\varsigma)<h(x_1)$.

\item $E_2$ has a unique maximum, $\{x_1,x_2\} = K \cap E_2$ such that $h(x_2)<h(x_1)$ and there exists an $A_{\varsigma}$ such that $A_{\varsigma}$ has a unique minimum and $h(\varsigma)>h(x_2)$.

\item $E_2$ has a unique minimum, $\{x_1,x_2\} = K \cap E_2$ such that $h(x_2)<h(x_1)$ and there exists an $A_{\varsigma}$ such that $A_{\varsigma}$ has a unique maximum and $h(\varsigma)<h(x_1)$.

\item $E_1$ has a unique minimum where $\{x_1,x_2\} = K \cap E_1$ such that $h(x_2)<h(x_1)$ and $E_2$ has a unique maximum where $\{y_1,y_2\} = K \cap E_2$ such that $h(y_2)<h(y_1)$ and $h(x_1)>h(y_2)$.

\end{enumerate}

Suppose $E_1$ has a unique minimum where $\{x_1,x_2\} = K \cap E_1$ such that $h(x_2)<h(x_1)$ and $E_2$ has a unique maximum where $\{y_1,y_2\} = K \cap E_2$ such that $h(y_2)<h(y_1)$. Let $\{\kappa_1,...,\kappa_k\}$ be the set of all saddles such that $A_{\kappa_i}$ contains a unique maximum and $\{\varsigma_1,..., \varsigma_s\}$ be the set of all saddle such that $A_{\varsigma_i}$ contains a unique minimum. By the above eliminations, $min(h(y_2),h(\kappa_1),...,h(\kappa_k))>max(h(x_1),h(\varsigma_1),..., h(\varsigma_s))$ and any level sphere with height strictly between these two values is a bridge sphere satisfying the conclusions of the theorem.

Suppose that $E_1$ has a unique maximum where $\{x_1,x_2\} = K \cap E_1$ such that $h(x_2)<h(x_1)$ and $E_2$ has a unique maximum where $\{y_1,y_2\} = K \cap E_2$ such that $h(y_2)<h(y_1)$. Let $\{\kappa_1,...,\kappa_k\}$ be the set of all saddles such that $A_{\kappa_i}$ contains a unique maximum and $\{\varsigma_1,..., \varsigma_s\}$ be the set of all saddle such that $A_{\varsigma_i}$ contains a unique minimum. By the above eliminations, $min(h(y_2), h(x_2),h(\kappa_1),...,h(\kappa_k))>max(h(\varsigma_1),..., h(\varsigma_s))$ and any level sphere with height strictly between these two values is a bridge sphere satisfying the conclusions of the theorem. The case when both $E_1$ and $E_2$ have unique minima follows similarly.
\end{proof}

\begin{rmk}\label{saddlenumber}In addition, we have shown in the above proof that the number of components in $B_1 \cap \Sigma$ is one more than the number of saddles of $S$ when $S$ is taut.\end{rmk}

\begin{theorem}\label{thm:bridge}
The bridge position and the thin position for $K$ are distinct.
\end{theorem}

\begin{proof}
Assume, towards a contradiction, that bridge position and thin position for $K$ coincide. Recall that $w(K) \leq 134$, so it is enough to show that the bridge number of $K$ is at least 9, or equivalently, that $K$ intersects any bridge sphere in at least 18 points. Let $\Sigma$ be a minimal bridge sphere for $K$ and consider how $S_1$ intersects $\Sigma$. By Theorem \ref{spherenice}, we may assume that one of the 3-balls that $S_1$ bounds intersects $\Sigma$ in a collection of disks and $S_1-\Sigma$ is a collection of annuli and 2-punctured disks. Suppose the number of intersection curves $|S_1 \cap \Sigma|$ is minimal subject to these constraints. Let $c_1,...,c_n$ be the curves of $S_1 \cap \Sigma$ and let $D_1,...,D_n$ be the disjoint disks $c_1,..,c_n$ bound in $\Sigma$. As $S_1$ is c-incompressible each $D_i$ has at least 2 punctures.

\medskip

\noindent{\bf Case 1:} $n \geq 2$.

\medskip

\textbf{Claim 1:} Either $\Sigma - \cup\{D_i\}$ is c-compressible both above and below or bridge position for $K$ is not thin position.

{\em Proof:} The bridge sphere $\Sigma$ separates $S^3$ into two 3-balls $H_1$ and $H_2$. The planar surface $S_1 \cap H_1$ consists of a, possibly empty, collection of c-incompressible annuli, $A_1$, ...,$A_n$ and a, possibly empty, collection of at most two c-incompressible 2-punctured disks $E_1$ and $E_2$. We will use the convention that $\partial E_1=\partial D_1$ and $\partial E_2=\partial D_n$. Since $n \geq 2$, $S_1 \cap H_i$ consists of at least one c-incompressible annulus or $S_1 \cap H_i = E_1 \cup E_2$.  Since every properly-embedded meridional surface in $H_1-K$ with non-empty boundary is boundary compressible in $H_1-K$, we can choose $F$ to be a boundary compressing disk for $S_1 \cap H_1$ in $H_1$. In particular, we can assume that $F \cap S_1$ is a single essential arc, $\alpha$, in $S_1 \cap H_1$. Let $\partial F=\alpha \cup \beta$ where $\beta$ is contained in $\Sigma$. In particular, $\beta$ is disjoint from $S_1$ except in its boundary. If $\alpha$ is contained in some $A_i$ and $\beta$ is contained in $\Sigma - \cup\{D_i\}$, then boundary compressing $A_i$ along $F$ produces a compressing disk for $\Sigma - \cup\{D_i\}$ above $\Sigma$. If $\alpha$ is contained in some $A_i$ and $\beta$ is contained in some $D_i$, then $\alpha$ is not essential in $A_i$, a contradiction. If $\alpha$ is contained in some $E_i$ and $\beta$ is contained in $\Sigma - \cup\{D_i\}$, then boundary compressing $E_i$ along $F$ produces a cut-disk for $\Sigma - \cup\{D_i\}$ above $\Sigma$. If $\alpha$ is contained in some $E_i$ and $\beta$ is contained in $\Sigma - \cup\{D_i\}$, then boundary compressing $E_i$ along $F$ produces a cut-disk for $D_1$ or $D_n$ above $\Sigma$. We can conclude that one of $\Sigma - \cup\{D_i\}$, $D_1$ or $D_n$ is c-compressible above. Similarly, we conclude one of $\Sigma - \cup\{D_i\}$, $D_1$ or $D_n$ is c-compressible below. In particular, if $D_1$ is c-compressible above we can always find a c-disk below $\Sigma$ that is disjoint from $D_1$ since $n\geq 2$. By examining all remaining possibilities for c-disks above and below $\Sigma$ and noting that $D_1$ and $D_n$ are distinct, we concluded that either $\Sigma - \cup\{D_i\}$ is c-compressible both above and below or $\Sigma$ has c-disks that satisfy the hypotheses of Lemma \ref{cdiskbridgenotthin}. Hence, either $\Sigma - \cup\{D_i\}$ is c-compressible both above and below or bridge position of $K$ is not thin position. $\square$

\medskip

Suppose some $D_j$ is c-compressible. As $S_1$ is c-incompressible, we may assume that the c-disk is disjoint from it. By taking an innermost curve of intersection of this c-disk with all other disks $D_i$, we can find a disk $D_k$ that is c-compressible in the complement of the others. Let $\Delta$ be this c-disk. Without loss of generality, assume $\Delta$ is above $\Sigma$. By Claim 1, either $\Sigma - \cup\{D_i\}$ is c-compressible both above and below or bridge position of $K$ is not thin position. However, we have assumed that bridge position coincides with thin position, so we may assume $\Sigma - \cup\{D_i\}$ is c-compressible both above and below. Therefore, $\Delta$ and a c-disk for $\Sigma - \cup\{D_i\}$ below $\Sigma$ give a pair of c-disks for $\Sigma$ on opposite sides with disjoint boundaries. By Lemma \ref{cdiskbridgenotthin}, bridge position and thin position for $K$ are distinct, a contradiction.

By the previous argument, we may assume all $D_i$ are c-incompressible. If $B_1$ is contained on the same side of $S_1$ as the $D_i$, then each $D_i$ must have at least 20 punctures, by Theorem \ref{thm:boundess}. It follows that $\Sigma$ has at least 40 punctures so the bridge number of $K$ is at least 20.

It remains to consider the case when all $D_i$ are c-incompressible and are contained in $S^3-B_1$. As $S_2$ is essential, it must intersect $\Sigma$ and, therefore, it must intersect some $D_i$, say $D_1$. Since $S_2$ is c-incompressible, then, after isotoping $|S_2\cap D_1|$ to be minimal, an innermost curve of $S_2\cap D_1$ in $D_1$ must bound a subdisk $\Delta_1$ of $D_1$ containing at least 2 punctures. This subdisk must be c-incompressible as a c-disk for it would be a c-disk for $D_1$. If $\Delta_1$ is contained in $B_2$, it must have at least 20 punctures, by Theorem \ref{thm:boundess}, so $\Sigma$ has at least 22 punctures and the bridge number of $K$ is at least 11. We conclude that either bridge and thin position for $K$ do not coincide or all innermost curves of $D_1 \cap S_2$ bound disks in $D_1$ that have at least 2 punctures and are outside of $B_2$. If there are at least 8 such innermost curves, then $D_1$ has at least 16 punctures. As each of $D_2,...,D_n$ has at least 2 punctures and $n \geq 2$, it follows that $\Sigma$ has at least 18 punctures as desired. If there are fewer than 8 innermost curves, then a second innermost curve cobounds with some of the innermost curves a planar surface $F \subset B_2\cap D_1$ with at most 8 boundary components. A c-disk for $F$ would also be a c-disk for $D_1$ so $F$ is c-incompressible. Let $b$ be the number of boundary components of $F$ and $p$ be the number of points of intersection between $F$ and $K$. By Theorem \ref{thm:boundess}, it follows that $2-b-p = \chi(F_K)\leq -23$. However, $D_1$ meets $K$ in at least $2(b-1)$ points outside of $F$. Therefore, $D_1$ has at least $25-b+2(b-1)=23+b$ punctures. Since $b\geq 2$ and $n\geq 2$, we conclude that $\Sigma$ has at least 27 punctures in total and the bridge number of $K$ is at least 14.

\medskip

\noindent{\bf Case 2:} $n = 1$.

\textbf{Claim 2:} If $n=1$, there is an isotopy taking $S_1$ to a level sphere and adding at most one additional maximum to $h_K$.

{\em Proof:} Since $n=1$, $S_1$ is a standard round 2-sphere with no saddles, by Remark \ref{saddlenumber}. Label each point of $\{x_1,x_2,x_3,x_4\}=K\cap S_1$ with an $m$ if it is a local minimum of $h_{K\cap B_1}$ and label it with an $M$ if it is a local maximum of $h_{K\cap B_1}$. If all points of $K\cap S_1$ receive a common label, then $S_1$ is isotopic to a level sphere via an isotopy that preserves the number of maxima of $h_K$, as in Figure \ref{fig:S1level.eps}.

\begin{figure}[h]
\centering \scalebox{.5}{\includegraphics{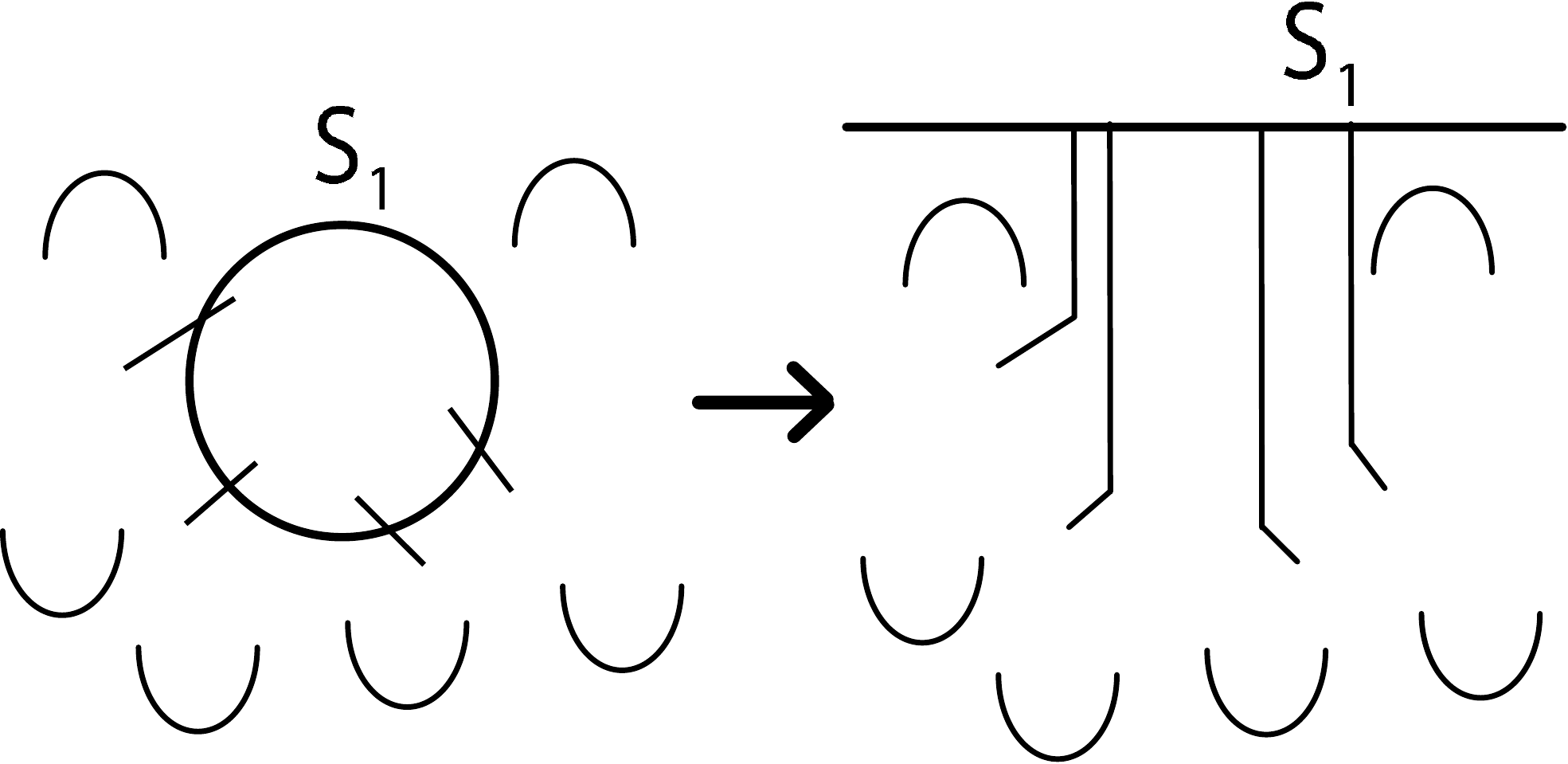}}
\caption{}\label{fig:S1level.eps}
\end{figure}

Order the points in $K \cap S_1$ in terms of increasing height so that $h(x_1)<h(x_2)<h(x_3)<h(x_4)$.

By Claim 1 of Theorem \ref{spherenice}, there is an isotopy of $K$ fixing $S_1$ and the number of maxima of $h_K$ so that after this isotopy $x_4$ receives a label of $M$. By a symmetric argument, we can assume that $x_1$ receives the label $m$.

Suppose $x_1$ and at least one additional point, say $x_2$, are labeled $m$. Poke a neighborhood of $x_4$ in $S_1$ along $K$ toward and just past the nearest maximum of $K$ causing the highest point of $K \cap S_1$ to now be labeled $m$. See Figure \ref{fig:pokeS1.eps}. After this isotopy, $S_1$ contains a single inessential saddle which can be removed as in the proof of Lemma \ref{inesssaddle}. Hence, we have isotoped $S_1$ to be a standard round 2-sphere with each of $x_1$, $x_2$ and $x_4$ receiving the label $m$ without introducing any new maxima to $h_K$. If $x_3$ is labeled $m$, then $S_1$ is isotopic to a level sphere via an isotopy that preserves the number of maxima of $h_K$, as in Figure \ref{fig:S1level.eps}. If $x_3$ is labeled $M$, then we can preform an isotopy of $K$ supported in a neighborhood of $x_3$ that introduces exactly 1 additional maximum to $h_K$ and results in $x_3$ receiving a label of $m$. Now that all points in $K \cap S_1$ receive the same label, $S_1$ is isotopic to a level sphere via an isotopy that preserves the number of maxima of $h_K$, as in Figure \ref{fig:S1level.eps}. By a symmetric argument, if $x_4$ and at least one additional point receive a label of $M$, then there is an isotopy taking $S_1$ to a level sphere and adding at most one additional maximum to $h_K$. $\square$

\begin{figure}[h]
\centering \scalebox{.5}{\includegraphics{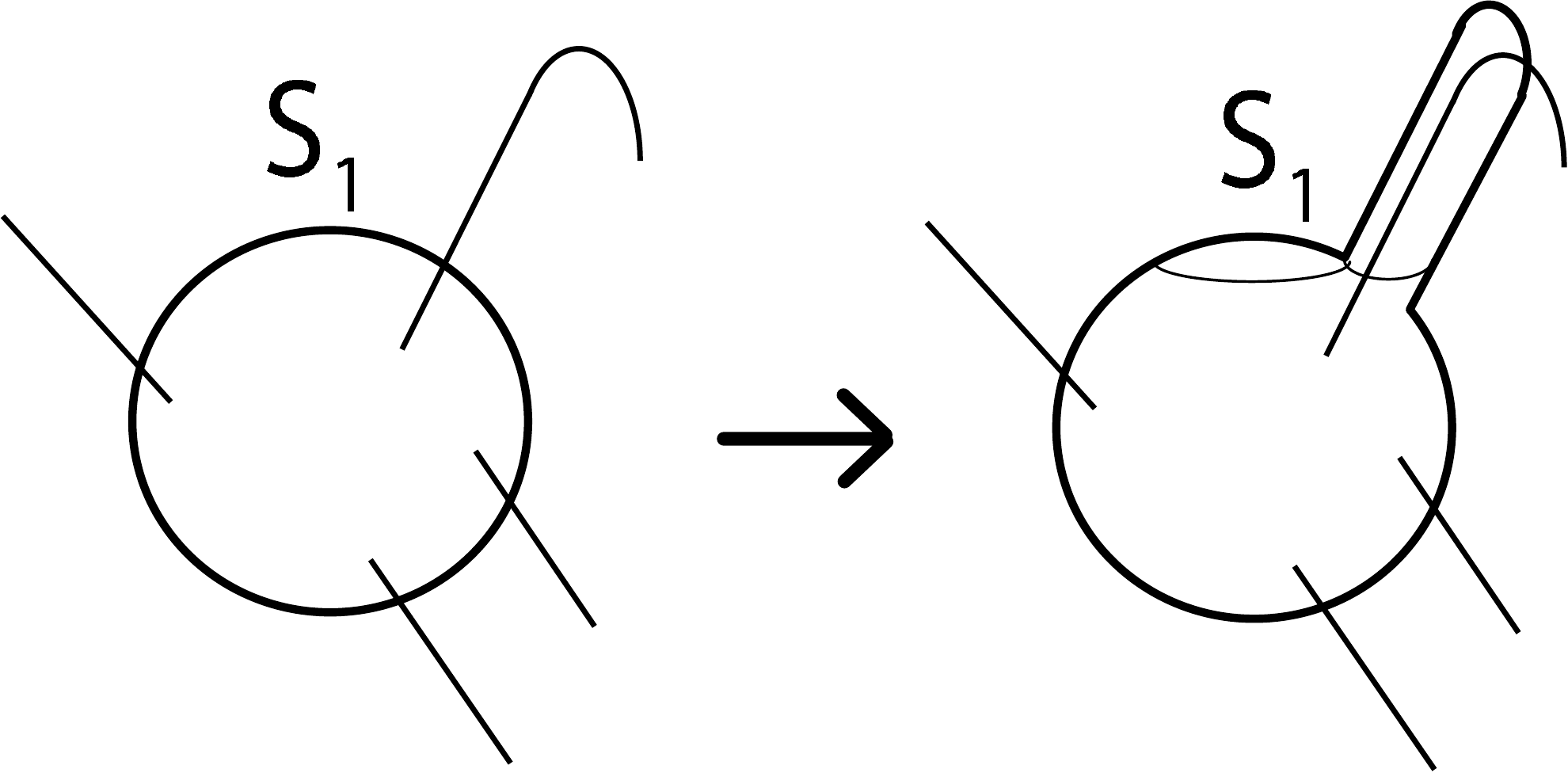}}
\caption{}\label{fig:pokeS1.eps}
\end{figure}

Use Claim 2 to isotope $S_1$ to be level at the cost of introducing at most one additional maximum to $h_K$. Suppose $B_1$ is contained below $S_1$ and the complement of $B_1$ is contained above $S_1$. By Property \ref{pro:14puncturedbridge}, the tangle below $S_1$ has at least 7 minima and therefore at least 5 maxima. By Proposition \ref{pro:distance}, the tangle above $S_1$ has at least 5 maxima and, therefore, in this position $K$ has at least 10 maxima. As at most one maximum was added, $K$ has bridge number at least 9.

\end{proof}

\section{Some Useful Lemmas}\label{sec:lemmas}

Recall that $K'$ denotes an embedding of $K$ that minimizes width. In this section, we establish additional restrictions on thin and thick spheres for $K'$.

\begin{lemma} \label{518}
If $G$ is an essential, 6-punctured, thin sphere for $K'$, then there are thick spheres of width 10 below and above $G$. These spheres are not necessarily adjacent to $G$.

\end{lemma}

\begin{proof}

We will show that there is a thick sphere of width 10 above $G$. If there is a thin sphere of width 8 or more above $G$, the result is clear so suppose all thin spheres above $G$, if there are any, have width 4 or 6. Let $P$ be the highest thin sphere above $G$, possibly $P=G$. If $w(P)=4$, then $P$ must be one of $S_1$, $S_2$ or $S_3$.  The result then follows by Property \ref{pro:distance} or Property \ref{pro:12puncturedbridge}. If $w(P)=6$, by Corollary \ref{cor:thinabove}, $P$ is c-incompressible above and, therefore, by Proposition \ref{prop:sixpunctured}, $P$ is cut-compressible below. It follows that both $B_1$ and $B_2$ are below $P$. In this case, the thick sphere above $P$ has width at least 10 by Property \ref{pro:6punctured}.

\end{proof}
\begin{lemma} \label{lem:useful}

Suppose $P$ and $P'$ are two adjacent thin spheres for $K'$ so that $4 \leq w(P), w(P')\leq 10$. Suppose $D^c$ is a c-disk for $P$ lying between them so that $\bdd D^c$ bounds a three or four punctured disk $\Delta$ in $P$. Then the sphere $F=D^c \cup \Delta$ is essential. Furthermore, if $\Sigma$ is the thick sphere between $P$ and $P'$ and $F$ does not separate $P$ and $P'$ then:

\begin{enumerate}
\item if $w(P)\geq 6$ and $w(P')\geq 6$, then $w(\Sigma)\geq 14$,
\item if $w(P)\geq 6$ and $w(P')=4$, then $w(\Sigma)\geq 12$,
\item if $w(P)=8$ and $w(P')=4$ and $D^c$ is a compressing disk, then $w(\Sigma)\geq 14$.
\end{enumerate}
\end{lemma}

\begin{proof}

Suppose $P$ and $P'$ are two adjacent thin spheres for $K'$ and $D^c$ is a c-disk for $P$ lying between them so that $\bdd D^c$ bounds a three or four punctured disk $\Delta$ in $P$. Consider the four punctured sphere $F=\Delta \cup D^c$ and suppose $F$ has a compressing disk $E$. By using an outermost arc argument we may assume that $\bdd E \subset \Delta$. As $F$ has only 4 punctures, $\bdd E$ bounds a twice punctured disk $\delta \subset \Delta$. Then $\delta \cup E$ is a twice punctured sphere. As $K$ is prime, the strand of the knot with both endpoints in $\delta$ can be isotoped to lie in $\delta$. It follows that this strand can be isotoped to lie just past the thin sphere $P$. This isotopy either eliminates a maximum and an adjacent minimum or slides a maximum below a minimum thus decreasing $w(K')$. As $K'$ is in thin position, this is a contradiction.

By Theorem \ref{thm:diskisvertical}, we may assume that $D^c$ is vertical. Let $B$ be the ball bounded by $F$ that is disjoint from $P'$ and let $E=B\cap \Sigma$. By Lemma \ref{lem:newbridge} $E$ together with the possibly once punctured disk $
\bdd E$ bounds in $D^c$ is a bridge sphere for the tangle $K'\cap B$.
As $F$ is an incompressible 4 punctured sphere, it must be parallel to one of $S_1, S_2$ or $S_3$. By one of Property \ref{pro:distance} or Property \ref{pro:6punctured}, the width of any bridge sphere for  $K' \cap B$ is at least 10 and therefore $E$ has at least 9 punctures. In addition, let $\tau_1,...\tau_n$ be the strands of $K$ between $P$ and $P'$ that are disjoint from $B$. Then $|\Sigma \cap \tau_i|\geq 1$ if $\tau_i$ has one of its endpoints in $P$ and one in $P'$ and $|\Sigma \cap \tau_i|\geq 2$ if $\tau_i$ has both of its endpoints on the same sphere. It is easy to check that the conclusions of the lemma are satisfied in all three cases, see Figure \ref{fig:useful}.

\begin{figure}[h]
\centering \scalebox{.4}{\includegraphics{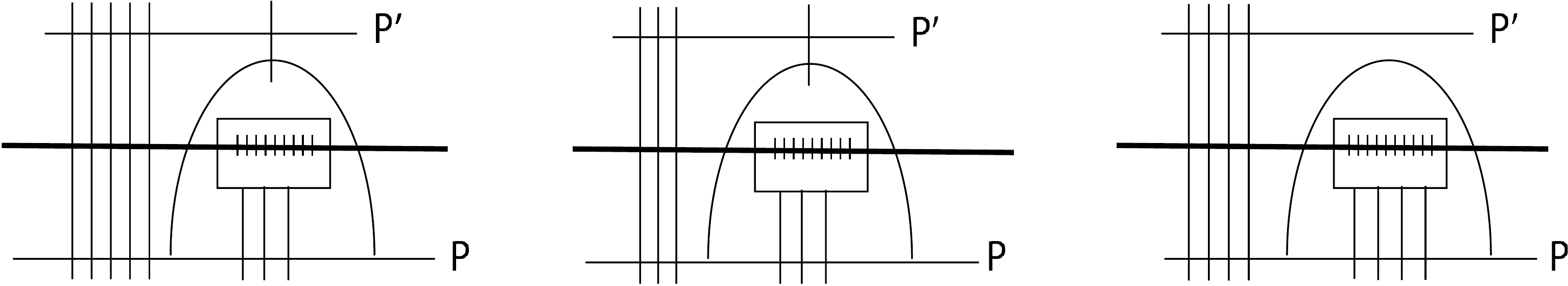}}
\caption{}\label{fig:useful}
\end{figure}
 \end{proof}

 \begin{rmk}
 Note that if $P$ is a six-punctured sphere with a cut-disk disjoint from all other thin spheres or if $P$ is an eight-punctured sphere with a compressing disk disjoint from all other thin spheres, then we can always choose $\Delta$ so that the sphere $F$ which is the union of the c-disk and $\Delta$ is 4-punctured and does not separate $P$ and its adjacent thin sphere.
 \end{rmk}

\begin{lemma}\label{lem:final}

Suppose $P$ and $P'$ are two adjacent thin spheres for $K'$ with $P'$ above $P$ and $P$ is incompressible but has a cut-disk $D^c$ above it and disjoint from $P'$. Suppose furthermore that $\bdd D^c$ bounds a 5 punctured disk in $P$ and the 6 punctured sphere $G$ which is the union of $D^c$ and this disk does not separate $P$ and $P'$. Then the thick sphere $\Sigma$ between $P$ and $P'$ has width at least $w(P)+4$.

In the special case where $G=S_1\sharp_{\alpha}S_2$ and $\alpha$ is a strand that intersects $B_T$ and $B_1$ is not contained between $P$ and $P'$, then $\Sigma$ has at least $w(P)+8$ punctures.

\end{lemma}

\begin{proof}

Let $B_G$ be the ball bounded by $G$ and disjoint from $P$ and let $T=K' \cap B_G$.
By Theorem \ref{thm:diskisvertical}, we may assume that $G \cap \Sigma=\beta$ where $\beta$ is a single essential simple closed curve in $\Sigma$. By Lemma \ref{lem:newbridge}, the sphere $R$ obtained by cut-compressing $\Sigma$ along the cut-disk that $\beta$ bounds in $D^c$ is a bridge sphere for $T$. A compressing disk for $G$ would result in a compressing disk for $P$, so we conclude that $G$ is incompressible and, therefore, it is one of the spheres in Proposition \ref{prop:sixpunctured}. In particular either $B_1$ and $B_2$ are contained in $B_G$ or they are both disjoint from it.

There are two cases to consider. If $B_1$ and $B_2$ are disjoint from $B_G$, then by Property \ref{pro:6punctured} it follows that $R$ has at least 10 punctures. If $B_1$ and $B_2$ are contained in $B_G$, we can apply Theorem \ref{thm:boundbridge} with $N=B_G$ with the bridge sphere $R$ and $M=B_1$ with the bridge sphere $\Sigma_1$. It again follows that $R$ has at least 10 punctures.

As each strand that has an endpoint in $P'$ must intersect $\Sss$ at least once, it follows that $|\Sss\cap K|\geq|R\cap K| +|P \cap K|-5\geq  w(P)+4$ as desired.

If $G=S_1\sharp_{\alpha} S_2$ and $B_1$ is not contained between $P$ and $P'$, then, by Property \ref{pro:6punctured}, $|R\cap K|\geq 13$, so $|\Sss\cap K|\geq|R\cap K| +|P \cap K|-5\geq  w(P)+8$.

\end{proof}

\section{$w(K\#trefoil)=w(K)$}\label{sec:134}

In this section, we show that the width of $K$ is $134$. This completes the proof that $w(K \# trefoil)=w(K)<w(K)+w(trefoil)-2$, thus disproving the knot width additivity conjecture. Let $K'$ be a knot isotopic to $K$ that is in thin position. The argument is separated into two parts depending on whether the minimal width thin sphere for $K'$ is cut-compressible or not.

\begin{theorem}\label{thm:thinnestccompressible}
If the thinnest thin sphere for $K'$ is cut-compressible, then $w(K') \geq 138$.

\end{theorem}

\begin{proof}

As all essential 4-punctured spheres in the complement of $K$ are c-incompressible, it follows that the thin sphere of lowest width has at least 6 punctures. In particular, if $P$ is a thin sphere for $K'$, $w(P)\geq 6$.

\medskip

\noindent \textbf{Case A:} Suppose first that $K'$ has a compressible thin sphere.

\medskip

\textbf{Claim:} There exist adjacent thin spheres $P$ and $P'$ so that $P$ has a compressing disk on the same side as $P'$ but disjoint from it.

{\em Proof:} Let $P''$ be a compressible thin sphere for $K'$ with compressing disk $D'$. Consider the intersection of $D'$ with the collection of all thin spheres for $K'$. Assume that $D'$ has been isotoped so this intersection is minimal. In particular, every curve of intersection is essential in the corresponding thin sphere. Let $\sigma$ be an innermost curve of intersection in $D'$. Let $P$ be the thin sphere containing $\sigma$. Then $P$ has a compressing disk $D \subset D'$ disjoint from all other thin spheres. By Corollary \ref{cor:thinabove}, there are other thin spheres on the same side of $P$ as $D$. In particular, we can choose $P'$ to be such a sphere so that $P$ and $P'$ are adjacent thin spheres.$\square$

\medskip

Let $P$ and $P'$ be the two spheres guaranteed by the claim. Assume $P$ is below $P'$. By Theorem \ref{thm:sixincomp}, it follows that $w(P)\geq 10$. If $w(P)\geq12$, then $w(K')\geq 138$, so suppose $w(P)=10$. By Corollary \ref{cor:altthin}, compressing $P$ along $D$ yields a copy of $P'$. As $w(P')\geq 6$ and $K$ is prime, it follows that $w(P')= 6$ and $\bdd D$ bounds a 4-punctured disk $\Delta \subset P$ so that $D \cup \Delta$ is a sphere that does not separate $P$ and $P'$. By Lemma \ref{lem:useful} part (1), the thick sphere between $P$ and $P'$ has width at least 14. By Lemma \ref{518}, there is a thick sphere of width at least 10 above $P'$ and one of width at least 12 below $P$. It follows that $w(K')\geq 152$.

\medskip

\noindent \textbf{Case B:} All thin spheres for $K'$ are incompressible.

 Let $P$ be the thinnest thin sphere for $K'$ and suppose it is cut-compressible above. By Corollary \ref{cor:thinabove}, $K'$ must have at least one other thin sphere, $P'$, above $P$. If $w(P')\geq 12$, then $w(K')\geq 138$, so suppose every thin sphere for $K'$ has width at most 10.

Let $D^c$ be the cut-disk for $P$. By taking an innermost curve of intersection of $D^c$ with the union of all other thin spheres for $K'$, we can find a thin sphere that has a cut-disk which is disjoint from all other thin spheres. Let $P'$ be the thin sphere of lowest width amongst all thin spheres for $K'$ that has this property and let $D'^c$ be its cut-disk. We will assume $D'^c$ is above $P'$. Let $P''$ be the thin sphere adjacent to $P'$ above it. By Corollary \ref{cor:thinabove}, this sphere exists and $w(P'')<w(P')$.

\medskip

\noindent \textbf{Case 1: $w(P')=6$.}  In this case $w(P'')\leq 4$, by Corollary \ref{cor:thinabove}. Hence, this case does not satisfy the hypotheses of the theorem at hand.

\medskip

\noindent \textbf{Case 2: $w(P')=8$.} By Corollary \ref{cor:altthin}, cut-compressing $P'$ along $D^c$ yields a copy of $P''$. As $w(P'')$ in this case must be 6, it follows that $\bdd D^c$ bounds a three punctured disk in $P'$ so that the union of $D^c$ and this disk is a sphere that does not separate $P'$ and $P''$. By Lemma \ref{lem:useful} part (1), the thick sphere between $P'$ and $P''$ has width at least 14. By Lemma \ref{518}, there is a thick sphere of width at least 10 above $P''$ and so $w(K')\geq 148$.

\medskip

\noindent \textbf{Case 3: $w(P')=10$.} There are three subcases to consider:

{\em Case 3a:} Suppose that $\bdd D^c$ bounds a three punctured disk in $P'$ so that the union of $D^c$ and this disk is a sphere that does not separate $P'$ and $P''$. By Lemma \ref{lem:useful} part (1), the thick sphere between $P'$ and $P''$ has width at least 14. If $w(P'')=6$, then, by Lemma \ref{518}, there is a thick sphere of width at least 10 below it and so $w(K')\geq 152$. If $w(P'')= 8$, then $w(K')\geq 138$.

{\em Case 3b:} Suppose that $\bdd D^c$ bounds two five-punctured disks in $P'$. As compressing $P'$ along $D^c$ yields a copy of $P''$, $w(P'')=6$. By Lemma \ref{lem:final}, the thick sphere between $P'$ and $P''$ has width at least 14. By Lemma \ref{518}, the thick sphere above $P''$ has width at least 10. It follows that $w(K')\geq 138$.

{\em Case 3c:} Suppose that $\bdd D^c$ bounds a seven-punctured disk in $P'$ so that the union of $D^c$ and this disk is a sphere that does not separate $P'$ and $P''$. By Theorem \ref{thm:diskisvertical}, we may assume that $D^c$ is vertical. By Corollary \ref{cor:altthin}, cut-compressing $P'$ along $D^c$ gives a copy of $P''$. Hence, $w(P'')=4$, contradicting the assumption.

\end{proof}

\begin{theorem}\label{thm:incomp}
If the thinnest thin sphere for $K$ is cut-incompressible, then $w(K') \geq 134$. Moreover $w(K')=134$ only if $K'$ is as in Figure \ref{fig:counterexample}.
\end{theorem}

\begin{proof}
Let $P$ be the thinnest thin sphere for $K'$ and note that, by Theorem \ref{thm:thinincomp}, $P$ is incompressible. As $P$ is cut-incompressible, by Proposition \ref{prop:noc-essential}, $P$ either has 4 punctures or at least 14 punctures. In the second case, $w(K') \geq 158$ so we may assume that $P$ is one of $S_1$, $S_2$ or $S_3$.

By symmetry, whenever $S_3$ is a thin sphere we may assume that the tangles above and below it have identical width functions. In particular, we may assume there are as many thin spheres above $S_3$ as there are below it and if $S_1$ is also a thin sphere then so is $S_2$.

\medskip

\noindent {\bf Case 1:} $K'$ has exactly one thin sphere, $P$.

Suppose $P$ is isotopic to $S_1$ and, without loss of generality, suppose that $B_1$ is below it. By Property \ref{pro:distance}, the thick sphere below $P$ has width at least 10 and, by Property \ref{pro:14puncturedbridge}, the thick sphere above it has width at least 14. It follows that $w(K')\geq 140$. Similarly, the same width bound follows if $P$ is isotopic to $S_2$.

Suppose then that $P$ is isotopic to $S_3$. By Property \ref{pro:12puncturedbridge}, the thick surfaces above and below $P$ have width at least 12. In this case, $w(K') \geq 136$ as desired.

\medskip

\noindent {\bf Case 2:} Neither $S_1$ nor $S_2$ is a thin sphere and there are at least 2 thin spheres.

In this case, we may assume $P=S_3$ and any other thin spheres have width at least 6 and, therefore, are c-compressible. By hypothesis, there is at least one other thin sphere, say above $P$, and so by symmetry there is also a thin sphere below $P$. Let $D'_a$ be c-disk for some thin sphere above $P$. As $P$ is cut-incompressible, $D'_a \cap P=\emptyset$. By taking an innermost curve of intersection of $D'_a$ with the collection of all other thin spheres for $K'$, we can find a cut or compressing disk $D_a$ for some thin sphere $P_a$ above $P$ that is disjoint from all other thin spheres for $K'$. Similarly, we can find a cut or compressing disk $D_b$ for some thin sphere $P_b$ below $P$ that is disjoint from all other thin spheres for $K'$. In addition, we can assume that either $D_a$ and $D_b$ are compressing disks or $P_a$ and $P_b$ are incompressible. By symmetry, we may assume $w(P_a)=w(P_b)$.

If $w(P_a)=w(P_b)\geq 10$, then $w(K)\geq180$. Hence, we can assume that $w(P_a)=w(P_b)\leq 8$. If $w(P_a)=w(P_b)=6$, then, by Lemma \ref{lem:useful} part (2), the thick spheres $T_a$ and $T_b$ that intersect $D_a$ and $D_b$ have widths at least 12. In this case, $w(K')\geq 164$. If $w(P_a)=w(P_b)=8$ and $D_a$ and $D_b$ are compressing disks, then, by Lemma \ref{lem:useful}, $w(T_a), w(T_b) \geq 12$ so $w(K')\geq 172$. Therefore, we may assume that  $w(P_a)=w(P_b)=8$, $P_a$ and $P_b$ are incompressible and $D_a$ and $D_b$ are cut-disks. In this case, either by Lemma \ref{lem:useful} or by Lemma \ref{lem:final}, $w(T_a), w(T_b)\geq 12$. We conclude that $w(K')\geq 172$.

\medskip

\noindent {\bf Case 3:} Exactly one of $S_1$ or $S_2$ is a thin sphere and there are at least two thin spheres.

Without loss of generality, we may assume that $P=S_1$ is a thin sphere and $B_1$ is below $P$. Note that in this case $S_3$ cannot be a thin sphere because $K$ is symmetric with respect to $S_3$ and this would imply that both $S_1$ and $S_2$ are thin spheres.

\medskip

\textbf{Claim 1:} There is a thick sphere of width at least 10 below $P$.

{\em Proof:} If there aren't any thin spheres below $P$, there is a thick sphere of width at least 10, by Property \ref{pro:distance}. If there is a thin sphere of width 8 or more, the result follows immediately. Suppose all thin spheres below $P$ have width at most 6. However, all such spheres are incompressible, by Theorem \ref{thm:sixincomp}, and there are no such spheres in $B_1$, by Proposition \ref{prop:nonein}. $\square$

\medskip

\textbf{Claim 2:} If there is a compressible thin sphere $P'$ above $P$ and $w(P')\leq 8$, then $w(K')\geq 154$.

{\em Proof:} Suppose that $K'$ has a thin sphere $P'$ that is compressible. By Theorem \ref{thm:sixincomp}, we may assume that $w(P')=8$. Let $D$ be its compressing disk. Then $D$ is disjoint from all thin spheres of width less than 8 as all such spheres are incompressible. Therefore, by taking the intersection of $D$ with all thin spheres for $K'$ and rechoosing $P'$ we may assume that $D$ is disjoint from all other thin spheres. If this new $P'$ has the property that $w(P')\geq 10$, then there are two distinct thin spheres above $P$ and $w(K')\geq 154$. Hence, we can assume that $w(P')=8$ and $D$ is a compressing disk for $P'$ contained between consecutive thin spheres $P'$ and $P''$. If $\partial D$ bounds a 2-punctured disk in $P'$, then this disk together with $D$ cobound a 3-ball containing an unknotted arc. Such a 3-ball gives rise to an isotopy that thins $K'$. Hence, we can assume that $\partial D$ bounds a 4-punctured disk to each side in $P'$. By Corollary \ref{cor:altthin}, compressing $P'$ along $D$ gives rise to a copy of $P''$. Since $\partial D$ bounds 4-punctured disks to each side in $P'$, then $P''$ is a 4-punctured sphere. By Lemma \ref{lem:useful} part (3), the thick sphere intersecting $D$ has width at least 14. Then $w(K')\geq 158$ as desired. $\square$

\medskip

{\bf Subcase 3A:} Suppose first that $K'$ has no thin spheres above $P$.  By Property \ref{pro:14puncturedbridge}, there is thick sphere of width at least 14 above $P$. By hypothesis, there is a thin sphere below $P$. By Theorem \ref{thm:boundess} this sphere cannot have width 4 or 6 as all such spheres are incompressible.  Therefore the sphere must have width at least 8 and so $w(K')\geq 158$.

{\bf Subcase 3B:} Suppose that $K'$ has exactly one other thin sphere $P'$ above $P$.

If $w(P')=4$, then $P'$ must be $S_3$. In this case, as $K'$ is symmetric with respect to $S_3$, it follows that $S_2$ is also a thin sphere, contradicting the hypothesis of this case. Therefore, $w(P')\geq 6$. By Claim 2, we may assume that $P'$ is incompressible. If $P'$ is c-incompressible, then $P'$ meets $K$ in at least 14 points, by Proposition \ref{prop:noc-essential}. Hence, we can assume that $P'$ is cut-compressible. The cut-disk for $P'$ is disjoint from $P$ and lies below $P'$, by Corollary \ref{cor:thinabove}.

If $w(P')=6$, then the thick sphere between $P$ and $P'$ has width at least 12, by Lemma \ref{lem:useful} part (2). By Lemma \ref{518}, it follows that the thick sphere above $P'$ has width at least 10 and, therefore, $w(K')\geq 146$. If $w(P') \geq 10$, then $w(K')\geq 136$ as desired.

Suppose then that $w(P')=8$, $P'$ is incompressible and has a cut-disk $D^c$. As $P'$ is the only thin sphere above $P$, by Corollary \ref{cor:thinabove}, $D^c$ must be below it and, by Theorem \ref{thm:diskisvertical}, we can assume that it is vertical. By Corollary \ref{cor:altthin}, cut-compressing $P'$ along $D^c$ results in a copy of $P$ and, therefore, $\bdd D^c$ bounds a 5-punctured disk $\Delta$ in $P'$ so that the 6-punctured sphere $G=D^c \cup \Delta$ does not separate $P$ and $P'$.

A compressing disk for $G$ would result in a compressing disk for $P'$ and $w(K')\geq 154$, by Claim 2. Hence, we may assume that $G$ is incompressible and, therefore, it is one of the spheres classified in Proposition \ref{prop:sixpunctured}. In particular, $G$ does not separate $B_1$ and $B_2$.
Let $B_G$ be the ball bounded by $G$ disjoint from $P$. As $B_1$ is disjoint from $B_G$, so is $B_2$. Therefore, there are three possibilities to consider: $B_2$ is below $P'$ but outside of $B_G$, $B_2$ is above $P'$ or $B_2$ intersects $P'$.

It is clear that $B_2$ cannot be contained below $P'$ and be disjoint from $B_G$ as that would imply that the essential surface $S_2$ is completely contained in the product region between $P$ and the 4-punctured sphere resulting from cut-compressing $P'$ along $D^c$.

If $B_2$ is completely contained above $P'$, then let $\Sigma$ be the thick surface for $K$ above $P'$. By Theorem \ref{thm:boundbridge}, it follows that either $\Sigma$ has at least 26 punctures or $\Sigma$ is isotopic to $\Sigma_2$ with possibly some tubes attached. If no tubes are attached, then $\Sigma$ is parallel to $\Sigma_2$ and $S_2$ is an essential sphere completely disjoint from the bridge sphere of the tangle lying above $P'$, which is not possible. Therefore, at least one tube is attached. We conclude that $\Sigma$ has at least 12 punctures. By Lemma \ref{lem:final}, the thick sphere below $P'$ also has width at least 12. Hence, $w(K)\geq 154$.

Suppose that $B_2$ intersects $P'$. We have already assumed that $P$ is isotopic to $S_1$ and we have established that cut-compressing $P'$ along $D^c$ produces $P$ and a incompressible 6-punctured sphere, $G$. Since tubing along a strand of $K$ is the inverse operation to cut-compressing, $P'$ is isotopic to $G\sharp_{\beta}S_1$. However, both $S_1$ and $G$ can be isotoped to be disjoint from $B_2$. Hence, $\beta$ intersects $B_2$, as otherwise $B_2$ could be isotoped to be disjoint from $P'$.

There are several cases to consider.

If $G=S_1\sharp_{\epsilon} S_3$ or $G=S_2\sharp_{\epsilon} S_3$ where $\epsilon$ is the strand not passing through $B_T$, then $G$ is compressible which is not possible.

Suppose $G=S_1\sharp_{\gamma} S_3$ where $\gamma$ is a strand passing through $B_T$, see the first schematic of Figure \ref{fig:A}. Then $P'=G\sharp_{\alpha} S_1$. The strand $\alpha$ may contain $\epsilon$ or it may not. In either case, the tangle $T'$ above $P'$ can be obtained from the tangle $\mathcal{T}$ contained on one side of $S_3$ by replacing one of the strands with three parallel strands as in Figure \ref{fig:A}.  By Property \ref {pro:12puncturedbridge}, every bridge surface for $\mathcal{T}$ has at least 12 punctures. Since each of the new strands must intersect the bridge surface in at least two points, the thick surface above $P'$ has at least 16 punctures. By Lemma \ref{lem:final}, the thick sphere directly below $P'$ has at least 12 punctures, so $w(K)\geq 210$.

\begin{figure}[h]
\centering \scalebox{.6}{\includegraphics{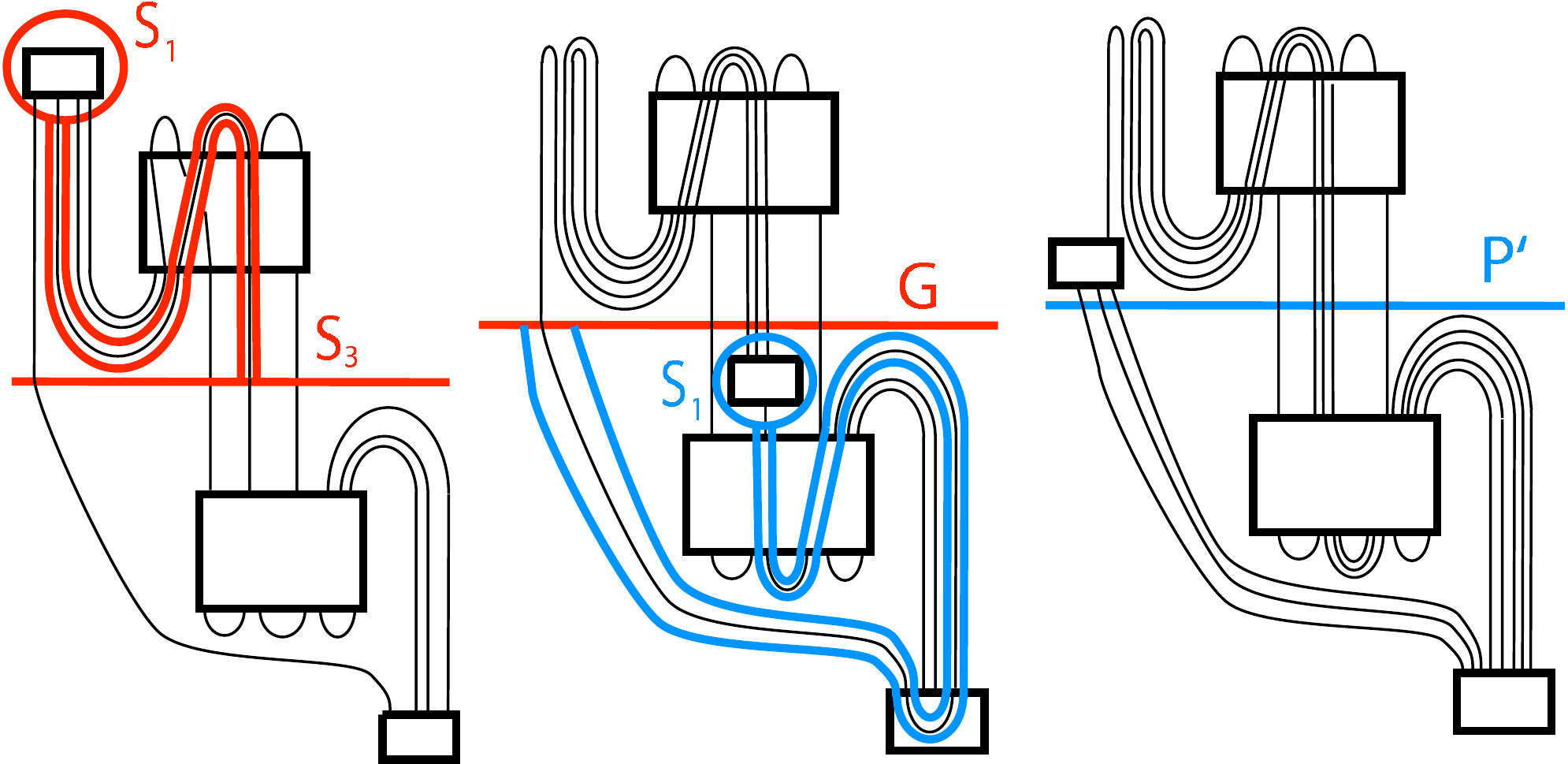}}
\caption{}\label{fig:A}
\end{figure}

Suppose $G=S_1\sharp_{\alpha} S_2$ where $\alpha$ is a strand passing through $B_T$. In this case, by Lemma \ref{lem:final}, the bridge sphere directly below $P'$ has width at least 16, so $w(K)\geq 188$.

Suppose $G=S_1\sharp_{\epsilon} S_2$ where $\epsilon$ is the strand not passing through $B_T$, i.e, $G=S_4$. Then the tangle $\mathcal{T}'$ above $P'$ can be obtained from the tangle $\mathcal{T}_2$ by replacing one of the strands with three parallel strands, see Figure \ref{fig:B}. As any bridge surface for $\mathcal{T}_2$ has at least $10$ punctures and each of the two additional strands has to intersect the thick surface for $\mathcal{T}'$ at least twice, it follows that the thick surface above $P'$ has at least 14 punctures. By Lemma \ref{lem:final}, the thick sphere directly below $P'$ also has at least 12 punctures, so $w(K)\geq 180$.

\begin{figure}[h]
\centering \scalebox{.6}{\includegraphics{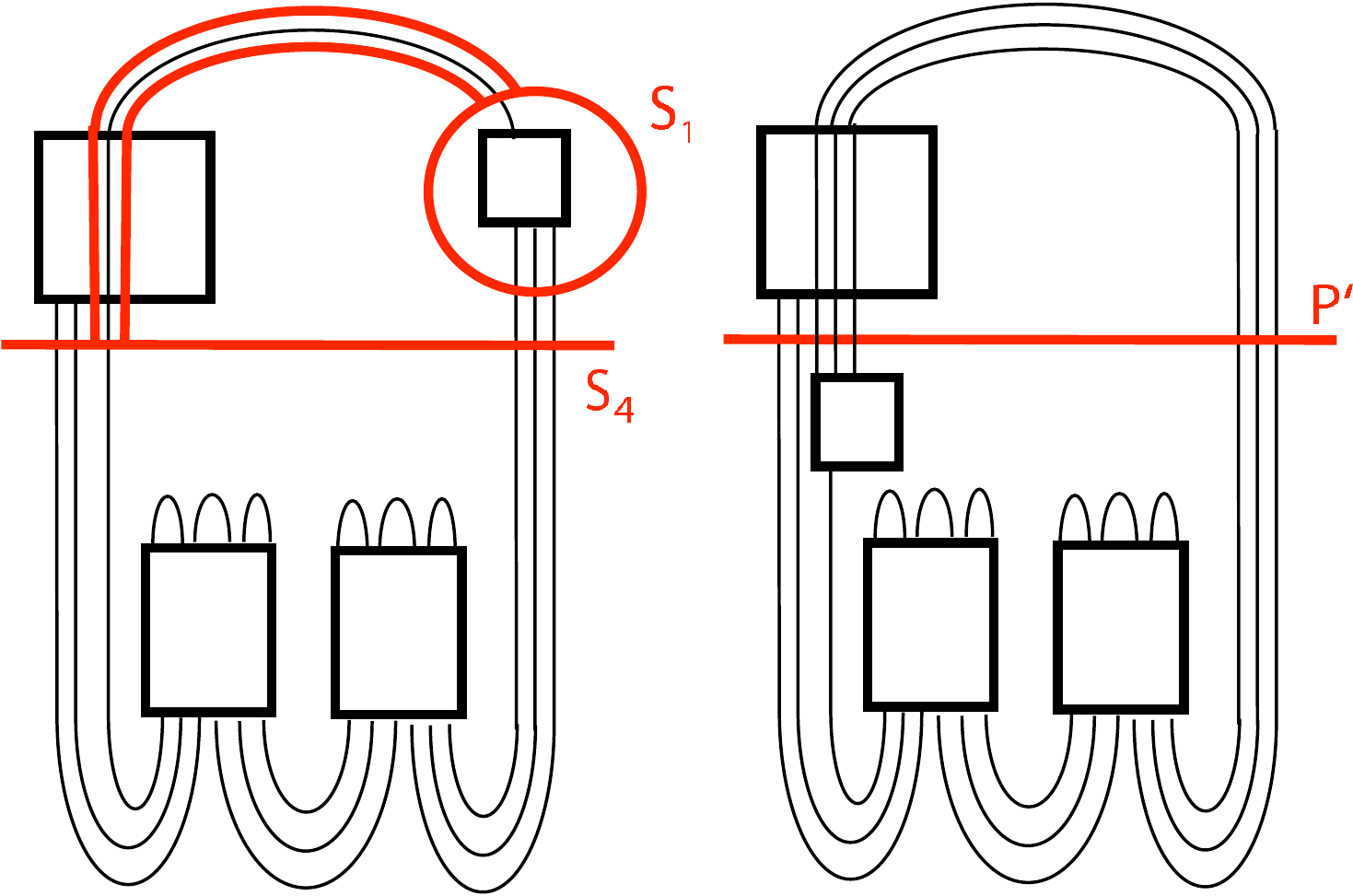}}
\caption{}\label{fig:B}
\end{figure}

Finally, suppose that $G=S_2\sharp_{\alpha} S_3$ where $\alpha$ is a strand passing through $B_T$, see Figure \ref{fig:C}.  Let $\mathcal{T}'$ be the 4 strand tangle above $P'$. By deleting two of the four strands of $\mathcal{T}'$, we can obtain the tangle $\mathcal{T}_2$. Therefore, by Property \ref{pro:distance} the thick surface above $P'$ intersects two of the strands in $\mathcal{T}'$ in at least 10 points and each of the other two strands in at least 2 points each. Hence, this thick sphere meets $K$ in at least 14 points. As before, the thick sphere below $P'$ has width at least 12, so $w(K)\geq 180$.

\begin{figure}[h]
\centering \scalebox{.6}{\includegraphics{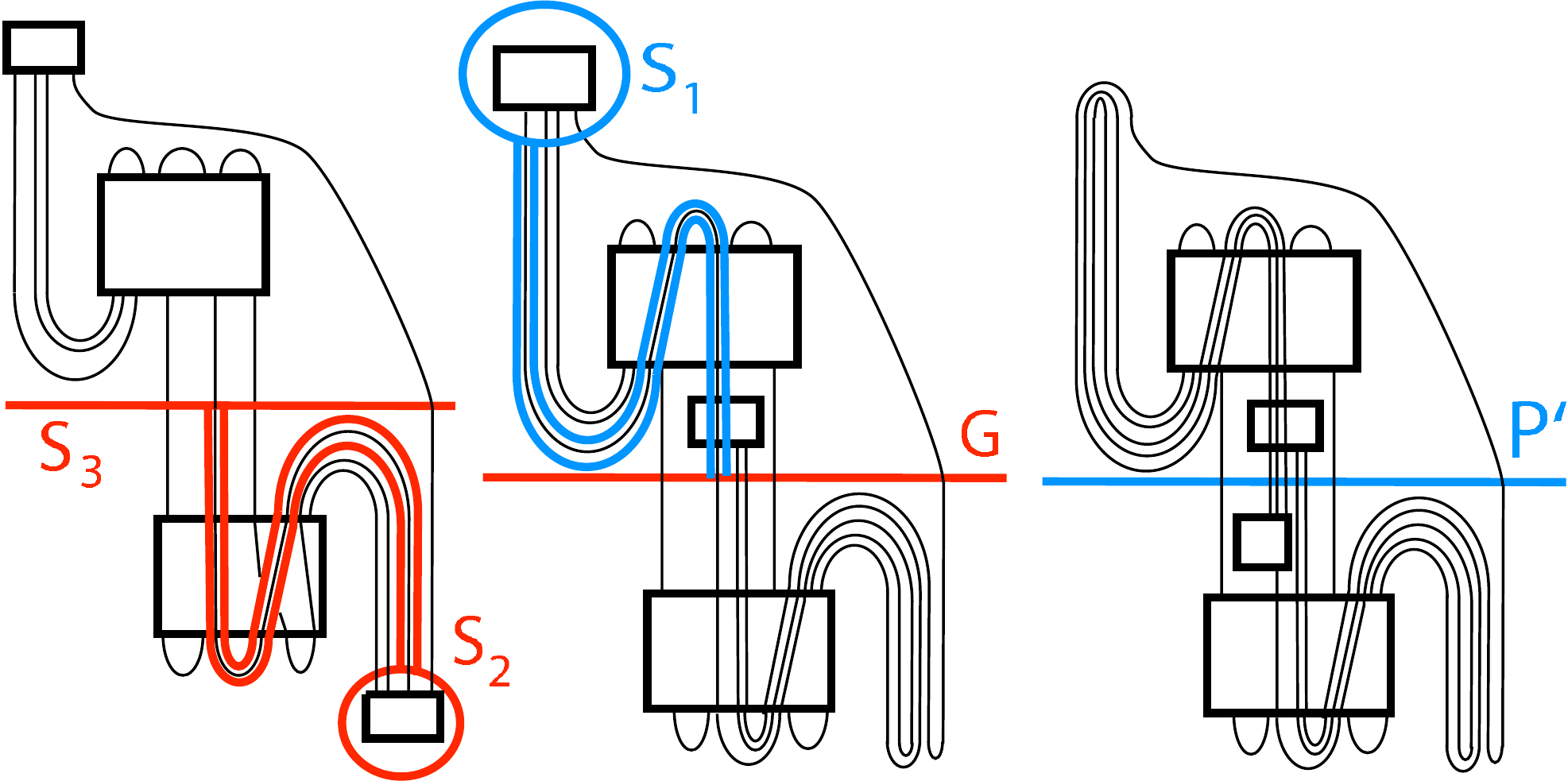}}
\caption{}\label{fig:C}
\end{figure}

\textbf{Subcase 3C:} Suppose that there are at least 2 thin spheres above $P$.

If at least one of these spheres has width at least 10, then $w(K')\geq 146$ so we may assume that all thin spheres have width at most 8.
By Claim 2, we may assume that all thin spheres for $K'$ are incompressible. By taking the intersection of a cut-disk with the union of all thin spheres, we can find a cut-disk $D^c$ for some thin sphere $P'$ that is disjoint from all other thin spheres. Let $P''$ be the thin sphere adjacent to $P'$ on the same side of $P'$ as $D^c$. This sphere exists, by Corollary \ref{cor:thinabove}. If $w(P')=6$, then, by Lemma \ref{lem:useful} part (2), the thick sphere between $P'$ and $P''$ has width at least 12. It follows that $w(K')\geq 142$. If $w(P')=8$, then, by Lemma \ref{lem:useful} or Lemma \ref{lem:final}, the thick sphere between $P'$ and $P''$ has width at least 12 and, so $w(K')\geq 142$.

\medskip

\noindent {\bf Case 4:} $S_1$ and $S_2$ are both thin spheres.

By Claim 1 of Case 3, it follows that there are thick spheres of width at least 10 below $S_1$ and above $S_2$. If there is a thin sphere between $S_1$ and $S_2$, it must either have width 4 or it must have width at least 8 as all 6-punctured thin spheres do not separate $B_1$ from $B_2$.  If there is a thin sphere of width 8, then $w(K)\geq 152$. Suppose there is a thin sphere of width 4 between $S_1$ and $S_2$. There can be only one such sphere and it is necessarily $S_3$. Let $T_a$ and $T_b$ be the thick spheres directly above and below $S_3$. By symmetry, $w(T_a)=w(T_b)$. If $w(T_a)\geq 8$, then $w(K)\geq 140$. If $w(T_a)=w(T_b)=6$, then $K$ has exactly one minimum and one maximum between $S_1$ and $S_3$ and similarly between $S_3$ and $S_2$. This implies that the tangle $T$ between $S_1$ and $S_3$ has a bridge sphere of width 8, contradicting Property \ref{pro:10puncturedbridge}.

Suppose that $S_1$ and $S_2$ are adjacent thin spheres. By Property \ref{pro:10puncturedbridge}, the width of the thick sphere between them is at least 10 and so $w(K')\geq 134$ as desired.

 \end{proof}

\begin{figure}[h]
\centering \scalebox{.4}{\includegraphics{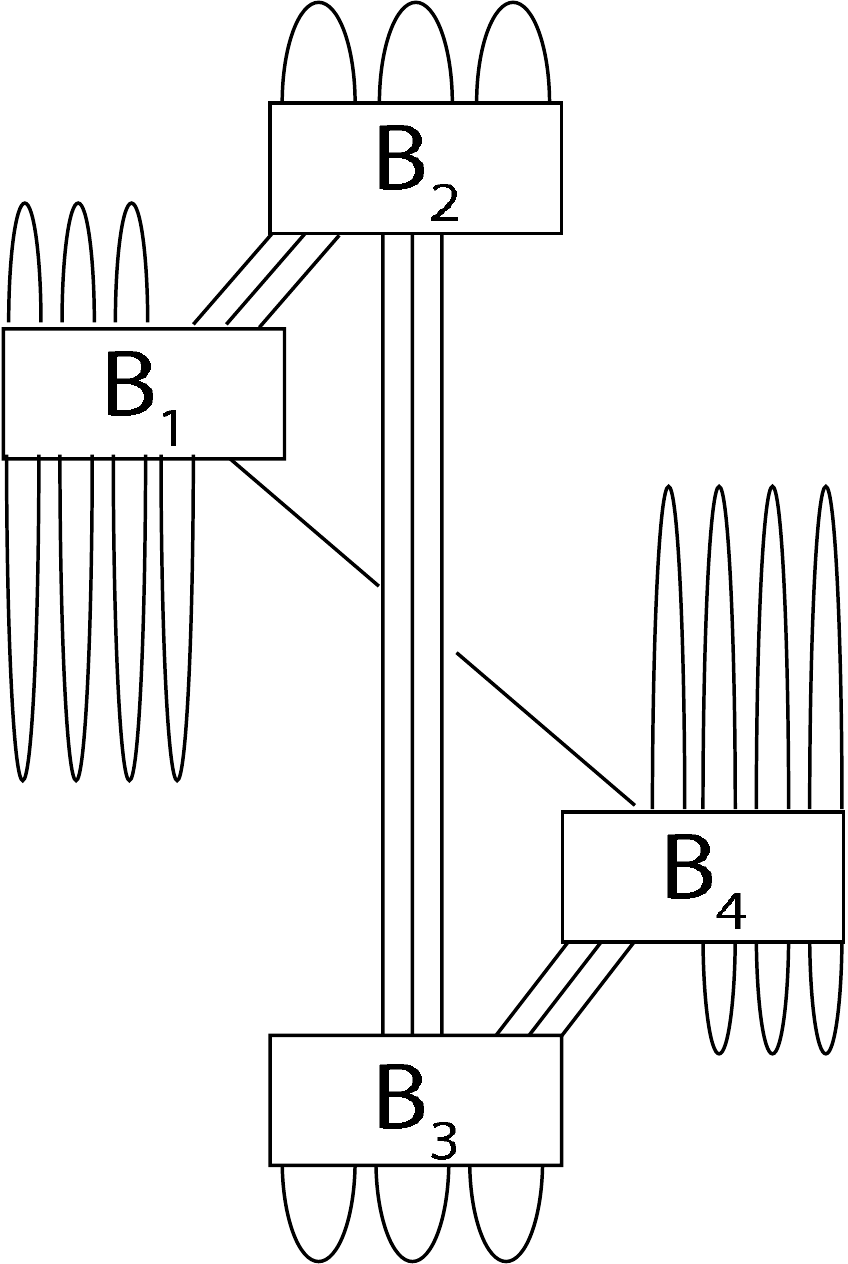}}
\caption{}\label{fig:lowbridge}
\end{figure}

{\em Proof of Theorem \ref{thm:counterwidth}:} Let $K'_{\alpha}$ be any two-bridge knot and let $K_{\alpha}$ be any of the knots constructed in Section \ref{sec:construction}. By Theorems \ref{thm:thinnestccompressible} and \ref{thm:incomp}, it follows that $w(K_\alpha)=134$. It is easy to see that the width of any two-bridge knot is eight. Figure \ref{fig:counter} demonstrates that $w(K_\alpha \# K'_{\alpha})\leq w(K_\alpha)=134$. By \cite{SchSch}, $w(K_\alpha \# K'_{\alpha})\geq w(K_\alpha)$ and, therefore, $w(K_\alpha \# K'_{\alpha})= w(K_\alpha)$.

\qed

{\em Proof of Theorem \ref{thm:counterbridge}:} By Theorems \ref{thm:thinnestccompressible} and \ref{thm:incomp}, it follows that $w(K_\alpha)=134$ and $K_\alpha$ has a unique thin position as in Figure \ref{fig:counterexample}. In this position, $K_{\alpha}$ has 11 maxima. However, Figure \ref{fig:lowbridge} demonstrates that $b(K'_{\alpha})\leq 10$.

\qed

 \end{document}